\UseRawInputEncoding
\documentclass[11pt,a4paper,reqno]{amsart}

\usepackage{tikz-cd}
\usepackage{esint}
\usepackage{amssymb}
\usepackage{amsmath}
\usepackage{MnSymbol}
\usepackage{graphicx}
\usepackage[toc,page]{appendix}
\usepackage{bm}
\usepackage{todonotes}

\usepackage[
hypertexnames=false, colorlinks, 
linkcolor=black,citecolor=red,urlcolor=black, linktocpage=true
]%
{hyperref} 
\hypersetup{bookmarksdepth=3}

\normalbaselines						  %

\newtheorem{theorem}{Theorem}[section]
\newtheorem{lemma}[theorem]{Lemma}
\newtheorem{corollary}[theorem]{Corollary}
\newtheorem{proposition}[theorem]{Proposition}

\theoremstyle{definition}

\newtheorem*{df*}{Definition}

\theoremstyle{remark}
\newtheorem{remark}[theorem]{Remark}
\newtheorem*{remark*}{Remark}

\theoremstyle{definition}
\newtheorem{Assumption}[theorem]{Assumption}

\numberwithin{equation}{section}

\newcommand{\p}{\partial}

\def\RR{{\mathbb R}}

\def\NN{\mathbb N}
\def\SS{\mathbb S}

\newcommand\abs[1]{\left|#1\right|}    
\newcommand{\norm}[1]{\left\lVert#1\right\rVert}

\def\cyr{\fontencoding{OT2}\fontfamily{wncyr}\selectfont}
\DeclareTextFontCommand{\textcyr}{\cyr}

{\end{list}}

\newcounter{vremennyj}

\begin{document}
\title[Boundary value problem for two-dimensional steady fluids]{On the Grad-Rubin boundary value problem for the two-dimensional magneto-hydrostatic equations}
\author{Diego Alonso-Or\'{a}n}
\address{Departamento de An\'{a}lisis Matem\'{a}tico, Universidad de La Laguna, C/Astrof\'{i}sico Francisco S\'{a}nchez s/n, 38271 – La Laguna, Spain}
\email{dlonsoo@ull.edu.es}
\author{Juan J. L. Vel\'azquez}
\address{Institut f\"{u}r Angewandte Mathematik\\  Universit\"{a}t Bonn\\ Endenicher Allee 60, 53115, Bonn \\ Germany}
\email{velazquez@iam.uni-bonn.de}
\thanks{D. Alonso-Or\'{a}n has been supported by the Alexander von Humboldt Foundation and by the Spanish MINECO through Juan de la Cierva fellowship FJC2020-046032-I. J. J. L.  Vel\'azquez acknowledges support through the CRC 1060 (The Mathematics of Emergent Effects) that is funded through the German Science Foundation (DFG), and the Deutsche Forschungsgemeinschaft (DFG, German Research Foundation) under Germany as Excellence Strategy EXC-2047/1-390685813}
\date{\today}

\begin{abstract}
	In this work, we study the solvability of a boundary value problem for the magneto-hydrostatic equations originally proposed by Grad and Rubin in \cite{Grad-Rubin-1958}. The proof relies on a fixed point argument which combines the so-called current transport method together with  H\"older estimates for a class of non-convolution singular integral operators. The same method allows to solve an analogous boundary value problem for the steady incompressible Euler equations.
\end{abstract}

\maketitle
\setcounter{tocdepth}{1}

\tableofcontents
\setcounter{tocdepth}{2}

\section{Introduction and prior results}\label{Chap:1}
In this paper we consider some boundary value problem for the two dimensional magneto-hydrostatic equation \textsc{(MHS)} given by 
\begin{equation}\label{MHS2D:eq}
	\left\lbrace
	\begin{array}{lll}
		j\times B&=\nabla p, \quad \mbox{in } \Omega \\
		\nabla\times B&=j, \quad \mbox{in } \Omega \\
		\nabla\cdot B&=0,  \quad \mbox{in }  \Omega
	\end{array} \right.
\end{equation}
where $B$ denotes the magnetic field, $j=\nabla \times B$ the current density and $p$ the fluid pressure on a suitable two dimensional manifold $\Omega$. The MHS equations are a particular case of the ideal steady magneto-hydrodynamics equations with trivial fluid flow $v=0$. Magneto-hydrostatics is relevant in a wide variety of problems in astrophysical plasmas describing coronal field structures and stellar winds as well as in the study of plasma confinement fusion, (cf. \cite{Goedbloed-Poedts-2010, Goedbloed-Poedts-2010-2, Priest-2014}). Using the vector identity $j\times B=(\nabla\times B)\times B=B\cdot\nabla B-\frac{1}{2}\nabla (\abs{B}^2)$ and defining the magnetic pressure or total pressure $p_{m}=p+\frac{1}{2}\abs{B}^2$, equations \eqref{MHS2D:eq} recast into
\begin{equation}\label{MHS2Dmp:eq}
	\left\lbrace
	\begin{array}{lll}
		B\cdot \nabla B&=\nabla p_{m}, \quad \mbox{in } \Omega \\
		\nabla\cdot B&=0,  \quad \mbox{in }  \Omega.
	\end{array} \right.
\end{equation}
Using the appropriate identification of variables, equations \eqref{MHS2Dmp:eq} are equivalent to the well-known equations of steady incompressible Euler equations, namely,
\begin{equation}\label{Euler2D:eq}
	\left\lbrace
	\begin{array}{lll}
		v\cdot \nabla v&=-\nabla p, \quad \mbox{in } \Omega \\
		\nabla\cdot v&=0,  \quad \mbox{in }  \Omega
	\end{array} \right.
\end{equation}
where $v:\Omega\to \mathbb{R}^2$ is the velocity fluid vector field and $p:\Omega\to \mathbb{R}$ denotes the fluid pressure. Indeed, a quick inspections shows that \eqref{MHS2Dmp:eq} is equivalent to \eqref{Euler2D:eq} using the transformations of variables $v\leftrightarrow B$ and $-p\leftrightarrow p_{m}$.

In this paper we are interested in studying some specific boundary value problems for \eqref{MHS2D:eq} where information about the magnetic field $B$ is given in different parts of the boundaries. Hereafter we will describe in detail the boundary value conditions into consideration for the case of the MHS equations \eqref{MHS2D:eq}.  Since from the mathematical point of view systems \eqref{MHS2Dmp:eq} and \eqref{Euler2D:eq} are identical,  a similar analysis and results can be shown for the steady Euler equations \eqref{Euler2D:eq}. Nevertheless a specific boundary value problem for one of the equations might not be physically relevant for the other and vice-versa.

Let $\Omega$ be a two dimensional orientable manifold with smooth boundary $\partial\Omega$. We will denote by $n$ the outer normal to the boundary and assume that the normal component of the magnetic field $B\cdot n$ is given. We now decompose the boundary $\partial\Omega=\partial\Omega_{+}\cup\partial\Omega_{-}$ where
$$\partial\Omega_{+}=\{ \bm{x}\in \partial \Omega :  (n\cdot B)(\bm{x})\geq0 \} \mbox{ and }  \partial\Omega_{-}=\{ \bm{x}\in \partial \Omega :  (n\cdot B)(\bm{x})\leq 0 \}.$$

The boundary problem which we treat in this paper consists in prescribing in addition to the normal component $B\cdot n$ on $\partial\Omega$, the tangential component $B\cdot \tau$ in one part of the boundary, namely on $\partial\Omega_{-}$. Here and in the following we denote by $\tau$ a unit vector tangent to the boundary. This boundary value problem was introduced in the seminal paper of Grad and Rubin \cite{Grad-Rubin-1958}. To the best of our knowledge, the well-posedness of this boundary value problem remains open even in the two-dimensional case. Furthermore, in  \cite{Grad-Rubin-1958} the authors also suggested different boundary value problems for the MHS equations in two dimensional and three dimensional cases. A relevant feature of the solutions constructed in this article is that the current $j$ is different from zero for generic choices of the boundary values. For the construction of zero current density solutions, i.e. $j=0$, it is well-known that system \eqref{MHS2D:eq} reduces to the study of the Laplace equation where the theory of harmonic functions can be applied to study the existence of solutions. 

In this work, we will restrict ourselves to a very particular geometric setting, namely we will assume that 
\begin{equation}\label{domain}
	\Omega=\mathbb{S}^{1}\times [0,L],
\end{equation}
with $L>0.$ The reason to choose this manifold is the following: for $\Omega$ as in \eqref{domain} we can choose the values of $B\cdot n$ in such a way that $\partial\Omega_{+}\cap \partial\Omega_{-}=\emptyset$ and in particular we can guarantee that $B\cdot n\neq 0$ at all points $\bm{x}\in \partial\Omega$.  As it has been discussed in \cite{Alo-Velaz-2021} at the points of the set $\partial\Omega_{+}\cap \partial\Omega_{-}$ some singular behaviour for $B$ arise for generic domains $\Omega$. In order to avoid the technical difficulties that should be considered in that situation, we will just work on the particular manifold \eqref{domain}.

It is worth to notice that several boundary value problems for the steady Euler or MHS equations have been considered in the literature \cite{Alber-1992, Alo-Velaz-2021,Buffoni-Wahlen-2019,
	Molinet-1999,Seth-2020,Tang-Xin-2009}. We refer the interested reader to \cite{Alo-Velaz-2021} for a thorough description of the currently available results considering the well-posedness of the different boundary value problems for the steady Euler or MHS equations.

In order to solve boundary value problems for both equations, two main methods have been considered in the literature: the Grad-Shafranov method \cite{Grad-1967,Safranov-1966} and the vorticity transport method introduced by Alber \cite{Alber-1992}. The former is restricted to two dimensional settings or to problems with particular symmetries, for instance axisymmetric or toroidal symmetries. The main idea behind the Grad-Shafranov method relies on reducing the steady Euler or the MHS equations to an elliptic equation where large number of techniques are available. See for instance, \cite{CDG-2020,CDG-2020-2,CLV,Hamel-Nadirashvili-2017} for ideas closely related to the Grad-Shafranov approach that have been recently applied to derive properties solutions of the steady Euler equation and MHS equation.

A different approach to obtain solutions with non-vanishing vorticity (since it was originally applied for the steady Euler equation) was introduced by Alber \cite{Alber-1992}. Roughly speaking, he constructed solutions where the velocity field $v$ can be split into $v=v_0+V$ where $v_0$ is an irrotational solution to \eqref{Euler2D:eq} and $V$ a small perturbation. The boundary value problem for the Euler equations is reduced to a fixed point problem for a function $V$ combining the fact that the vorticity satisfies a suitable transport equation and that the velocity can be recovered from the vorticity using the Biot-Savart law. This idea will be discussed later in more detail. In particular, we will explain why Alber's method cannot be directly applied to solve the boundary value problem we are interested in and more importantly what are the new key tools we implement to address the problem.

\subsection{Notation}\label{Sec:11}
We will use the following notation throughout the manuscript. 
\begin{itemize}
	\item We recall that we are working on a manifold with boundary $\Omega= \mathbb{S}^1\times [0,L]$ with $L>0$. The boundary of the manifold  $\Omega$, will be denoted by $\partial\Omega=\partial\Omega_{+}\cup\partial\Omega_{-}$ where $\partial\Omega_{+}=\mathbb{S}^{1}\times\{L\}$ and $\partial\Omega_{-}=\mathbb{S}^{1}\times\{0\}$.  We will use several operators that will be defined in $\partial\Omega_{-}$. In those cases it will be convenient to identify $\partial\Omega_{-}$ with $\SS^1$ and then to consider that the operators are acting of spaces of functions with domain $\SS^1$ instead of $\partial\Omega_{-}$. Notice that these spaces of functions are isomorphic.

	\item  Let us denote by $n$ the outer normal to 
	$\partial \Omega$ in the points of $\partial\Omega_{+}$, the inner normal to $\partial \Omega$ in the points of $\partial\Omega_{-}$ and by $\tau$ the tangential vector.   
	
	\item In order to simplify the exposition, we will also use the bold notation $\bm{x}\in \Omega$ to denote a pair $\bm{x}=(x,y)\in\Omega$. 
	\item Let $C_{b}(\Omega)$ be the set of bounded continuous functions on $\Omega.$ For any bounded continuous function and $0<\alpha<1$ we call $f$ uniformly H\"older continuous with exponent $\alpha$ in $\Omega$ if the quantity
	$$ \left[ f \right]_{\alpha,\Omega}:= \displaystyle\sup_{\bm{x}\neq \bm{y}; \bm{x},\bm{y}\in \overline{\Omega}} \frac{\abs{f(\bm{x})-f(\bm{y})}}{\abs{\bm{x}-\bm{y}}^\alpha}$$
	is finite. However, this is just a semi-norm and hence in order to work with Banach spaces we define the space of H\"older continuous functions as 
	$$ C^{\alpha}(\Omega)=\{ f\in C_{b}(\Omega): \norm{f}_{C^\alpha(\Omega)}< \infty\},$$
	equipped with the norm
	$$ \norm{f}_{C^{\alpha}(\Omega)}:=\displaystyle \sup_{\bm{x}\in \overline{\Omega}}\abs{f(\bm{x})}+  \left[ f \right]_{\alpha,\Omega}.$$
	Similarly, for any non-negative integer $k$ we define the H\"older spaces $C^{k,\alpha}(\Omega)$ as 
	$$ C^{\alpha}(\Omega)=\{ f\in C^{k}_{b}(\Omega): \norm{f}_{C^{k,\alpha}(\Omega)}< \infty\},$$
	equipped with the norm
	$$ \norm{f}_{C^{k,\alpha}(\Omega)}=\displaystyle\max_{\abs{\beta}\leq k }\sup_{\bm{x}\in \overline{\Omega}}\abs{\partial^{\beta}f(\bm{x})}+\displaystyle\sum_{\abs{\beta}=k}\left[\partial^{\beta}f \right]_{\alpha,\Omega}.$$
	where $\beta=(\beta_1,\beta_2)\in \NN^{2}_{0}$ and $\NN_{0}=\{0,1,2,\ldots\}$.
	Notice that in the definitions above the H\"older regularity holds up to the boundary, i.e in $\overline{\Omega}$.
	We omit in the functional spaces whether we are working with scalars or vectors fields, this is $C^{k,\alpha}(\Omega,\mathbb{R})$ or $C^{k,\alpha}(\Omega,\mathbb{R}^2)$ and instead just write $C^{k,\alpha}(\Omega)$. The specific type of functional space (scalar or vector) will be clear from the context. Moreover, we will denote H\"older spaces on the boundary of the manifold, namely on $\partial\Omega$, $\partial\Omega_{+}$ and $\partial\Omega_{+}$ by $C^{k,\alpha}(\partial\Omega)$, $C^{k,\alpha}(\partial\Omega_{+})$ and $C^{k,\alpha}(\partial\Omega_{-})$ respectively. 
	
	\item Let $M>0$ and let $X$ be Banach space. Then we define by $B_{M}(X)$ the closed ball in $X(\Omega)$ with radius $M$, i.e.
	$$ B_{M}(X)=\{f\in X: \norm{f}_{X}\leq M \}.$$

	\item We identify the functions  $f\in C^{k,\alpha}(\mathbb{S}^1)$,$k=0,1,2...$, $\alpha\in (0,1)$ with the subspace of $C^{k,\alpha}(\mathbb{R})$ such that $f(x+2\pi)=f(x)$. Moreover, we will also identify $\SS^1$ with any interval $[a,b]$ where $b-a=2\pi$.
	
	\item For a sufficiently smooth $2\pi$-periodic function in the $x$ variable $f$, we define the Fourier coefficients of $f$ in the first variable by $$\widehat{f}(n,y)=\frac{1}{2\pi}\int_{0}^{2\pi}f(x,y)e^{-inx} \ dx.$$ Then we have the Fourier series representation, 
	$f(x,y)=\displaystyle\sum_{n=-\infty}^{n=\infty}\widehat{f}(n,y)e^{inx}.$
	
	\item Throughout the manuscript we will denote with $C$ a positive generic constant that depends only on fixed parameters. More precisely, they will depend on the the parameter $L$ and the H\"older exponent $\alpha$. Note also that this constant might differ from line to line.
	
	\item We will also use the brackets $\big[ \cdot \big]$, in order to denote the dependence of an operator on the bracketed function, namely $T\big[ f\big]$ denotes that the operator $T$ depends in a certain way on the function $f$.
	\item Let $E$ and $F$ be  Banach spaces. We say that $T$ is a bounded operator from $E$ to $F$ if there exists a constant $c\geq 0$ such that $\norm{Tu}_{F}\leq c\norm{u}_{E}$, $\forall u\in E.$ The norm of the bounded operator $T$ is defined and denoted as
	$$\norm{T}_{\mathcal{L}(E,F)}= \displaystyle\sup_{u\neq 0}\frac{\norm{Tu}_{F}}{\norm{u}_{E}}. $$
	Moreover, if $E=F$, we just write $\mathcal{L}(E)$ instead of $\mathcal{L}(E,E)$.
	
\end{itemize}

\subsection{Main result}\label{Sec:12}
The main result in this article deals with the well-posedness of a boundary value problem for the MHS equations suggested by Grad-Rubin in \cite{Grad-Rubin-1958}. Specifically we prescribe the normal component $B\cdot n$ on $\partial\Omega$ and  the tangential component $B\cdot \tau$ in one part of the boundary, namely on $\partial\Omega_{-}$. In particular, our result reads as follows
\begin{theorem}\label{theorem2D}
	Let $\Omega=\{ (x,y)\in \mathbb{S}^{1}\times [0,L]\}$, with $L>0$ and $\alpha\in (0,1)$. There exists $M=M(\alpha,L)>0$ sufficiently small such that for 
	$f\in C^{{2,\alpha}}(\partial \Omega)$ and $g\in C^{{2,\alpha}}(\partial \Omega_{-})$ satisfying 
	\begin{equation}\label{smallness}
		\norm{g}_{C^{{2,\alpha}}(\partial \Omega_{-})}+\norm{f}_{C^{{2,\alpha}}(\partial \Omega)} \leq M ,
	\end{equation} 
	and 
	\begin{equation}\label{compcond}
		\int_{\partial\Omega_{-}}f \ dx=  \int_{\partial\Omega_{+}}f \ dx,
	\end{equation}
	there exists a unique $(B,p)\in C^{{2,\alpha}}(\Omega)\times C^{{2,\alpha}}(\Omega)$ with $B=(B_1,B_2)$ to \eqref{MHS2D:eq}  with $$\norm{B-(0,1)}_{C^{2,\alpha}(\Omega)}\leq M$$ such that
	\begin{equation}\label{bvc}
		B\cdot n =1+f \ \mbox{on} \ \partial \Omega \mbox{ and } \ B\cdot \tau=g  \ \mbox{on} \ \partial \Omega_{-}.
	\end{equation}
\end{theorem} 
\begin{remark}
	Notice that the solutions $(B,p)$ are obtained as small perturbations around the particular vertical constant magnetic field $B_0=(0,1)$. The constant magnetic fields of the form $B_0=(0,a)$ for $a>0$ can be reduced by a re-scaling argument to the unitary magnetic field $B_0=(0,1)$. On the other hand, it is not a priori clear if it is possible to perturb around more general non-constant magnetic fields.
\end{remark}
\begin{remark}
	A question that could be interesting to explore is whether one can generalize Theorem \ref{theorem2D} to more general domains $\Omega= \{(x_1,x_2) : \gamma_1(x_1)< x_2< \gamma_2(x_1) \}$
	where $\gamma_j$ are smooth functions satisfying the periodicity condition $\gamma_{j}(x_1+2\pi)=\gamma_{j}(x_1)$ for $j=1,2$. In the proof of Theorem \ref{theorem2D}, several computations which can be made in a explicit manner in the case of the domain $\Omega=\SS^1\times [0,L]$, will become more involved for more general domains.
\end{remark}
\begin{remark}
In the three dimensional setting $\Omega=\SS^1\times \SS^1\times [0,L]$, we believe that the same ideas developed in this paper can be carried out, although the computations are more involved. In particular, we will need to derive H\"older estimates for non-convolution singular integral operators that in the three dimensional case are more delicate. 
\end{remark}
Notice that using the change of variables $B\leftrightarrow v$ and $p_{m}\leftrightarrow -p$ the following result can be obtained for the steady Euler equations
\begin{theorem}\label{theorem2D:Euler}
	Let $\Omega=\{ (x,y)\in \mathbb{S}^{1}\times [0,L]\}$, with $L>0$ and $\alpha\in (0,1)$. There exists $M=M(\alpha,L)>0$ sufficiently small such that for 
	$f\in C^{{2,\alpha}}(\partial \Omega)$ and $g\in C^{{2,\alpha}}(\partial \Omega_{-})$ satisfying 
	\begin{equation*}
		\norm{g}_{C^{{2,\alpha}}(\partial \Omega_{-})}+\norm{f}_{C^{{2,\alpha}}(\partial \Omega)}\leq M ,
	\end{equation*} 
	and 
	\begin{equation*}
		\int_{\partial\Omega_{-}}f \ dx=  \int_{\partial\Omega_{+}}f \ dx,
	\end{equation*}
	there exists a unique $(v,p)\in C^{{2,\alpha}}(\Omega)\times C^{{2,\alpha}}(\Omega)$ with $v=(v_1,v_2)$ to \eqref{Euler2D:eq}  with $$\norm{v-(0,1)}_{C^{2,\alpha}(\Omega)}\leq M$$ such that
	\begin{equation*}
		v\cdot n =1+f \ \mbox{on} \ \partial \Omega \mbox{ and }  \ v\cdot \tau=g  \ \mbox{on} \ \partial \Omega_{-}.
	\end{equation*}
\end{theorem}

\subsection{Strategy behind the proof and novelties}
\label{sec:13}
The strategy of the proof is based on two ingredients, namely the transport equation for the current and the div-curl problem that recovers the magnetic field in terms of the current. Suppose that we have a magnetic field with the form $(B_1,B_2)=(0,1)+b$ where $b$ is a small perturbation of the vertical base magnetic field. 

For magnetic fields for which the magnetic vector is contained always in a given plane, the current $j$ is a vector in the direction of the normal to the plane. However, in these two dimensional settings it is more convenient to assume that the current is a scalar quantity and therefore we will use the notation  $j=\nabla\times B=-\p_{y}B_{1}+\p_{x}B_2$. 

It is well-known that if $B$ solves \eqref{MHS2D:eq}, the current density $j$ solves the following transport equation 
\begin{equation}\label{transport:problem:just}
	B\cdot\nabla j =((0,1)+b)\cdot\nabla j=0, \ \mbox{in} \ \Omega.
\end{equation} 
On the other hand, assuming that we have a current $j$ we can recover the corresponding magnetic field $B$ solving the following system of equations
\begin{equation}\label{div:curl:just}
	\left\lbrace
	\begin{array}{lll}
		\nabla\times B=\nabla \times b= j, \ \mbox{in} \ \Omega \\
		\mbox{div } B= \mbox{div } b=0, \ \mbox{in} \ \Omega \
	\end{array} \right.
\end{equation}
The equations \eqref{transport:problem:just}-\eqref{div:curl:just} must be solved under suitable boundary value conditions. It turns out that given the function $j$ we can obtain a unique solution $B$ to \eqref{div:curl:just} if we prescribe the normal component of the magnetic field on the two connected components of the boundary of $\Omega$
\begin{align}\label{normal:component:basi}
	B\cdot n=f, \ \mbox{on} \ \partial \Omega 
\end{align}
as well as the horizontal flux for the magnetic field
\begin{align}\label{lateral:flow:basic}
	\displaystyle \int_{0}^{L} B_1(0,y) \ dy=J.
\end{align}
We will see later, that the value of $J$ has to be chosen in a very specific way to obtain a uni-valued pressure $p$ on $\Omega$.

On the other hand, if we assume that $b$ is sufficiently small (in a sense to be precise later), the current $j$ is uniquely determined in $\Omega$ if we prescribe it in any of the two connected components of $\partial\Omega$. For instance, if 
\begin{equation}\label{equation:omega:basic}
	j(x,0)=j_0(x), \ \mbox{on } \partial\Omega_{-}
\end{equation}
is given, we can obtain $j$ in $\Omega$ just by using the method of characteristics. Notice however that the boundary conditions for the problem \eqref{MHS2D:eq}-\eqref{bvc} do not allow to compute the value of $j_0$ in \eqref{equation:omega:basic}. 

On the other hand, we have an additional boundary condition that yields the tangential component of the magnetic field
\begin{equation}\label{tangential:component:basi}
	B\cdot \tau=g \mbox{ on } \ \partial \Omega_{-}. 
\end{equation}

The structure of the problem suggests to use a fixed point argument in order to construct the solution. More precisely, given a vector field $B$ defined in $\Omega$ as a well as a function $j_0$ on $\partial\Omega_{-}$ we can solve \eqref{transport:problem:just} with the boundary condition \eqref{equation:omega:basic} to construct a current field $j[B;j_0](\cdot)$ defined in  $\Omega$ . Using this current function we can solve \eqref{div:curl:just} with boundary conditions \eqref{normal:component:basi} and \eqref{lateral:flow:basic} to find a new vector field $\widetilde{B}[B;j_{0}](\cdot)$ in $\Omega$. Notice that the new vector field $\widetilde{B}$ does not satisfy in general the boundary condition \eqref{tangential:component:basi}. However, this equation can be reformulated as 
\begin{equation}
	\widetilde{B}[B;j_{0}]\cdot \tau= g \ \mbox{ on } \partial\Omega_{-},
\end{equation}
that turns out to be an integral equation for the function $j_0$ on $\partial\Omega_{-}$. We can prove that this integral equation can be solved by means a fixed point argument using regularity estimates for non-convolution singular integral operators in H\"older spaces. The solution of this equation yields an operator $B\to j_{0}[B]$. Notice that this operator depends also on the boundary value conditions $f,g$, but we will not write this dependence explicitly. We can now define an operator $B\to\Gamma[B]=\widetilde{B}[B,j_0[B]]$. A fixed point argument for the operator $\Gamma(\cdot)$ solves the problem \eqref{transport:problem:just}, \eqref{div:curl:just}, \eqref{normal:component:basi}, \eqref{lateral:flow:basic} and \eqref{tangential:component:basi}. Using now the fact that $\nabla\times (j\times B)=B\cdot\nabla j=0$ one can show, arguing as in \cite{Alo-Velaz-2021}, that there exists a pressure function $p$ such that $(B,p)$ satisfies \eqref{MHS2D:eq} and \eqref{bvc}.

It is worth to notice that there are several important differences regarding the problem treated here and previous works \cite{Alber-1992, Alo-Velaz-2021,Buffoni-Wahlen-2019,
	Molinet-1999,Seth-2020,Tang-Xin-2009}. For a more detailed description of the different boundary value problems mentioned previously, we refer the reader to \cite{Alo-Velaz-2021}. In the case treated by Alber \cite{Alber-1992} for the steady incompressible Euler equation, the vorticity $\omega_0$ (or current $j_0$ in our case) on $\partial\Omega_{-}$ can be readily obtained from the boundary values given in the problem, so roughly speaking $\omega_0$ (or $j_0$) on $\partial\Omega_{-}$ is already prescribed. On the other hand, this is not the case for the boundary value type problems solved in \cite{Alo-Velaz-2021} where the vorticity  $\omega_0$ (or current $j_0$ in our case) is not fully prescribed by the boundary values. Instead in those cases,  $\omega_0$  (or $j_0$) is part of the solution. Nevertheless, it can be obtained by means of the fixed point argument. More precisely, the value of $\omega_0$  (or $j_0$) can be computed using the Euler equation \eqref{Euler2D:eq} and is given in terms of $v$  (or $B$), its derivative and the boundary value conditions. Using the characteristics one can solve the transport equation \eqref{transport:problem:just} to construct $\omega[v;\omega_0](\cdot)$ (or $j[B;j_0](\cdot)$) and then equation \eqref{div:curl:just} to construct the new velocity field $\widetilde{v}[v;\omega_{0}](\cdot)$ or magnetic field $\widetilde{B}[B;j_{0}](\cdot)$. The crucial point is that the new velocity field or magnetic field already satisfies the required boundary value conditions, since $\omega_0$ or $j_0$ has been chosen in terms of the boundary conditions and $v$ or $B$ in a precise way.

To deal with the boundary value conditions imposed in \eqref{bvc}, we have to use a more sophisticated argument to compute the value of $j_0$ on $\partial\Omega_{-}$. As we have explained above, this reduces to study an integral equation containing singular integral operators. To show the existence and uniqueness of the integral equation, we derive some general results providing H{\"o}lder estimates for a class of non-convolution singular integral operators which are of independent interest (cf. Section \ref{S4}). The use of H\"older spaces instead of Sobolev spaces (as in \cite{Alber-1992} for instance) is an important detail. Indeed, the value $j_0$ at the boundary $\partial\Omega_{-}$ depends on the value of $B$ (and the boundary data) and therefore, if the estimates for $B$ are given in terms of Sobolev spaces, we obtain less regularity for $j_0$ due to the classical regularity trace theorem. Once $j_0$ is obtained we can compute $j$ along $\Omega$ using the transport equation \eqref{transport:problem:just} which does not improve the regularity due its hyperbolic character. Therefore, the new function $B$ computed via the div-curl problem \eqref{div:curl:just} has a loss of regularity which prevents to close a fixed point argument. This obstructions can be avoided by making use of H\"older spaces.

\subsection{Plan of the paper}\label{Sec:14}
In Section \ref{Sec:2} we illustrate the main formal idea used to construct the solution of \eqref{MHS2D:eq}, \eqref{bvc} by means of the study of a suitable linearized problem which can be explicitly solved by using Fourier series. In Section \ref{Sec:3} it is seen how to reformulate the full non-linear boundary value problem \eqref{MHS2D:eq}, \eqref{bvc} as a fix point problem for a suitable operator. The precise definitions of the operators needed to reformulate the problem are postponed until Section \ref{sec:6} since the proof that the operators are well-defined required several estimates showed in Sections \ref{S4} - Section \ref{S5}. In Section \ref{Sec:3} (more precisely in Subsection \ref{S:33}) we derive an integral equation for the current $j_0$ which is a consequence of the equations \eqref{MHS2D:eq} and the boundary values \eqref{bvc}. This integral equation plays a crucial role in the proof of the result proved in this paper. In Section \ref{S4} we derive some general lemmas showing $C^{1,\alpha}$ and $C^{\alpha}$  H\"older estimates for non-convolution singular integral operators. These operators are a suitable class of perturbations of convolution operators. In Section \ref{sec:4:2} we will provide the $C^\alpha$ and $C^{1,\alpha}$ H\"older estimates for the operators contained in the integral integral equation for $j_0$. In Section \ref{S5} we show the existence and uniqueness of solutions to the integral equation for $j_0$ by using the previous derived estimates. In Section \ref{sec:6}, as indicated above, we provide the precise definitions of the operators required to reformulate the original boundary value problem \eqref{MHS2D:eq},\eqref{bvc} as a fixed point problem for a suitable operator. Moreover, we also show that the operator has a fixed point on a suitable functional space. To conclude the article, in Section \ref{Sec:7} we prove Theorem \ref{theorem2D} as a direct application of the fixed point theorem showed in the previous section. 

\section{The linearized problem}\label{Sec:2}
In this section, we will describe the formal idea behind the method to construct solutions $(B,p)$ to \eqref{MHS2D:eq} satisfying the boundary value conditions \eqref{bvc}. As we have mentioned in the introduction, the proof is based on defining an adequate operator $\Gamma$ on a subspace of $C^{2,\alpha}(\Omega)$ which has a fixed point $b$ such that $B=(0,1)+b$ is a solution to \eqref{MHS2D:eq} and \eqref{bvc}.  We define the operator $\Gamma: B_{M}(C^{2,\alpha}(\Omega)) \to C^{2,\alpha}(\Omega)$ in two steps. First, given $b\in B_{M}(C^{2,\alpha}(\Omega))$ we define $j\in C^{1,\alpha}(\Omega)$ solving the following the transport type problem 
\begin{equation}\label{transport:problem}
	\left\lbrace
	\begin{array}{lll}
		((0,1)+b)\cdot\nabla j =0, \ \mbox{in} \ \Omega \\
		j = j_{0}, \ \mbox{on} \ \partial \Omega_{-}
	\end{array} \right.
\end{equation} 
where $j_{0}$ is a priori an unknown quantity. As a second step, we define $W\in C^{2,\alpha}(\Omega)$ as the unique solution to the div-curl problem
\begin{equation}\label{div:curl:problem}
	\left\lbrace
	\begin{array}{lll}
		\nabla\times W= j, \ \mbox{in} \ \Omega \\
		\mbox{div } W=0, \ \mbox{in} \ \Omega \\ 
		W\cdot n=f, \ \mbox{on} \ \partial \Omega \\
		W\cdot \tau=g, \ \mbox{on} \ \partial \Omega_{-} \\
		\int_{0}^{L} W_{1}(0,y) \ dy =J.
	\end{array} \right.
\end{equation} 
Thus we define $\Gamma(b,J)=W$. We remark that $J$
is a degree of freedom of the problem, since there exists non trivial solutions $(W,J)$ of the homogeneous problem \eqref{div:curl:problem} with $f=g=0$ given by 
\begin{equation}
W=(\frac{2Jy}{L^2},0) \mbox{ and } j=-\frac{2J}{L^2}.
\end{equation}
This degree of freedom will be used later to obtain a uni-valued function pressure $p$ in $\Omega$.

We are interested in obtaining solutions of the form $B=(0,1)+b$ where $b=(b_1,b_2)$ is a small perturbation, i.e. $\norm{b}_{C^{2,\alpha}(\Omega)} \leq M$ with $M\leq M_0$ and $M_0$ sufficiently small. Therefore, in the lowest order (dropping the \textit{small} nonlinear terms of order $M^2$), the transport equation
\eqref{transport:problem} reduces to
\begin{equation}\label{condition:j_0:x1} \partial_{y}j(x,y)=0, \mbox{ in } \Omega 
\end{equation}
and hence $j(x,y)= j_0(x)$. Then, with this approximation, the div-curl problem \eqref{div:curl:problem} becomes
\begin{equation}\label{DCP:problem:linear}
	\left\lbrace
	\begin{array}{lll}
		\nabla \times W =j_0(x), \ \mbox{in} \ \Omega \\ 
		\mbox{div } W =0, \ \mbox{in} \ \Omega \\ 
		W\cdot n= f, \ \mbox{on} \ \partial\Omega \\
		W\cdot \tau= g, \ \mbox{on} \ \partial\Omega_{-} \\
		\int_{0}^{L} W_{1}(0,y) \ dy =J.
	\end{array}\right.
\end{equation}
Notice that \eqref{DCP:problem:linear} is a non-homogeneous linear problem for $W$.
To solve \eqref{DCP:problem:linear}, we examine the following auxiliary problem (cf.  \cite[\S 3.1.1]{Alo-Velaz-2021}), namely
\begin{equation}\label{stream:formulation:linear}
	\left\lbrace
	\begin{array}{lll}
		\Delta \psi= j_{0}(x), \ \mbox{in} \ \Omega \\ 
		\psi(x,L)= -J+h^{+}(x),  \  x\in \mathbb{R} \\
		\psi(x,0)=h^{-}(x), \  x\in \mathbb{R}\\
		\partial_{y} \psi (x,0)= -g, \  x\in \mathbb{R}
	\end{array}\right.
\end{equation}
where 
\begin{equation}\label{def:h}
	h^{+}(x)=\displaystyle\int_{0}^{x} (f(\xi,L)-A) \ d\xi, \quad h^{-}(x)= \displaystyle\int_{0}^{x} (f(\xi,0)-A) \ d\xi 
\end{equation}
and
\begin{equation}\label{def:A}
	A=\int_{\partial\Omega_{+}} f \ dS=\int_{\partial\Omega_{-}} f  \ dS.
\end{equation}
For a sufficiently smooth stream function $\psi$, the function $W=(0,A)+\nabla^{\perp}\psi$, where $\nabla^{\perp}\psi=(-\frac{\partial \psi}{\partial y},\frac{\partial \psi}{\partial x})$, solves  \eqref{DCP:problem:linear}.  However, for any fixed $j_0(x)$ the problem \eqref{stream:formulation:linear} is over-determined. This fact will be used in order to obtain the a priori unknown function $j_0(x)$.

In order to obtain a complete linearized version of the problem \eqref{MHS2D:eq} satisfying boundary conditions \eqref{bvc}, it remains to add a condition that guarantees that the pressure is a uni-valued function on $\Omega$. Indeed, a linearized version of \eqref{MHS2D:eq}
with $B=(0,1)+b$ (with $b$ small) is given by
\begin{align}\label{condition:pressure:1}
-j(x,y)=\partial_{x} p, \quad 0=\partial_{y} p.
\end{align}
A necessary condition for the solvability of this problem is that $\partial_{y}j(x,y)=0$ and hence $j(x,y)=j_0(x)$, similar as the condition derived in \eqref{condition:j_0:x1}. Therefore, \eqref{condition:pressure:1} reduces to
\begin{align}\label{final:pressure:1}
 -j_0(x)=\partial_{x} p,     \quad 0=\partial_{y} p.
\end{align}
Then, we can obtain a solution to \eqref{final:pressure:1} given by
\begin{equation}
    p(x,y)=\int_{\bm{0}}^{\bm{x}} j_{0}(x) \ dx
\end{equation}
	where the integral on the right hand side is the line integration computed along any contour connecting $\bm{0}=(0,0)$ and $\bm{x}\in\Omega$. Notice that a necessary and sufficient condition to ensure that $p(x,y)$ is a uni-valued function in $\Omega$ is that
	\begin{equation}\label{nec:uni:valued:p:linear}
	    \int_{0}^{2\pi} j_{0}(x) \ dx=0.
	\end{equation}

To this end, we apply the Fourier transform in the $x$ variable to equation \eqref{stream:formulation:linear}. This transforms the PDE \eqref{stream:formulation:linear} into the following second-order non-homogeneous ODEs  with constant coefficients 
\begin{equation}\label{stream:function:lower:2}
	\left\lbrace
	\begin{array}{lll}
		-n^{2}\widehat{\psi}(n,y)+\partial_{yy}\widehat{\psi}(n,y)= \widehat{j_{0}}(n), \quad (n,y)\in \mathbb{Z}\times (0,L) \\ 
		\widehat{\psi}(n,L)= -J\delta_{0,n}+\widehat{h^{+}}(n), \quad n\in \mathbb{Z} \\
		\widehat{\psi}(n,0)= \widehat{h^{-}}(n), \quad n\in \mathbb{Z} \\
		\partial_{y} \widehat{\psi}(n,0)=-\widehat{g}(n), \quad n\in \mathbb{Z}.
	\end{array}\right.
\end{equation}
Above, $\widehat{h^{+}}(n), \widehat{h^{-}}(n)$ are the Fourier coefficients associated with the function $h^{+}, h^{-}$ respectively, and $\widehat{g}(n)$ the Fourier coefficients of the function $g$.

After a straightforward calculation using variation of parameters method we find that 
\begin{align}\label{psi:fourier:exp}
	\widehat{\psi}(n,y)&=-J\frac{y}{L}\delta_{0,n}+\widehat{h^{+}}(n)\frac{\sinh(\abs{n}y)}{\sinh(\abs{n}L)}+\widehat{h^{-}}(n) \frac{\sinh(\abs{n}(L-y))}{\sinh(\abs{n}L)} \nonumber \\ 
	&\quad \quad  +\widehat{j_{0}}(n)\frac{\sinh(\abs{n}(L-y))-\sinh(\abs{n}L)+\sinh(\abs{n}y)}{\abs{n}^{2}\sinh(\abs{n}L)},
\end{align}
for $n\in \NN$. In the case $n=0$, the functions multiplying $\widehat{h^{+}}(n), \widehat{h^{-}}(n)$ and $\widehat{j_0}(n)$ must be understood as the limit when $n$ tends to zero. More precisely, for $n=0$, we use the replacements 
\begin{equation*}
	\frac{\sinh(\abs{n}y)}{\sinh(\abs{n}L)}\longmapsto \frac{y}{L}, \quad \frac{\sinh(\abs{n}(L-y))}{\sinh(\abs{n}L)}\longmapsto \frac{L-y}{2},
\end{equation*}
and
\begin{equation*}
	\frac{\sinh(\abs{n}(L-y))-\sinh(\abs{n}L)+\sinh(\abs{n}y)}{\abs{n}^{2}\sinh(\abs{n}L)}\longmapsto\frac{L^2y}{2}\left( \frac{y}{L}-1\right).
\end{equation*}
This convention of understanding several combinations of trigonometric hyperbolic functions when $n=0$ as the limit when $n$ tends to zero will be used throughout the paper. Imposing the last boundary condition $\partial_{y} \widehat{\psi}(n,0)=-\widehat{g}(n)$ in \eqref{stream:function:lower:2}, we find that
\begin{align*}
	\partial_{y} \widehat{\psi}(n,0)&=-\frac{J}{L}\delta_{0,n}+\widehat{h^{+}}(n)\frac{\sinh(\abs{n}y)}{\sinh(\abs{n}L)}+\widehat{h^{-}}(n) \frac{\sinh(\abs{n}(L-y))}{\sinh(\abs{n}L)}+\widehat{j_{0}}(n)\frac{1-\cosh(\abs{n} L)}{\abs{n} \sinh(\abs{n}L)}=-\widehat{g}(n).
\end{align*}

Taking the inverse Fourier transform in the first variable we obtain 
\begin{align}\label{psi:eq}
	\partial_{y} \psi(x,0)=-\frac{J}{L}+\mathcal{Z}(x)
	+\frac{1}{2\pi} \int_{\mathbb{S}^{1}}\displaystyle\sum_{n=-\infty}^{n=\infty} \frac{1-\cosh(\abs{n} L)}{\abs{n} \sinh(\abs{n}L)}e^{in(x-\eta)}j_{0}(\eta) \ d \eta=-g(x)
\end{align}
with 
\begin{equation}\label{B}
	\mathcal{Z}(x)=\frac{1}{2\pi} \displaystyle\sum_{n=-\infty}^{n=\infty}\left(\widehat{h^{+}}(n) \frac{\abs{n}}{\sinh (\abs{n}L)}-
	\widehat{h^{-}}(n)\frac{\abs{n}}{\tanh (\abs{n}L)}\right) e^{inx}.
\end{equation}
Using the symmetry in $n$ and denoting the kernel 
\begin{equation}\label{operator:convo:L}
	\mathcal{G}^{L}(x)=\displaystyle\sum_{n=-\infty}^{n=\infty} \frac{\cosh(n L)-1}{n \sinh(nL)}e^{inx}
\end{equation}
and 
\begin{equation}\label{modified:g}
	\widetilde{g}(x)=-g(x)-\mathcal{Z}(x),
\end{equation}
we have that \eqref{psi:eq} can be expressed as the following convolution equation for $j_0$,
\begin{equation}\label{operator:omega:linear}
	\mathcal{T}^{L}j_{0}(x)=-\frac{1}{2\pi}\int_{\mathbb{S}^1}\mathcal{G}^{L}(x-\eta)j_0(\eta) d\eta= \widetilde{g}(x)+\frac{J}{ L}.
\end{equation}
Notice that the function $\widetilde{g}$ depends only on the boundary values $g$ and $f$. Using the fact that the Fourier coefficients in \eqref{operator:convo:L} are different than zero, we can use standard Fourier techniques to invert the operator yielding
\begin{equation}\label{inversion:operator:linearized}
	j_0(x)=(\mathcal{T}^{L})^{-1}\widetilde{g}(x)=-\frac{1}{2\pi}\int_{\mathbb{S}^1}\widetilde{\mathcal{G}^L}(x-\eta)\widetilde{g}(\eta) \ d\eta +\frac{2J}{L^2}
\end{equation}
where the kernel function $\widetilde{\mathcal{G}^L}(x)$ can be explicitly computed as
\[\widetilde{\mathcal{G}^L}(x)=\displaystyle\sum_{n=-\infty}^{n=\infty} \frac{n \sinh(nL)}{\cosh(n L)-1}e^{inx} .\]
The value of $J$ that until now is undetermined is obtained by means of the previous derived formula \eqref{nec:uni:valued:p:linear}. Using \eqref{inversion:operator:linearized} we find that
\begin{equation}
    J=\frac{L}{2\pi}\int_{0}^{2\pi}\widetilde{g}(\eta) d\eta= -L\widehat{g}(0)-(\widehat{h}^{+}(0)-\widehat{h}^{-}(0)).
\end{equation}

Once we have obtained the value of $j_0$ and $J$, we can use formula \eqref{psi:fourier:exp} which combined with Fourier inverse formula yields $\psi(x,y)$ 
and hence $W$ since $W=(0,A)+\nabla^{\perp}\psi$.  

In the following sections we will show how to solve the full non-linear problem \eqref{transport:problem}, \eqref{div:curl:problem} and \eqref{bvc} by using a perturbative argument with respect to the linear problem.

\section{The non-linear problem: an integral equation for the current}\label{Sec:3}
In this section we will derive an integral equation for the current $j$ on $\partial\Omega_{-}$, namely $j_0=j(x,0)$, $x\in \mathbb{S}^1$. As expected, this integral equation will be a perturbation of equation \eqref{inversion:operator:linearized} that we have obtained for the linearized problem. The solution of this equation will give $j_0(x)$ in terms of the perturbation magnetic field $b$ and the boundary values $f$ and $g$. In the following subsection, using a formal argument that assumes the convergence of some Fourier series, we show how to arrive to an integral equation for $j_0(x)$. We will not consider in detail the convergence of the Fourier series and the precise definitions of the operators that appeared in this section will be given later (cf. Subsection \ref{S:33}).
\subsection{The formal argument using Fourier series}\label{S:31}
Proceeding as in the Section \ref{Sec:2}, we define the operator $\Gamma: B_{M}(C^{2,\alpha}(\Omega)) \to C^{2,\alpha}(\Omega)$ using two building blocks: a transport type problem and a div-curl problem. Given $b\in B_{M}(C^{2,\alpha}(\Omega))$ we define $j\in C^{1,\alpha}(\Omega)$ as the solution to the transport type problem 
\begin{equation}\label{transport:problem:nonlinear}
\left\lbrace
\begin{array}{lll}
	((0,1)+b)\cdot\nabla j =0, \ \mbox{in} \ \Omega \\
	j= j_{0}, \ \mbox{on} \ \partial \Omega_{-}
\end{array} \right.
\end{equation} 
where $j_{0}$ is a priori an unknown quantity. As a second step, we define $W\in C^{2,\alpha}(\Omega)$ as the unique solution to the following div-curl problem
\begin{equation}\label{div:curl:problem:nonlinear}
\left\lbrace
\begin{array}{lll}
	\nabla\times W= j, \ \mbox{in} \ \Omega \\
	\mbox{div } W=0, \ \mbox{in} \ \Omega \\ 
	W\cdot n=f, \ \mbox{on} \ \partial \Omega \\
	W\cdot \tau=g, \ \mbox{on} \ \partial \Omega_{-} \\
	\int_{0}^{L} W_{1}(0,y) \ dy=J.
\end{array} \right.
\end{equation} 
Then, we define $\Gamma(b)=W$. 
By the theory of transport equations (cf. \cite[Proposition 3.8]{Alo-Velaz-2021}), it is well-known that system \eqref{transport:problem:nonlinear} can be solved by using the integral curves of the vector field $B=(0,1)+b$. More precisely, the explicit solution to \eqref{transport:problem:nonlinear} is given
\begin{equation} 
j(x,y)= j_{0}(X^{-1}(x,y))
\end{equation}
where $X^{-1}$ is the inverse of the mapping $\xi\to X(\xi,y)$ solving the ordinary differential equation 
\begin{equation}
\left\lbrace
\begin{array}{lll}
	\partial_{y} X (\xi,y)= \frac{b_{1}(X(\xi,y),y)}{1+b_{2}(X(\xi,y),y)} \\
	X(\xi,0)=\xi.
\end{array}\right.
\end{equation}
Arguing as in \eqref{stream:formulation:linear} in Section \ref{Sec:2} using the stream function $\psi$, the div-curl problem \eqref{div:curl:problem:nonlinear} becomes 
\begin{equation}\label{stream:function:general}
\left\lbrace
\begin{array}{lll}
	\Delta \psi= j_{0}(X^{-1}(x,y)),\ \mbox{in } \Omega \\ 
	\psi(x,L)= -J+h^{+}(x),  \ x\in \mathbb{R} \\
	\psi(x,0)=h^{-}(x),  \ x\in \mathbb{R}\\
	\partial_{y} \psi(x,0)= -g(x), \ x\in \mathbb{R}
\end{array}\right.
\end{equation}
where we recall that $h^{+}(x), h^{-}(x)$ and $A$ are given in \eqref{def:h} and \eqref{def:A} respectively. To solve \eqref{stream:function:general} we do not use variation of parameters but compute directly the fundamental solution $\Phi(x,y,y_0)$ solving the problem
\begin{equation}
\left\lbrace
\begin{array}{lll}
	\Delta \Phi (x,y,y_0)= \delta(x)\delta(y-y_{0}), \ \mbox{in } \Omega \\ 
	\Phi=0, \ \mbox{on } \partial \Omega. \\
\end{array}\right.
\end{equation}
Using Fourier transform and imposing the continuity jump conditions we infer that 
\begin{equation}
\Phi (x,y,y_{0})=
\left\lbrace
\begin{array}{lll}
	-\dfrac{1}{2\pi}\displaystyle\sum_{n=-\infty}^{n=\infty}\frac{\sinh(n(L-y_{0}))\sinh(ny)}{n\sinh(nL)}e^{inx}, \ \mbox{for }  y<y_{0}, \\
	-\dfrac{1}{2\pi}\displaystyle\sum_{n=-\infty}^{n=\infty}\frac{\sinh(n(L-y))\sinh(ny_{0})}{n\sinh(nL)}e^{inx}, \ \mbox{for }  y>y_{0}. 
\end{array}\right.
\end{equation}
Moreover, the normal derivative at $y=0$ is given by
\[\partial_{y} \Phi (x,0,y_{0})= -\frac{1}{2\pi} \displaystyle\sum_{n=-\infty}^{n=\infty} \frac{\sinh(n(L-y_{0}))}{\sinh(nL)} e^{inx}. \]
Computing an homogeneous solution and imposing the boundary value conditions 
\[\psi(x,L)=-J+h^{+}(x),\quad  \psi(x,0)=h^{-}(x) \mbox{ and } \partial_{y} \psi (x,0)= -g(x), \]
we conclude (similarly as in Section \ref{Sec:2}) that 
\begin{equation}
\partial_{y} \psi (x,0)= -\frac{J}{L}+
\mathcal{Z}(x)-\frac{1}{2\pi}\int_{0}^{L} dy_{0} \int_{\mathbb{S}^{1}}\displaystyle\sum_{n=-\infty}^{n=\infty} \frac{\sinh(n(L-y_{0}))}{\sinh(nL)} e^{in(x-\xi)}j_{0}(X^{-1}(\xi,y_{0})) \ d \xi \label{psi:invers:full}
\end{equation} 
where $\mathcal{Z}(x)$ is defined in \eqref{B}. Therefore we have that for $\widetilde{g}$ as in \eqref{modified:g} we can write \eqref{psi:invers:full} as the following integral equation for $j_0$

\begin{equation}\label{operator:omega:full}
\mathcal{T}^{NL}j_{0}(x)=-\frac{1}{2\pi}\int_{0}^{L} dy_{0} \int_{\mathbb{S}^{1}}\displaystyle\sum_{n=-\infty}^{n=\infty} \frac{\sinh(n(L-y_{0}))}{\sinh(nL)} e^{in(x-\xi)}j_{0}(X^{-1}(\xi,y_{0})) \ d \xi= \widetilde{g}(x)+\frac{J}{L}.
\end{equation}
Notice that the operator in \eqref{operator:omega:full} reduces to \eqref{operator:omega:linear} in the particular case where $X^{-1}(\xi,y_0)=\xi$, which corresponds to the linearized case considered in Section \ref{Sec:2}. From now on, in integral expressions like \eqref{operator:omega:full} which results in functions depending only on $x$ we replace the integration variable $y_0$ for $y$ for the sake of simplicity. We rewrite the operator equation \eqref{operator:omega:full} into a more convenient form. To that purpose, we define
\begin{equation}\label{expression:char}
\Theta(\xi,y)=X^{-1}(\xi,y)-\xi
\end{equation} 
and plugging \eqref{expression:char} in \eqref{operator:omega:full} we infer that the operator $\mathcal{T}^{NL}$ can be expressed as
\begin{align}
\mathcal{T}^{NL}j_{0}(x)=-\frac{1}{2\pi}\int_{0}^{L} dy \int_{\mathbb{S}^1} \displaystyle\sum_{n=-\infty}^{n=\infty} \frac{\sinh(n(L-y))}{\sinh(nL)} e^{in(x-\xi)}j_{0}(\xi+\Theta(\xi,y)) \ d \xi.
\end{align}
Using the following changes of variables
\begin{equation}\label{change:variables}
X^{-1}(\xi,y)=\xi+\Theta(\xi,y)=\eta, \quad  d\xi=\frac{d\eta}{(1+\partial_{\xi} \Theta(X(\eta,y),y))},
\end{equation}
we obtain
\begin{align}\label{mathfrakTo:first}
\mathcal{T}^{NL}j_{0}(x)&=-\frac{1}{2\pi}\int_{0}^{L} dy \int_{\mathbb{S}^1} \displaystyle\sum_{n=-\infty}^{n=\infty} \frac{\sinh(n(L-y))}{\sinh(nL)}e^{in(x-X(\eta,y))} j_{0}(\eta) \frac{1}{(1+\partial_{\xi} \Theta(X(\eta,y),y))} \ d\eta. 
\end{align}
Defining, $\Lambda(\eta,y)=X(\eta,y)-\eta$ we notice that the operator \eqref{mathfrakTo:first} can be written as 
\begin{align}\label{Fourier:series:operator}
\mathcal{T}^{NL}j_{0}(x)=-\frac{1}{2\pi}\int_{\mathbb{S}^1}\mathcal{G}^{NL}(x-\eta,\eta)j_{0}(\eta)  \ d\eta.
\end{align}
where
\begin{align}\label{Fourier:series:2}
\mathcal{G}^{NL}(x,\eta) = \ \displaystyle\sum_{n=-\infty}^{n=\infty}  a_{n}(\eta)e^{inx}
\end{align}
with
\begin{equation}\label{an:Fourier} 
a_{n}(\eta)=  \int_{0}^{L} \frac{\sinh(n(L-y))}{\sinh(nL)} \frac{e^{-in \Lambda (\eta,y)}}{(1+\partial_{\xi} \Theta(X(\eta,y),y))} dy.
\end{equation}
Thus the integral equation  \eqref{operator:omega:full} for $j_0$ becomes
\begin{equation}\label{full:integral:eq}
\mathcal{T}^{NL}j_{0}(x)=-\frac{1}{2\pi}\int_{\mathbb{S}^1}\mathcal{G}^{NL}(x-\eta,\eta) j_{0}(\eta)  \ d\eta=\widetilde{g}(x)+\frac{J}{L}
\end{equation}

\subsection{Decomposing the operator \texorpdfstring{$\mathcal{T}^{NL}$}{Lg}}\label{S:32}
In this subsection, we will decompose 
the operator $\mathcal{T}^{NL}$ defined in \eqref{Fourier:series:operator}-\eqref{an:Fourier} into several operators which are more tractable and easier to estimate. In particular, we will split the operator into one main term which is a convolution operator and several remainder terms which are perturbations of convolution operators.

To that purpose we first notice that the coefficients in \eqref{an:Fourier} can be written as
\begin{align*}
a_{n}(\eta) &= \int_{0}^{L} \frac{\sinh(n(L-y))}{\sinh(nL)} dy+ \int_{0}^{L} \frac{\sinh(n(L-y))}{\sinh(nL)} \left[ \frac{e^{-in \Lambda(\eta,y)}}{1+\partial_{\xi} \Theta(X(\eta,y),y))} -1 \right] dy \\
&=a_{n}^{0}+a_{n}^{1}(\eta).
\end{align*}

The first term can be easily integrated since it does not depend on $\eta$, giving 
\begin{equation}\label{a_0}
a_{n}^{0}= \frac{1}{n} \left[\frac{\cosh(nL)-1}{\sinh(nL)}\right]
\end{equation}
and the second term is split as 
\begin{align*}
a_{n}^{1}(\eta) &= \int_{0}^{L} \frac{\sinh(n(L-y))}{\sinh(nL)} \left[ \frac{e^{-in\Lambda(\eta,y)}-1}{1+\partial_{\xi} \Theta(X(\eta,y),y))} \right] dy - \int_{0}^{L} \frac{\sinh(n(L-y))}{\sinh(nL)} \left[ \frac{\partial_{\xi} \Theta(X(\eta,y),y))}{1+\partial_{\xi} \Theta(X(\eta,y),y))} \right] dy \\
&= a_{n}^{2}(\eta)+a_{n}^{3}(\eta).
\end{align*}
Moreover, we have that
\begin{align}
\frac{\sinh(n(L-y))}{\sinh(nL)}=
e^{-\abs{n}y}-M(n,y)
\end{align}
where 
\begin{equation} \label{function:M:first}
M(n,y)=\frac{e^{-2\abs{n}L}(e^{\abs{n}y}-e^{-\abs{n}y})}{(1-e^{-2\abs{n}L})}.
\end{equation}
By means of this computation, we find that
\begin{equation}\label{a_n:2}
a_{n}^{2}(\eta)=\int_{0}^{L} e^{-\abs{n}y}  \left[ \frac{e^{-in\Lambda(\eta,y)}-1}{1+\partial_{\xi} \Theta(X(\eta,y),y))} \right]  dy + R_{n}^{2}(\eta)
\end{equation}
where 
\begin{equation}\label{E_11}
R_{n}^{2}(\eta)= \int_{0}^{L} M(n,y) \left[ \frac{e^{-in\Lambda(\eta,y)}-1}{1+\partial_{\xi} \Theta(X(\eta,y),y))} \right] dy.
\end{equation}
Similarly, we have that
\begin{equation}\label{a12}
a_{n}^{3}(\eta)= \int_{0}^{L} e^{-\abs{n}y}  \left[ \frac{\partial_{\xi} \Theta(X(\eta,y),y))}{1+\partial_{\xi} \Theta(X(\eta,y),y))} \right]  dy +R_{n}^{3}(\eta)
\end{equation} 
where
\begin{equation}\label{E12}
R_{n}^{3}(\eta)= \int_{0}^{L} M(n,y) \left[ \frac{\partial_{\xi} \Theta(X(\eta,y),y))}{1+\partial_{\xi} \Theta(X(\eta,y),y))} \right] dy.
\end{equation} 
Therefore, collecting the expressions \eqref{a_0}-\eqref{E12}
\begin{align*}
a_n(\eta)= a_{n}^{0}+ a^{2}_{n}(\eta)+a^{3}_{n}(\eta)+R^{2}_{n}(\eta)+R^{3}_{n}(\eta)
\end{align*}
and using the definition of $\mathcal{G}^{NL}(x,\eta)$ given in \eqref{Fourier:series:2} we can rewrite $\mathcal{G}^{NL}(x,\eta)$ as
\begin{align*}
\mathcal{G}^{NL}(x,\eta)=\displaystyle\sum_{n=-\infty}^{n=\infty} \left(a_{n}^{0}+a_{n}^{1}(\eta)\right)e^{inx} &=\displaystyle\sum_{n=-\infty}^{n=\infty}\left(a^{0}_{n}+a^{2}_{n}(\eta)+a^{3}_{n}(\eta)+R^{2}_{n}(\eta)+R^{3}_{n}(\eta)\right)e^{inx}\\
&=\mathcal{G}^{NL}_{0}(x)+ \displaystyle\sum_{i=1}^{4} \mathcal{G}^{NL}_{i}(x,\eta)
\end{align*}
where the main term is given by
\begin{align}\label{main:G:term}
\mathcal{G}^{NL}_{0}(x)&=\displaystyle\sum_{n=-\infty}^{n=\infty}\frac{1}{n} \left[\frac{\cosh(nL)-1}{\sinh(nL)}\right]e^{inx}, 
\end{align}
and the remainder terms
\begin{align*}
\mathcal{G}^{NL}_{1}(x,\eta)&=\displaystyle\sum_{n=-\infty}^{n=\infty}e^{inx}\int_{0}^{L} e^{-\abs{n}y}  \left[ \frac{e^{-in\Lambda(\eta,y)}-1}{1+\partial_{\xi} \Theta(X(\eta,y),y))} \right]  dy,  \\
\mathcal{G}^{NL}_{2}(x,\eta)&=\displaystyle\sum_{n=-\infty}^{n=\infty}  e^{inx} \int_{0}^{L} e^{-\abs{n}y}  \left[ \frac{\partial_{\xi} \Theta(X(\eta,y),y))}{1+\partial_{\xi} \Theta(X(\eta,y),y))} \right]  dy,   \\
\mathcal{G}^{NL}_{3}(x,\eta)&=\displaystyle\sum_{n=-\infty}^{n=\infty} e^{inx}\int_{0}^{L} M(n,y) \left[ \frac{e^{-in\Lambda(\eta,y)}-1}{1+\partial_{\xi} \Theta(X(\eta,y),y))} \right] dy,  \\
\mathcal{G}^{NL}_{4}(x,\eta)&=\displaystyle\sum_{n=-\infty}^{n=\infty}  e^{inx} \int_{0}^{L} M(n,y)  \left[ \frac{\partial_{\xi} \Theta(X(\eta,y),y))}{1+\partial_{\xi} \Theta(X(\eta,y),y))} \right]  dy.
\end{align*}

Using this decomposition we write the operator $\mathcal{T}^{NL}$ in \eqref{Fourier:series:operator} as 
\begin{equation}\label{decomposition:operator:T}
\mathcal{T}^{NL}j_{0}(x)=\mathcal{T}^{NL}_{0}j_{0}(x)+ \displaystyle\sum_{i=1}^{4}\mathcal{T}^{NL}_{i}j_{0}(x)
\end{equation}
where
\begin{align}
\mathcal{T}^{NL}_{0}j_{0}(x)&=-\frac{1}{2\pi}\displaystyle\int_{\mathbb{S}^{1}}\mathcal{G}^{NL}_{0}(x-\eta)j_{0}(\eta) \ d\eta, \label{TNL:0}\\
\mathcal{T}^{NL}_{i}j_{0}(x)&=-\frac{1}{2\pi}\displaystyle\int_{\mathbb{S}^{1}}\mathcal{G}^{NL}_{i}(x-\eta,\eta)j_{0}(\eta)  d\eta, \quad \mbox{for } i=1,\ldots,4. \label{TNL:j}
\end{align}
\begin{remark}
Notice that the main term $\mathcal{G}^{NL}_{0}(x)$ does not depend on $\eta$ and coincides with the linearized kernel $\mathcal{G}^{L}(x)$ in \eqref{operator:convo:L}. Therefore, $\mathcal{T}^{NL}_{0}$ is a convolution operator that can be inverted using Fourier series. 
\end{remark}
We can formally rewrite the integral equation \eqref{full:integral:eq} for $j_0$ in the form a second order Fredholm integral equation. Indeed, using the fact that the operator $\mathcal{T}^{NL}_{0}$ is a convolution that can be inverted using Fourier series, we can write equation \eqref{full:integral:eq} as 
\begin{equation}\label{integral:eq:fixed}
j_{0}(x)+\displaystyle\sum_{i=1}^{4}\mathsf{T}_{i}j_{0}(x) =\mathsf{G}(x)+\frac{2J}{L^2}
\end{equation}
where
\begin{equation}\label{tildeT:tilde:g}
\mathsf{T}_{i}=\big[\mathcal{T}^{NL}_{0}\big]^{-1}\mathcal{T}^{NL}_{i}, \quad \mbox{for } i=1,\ldots,4 \quad \mbox{and } \mathsf{G}=\big[\mathcal{T}^{NL}_{0}\big]^{-1}\widetilde{g},
\end{equation}
with $\widetilde{g}$ defined in \eqref{modified:g}. We now argue as in the case of the linearized problem and explain how to choose $J$ in order to obtain a uni-valued pressure function $p$ on $\Omega$. To this end, we use equation \eqref{transport:problem:nonlinear} to construct the pressure $p$ by means of the following identity
	\begin{equation}\label{pressure:value:2}
		p(\bm{x})=\int_{\bm{0}}^{\bm{x}} \big[j\times B\big](\bm{y})
		\cdot d\bm{y}
	\end{equation}
	where the integral on the right hand side is the line integration computed along any curve connecting $\bm{0}=(0,0)$ and $\bm{x}\in\Omega$. The function $p$ given by \eqref{pressure:value:2} is uni-valued in $\Omega$ if and only if 
	\begin{equation}\label{condition:pressure:1:nonlineal}
	    \int_{0}^{2\pi}[j\times B]_{1}(x,0) \ dx = 0,
	\end{equation}
where $(j\times B)_{1}$ denotes the first component of the vector $j\times B$. Moreover, we notice that $(j\times B)_{1}(x,0)=-j_0(x)(1+b_{2}(x,0))=-j_{0}(x)(1+f^{-}(x))$ where $f^{-}= f|_{\partial\Omega_{-}}$. Then, \eqref{condition:pressure:1:nonlineal} is equivalent to 
\begin{equation}\label{combination:1}
  	    \int_{0}^{2\pi} j_0(x)(1+f^{-}(x))\ dx = 0.
\end{equation}
Using \eqref{integral:eq:fixed} we find that
\begin{equation}\label{combination:2}
  	    \frac{1}{2\pi}\int_{0}^{2\pi} j_0(x) dx =\frac{2J}{L^{2}} + \frac{1}{2\pi}\int_{0}^{2\pi}\mathsf{G}(x)  dx -\frac{1}{2\pi} \displaystyle\sum_{i=1}^{4}\int_{0}^{2\pi} \mathsf{T}_{i}j_{0}(x) \ dx.
\end{equation}
Combining \eqref{combination:1} and \eqref{combination:2} we obtain that
\begin{equation}\label{combination:3}
\frac{2J}{L^2}=-\frac{1}{2\pi}\int_{0}^{2\pi}\mathsf{G}(x)  dx+\frac{1}{2\pi} \displaystyle\sum_{i=1}^{4}\int_{0}^{2\pi} \mathsf{T}_{i}j_{0}(x) \ dx \ dx -\frac{1}{2\pi}\int_{0}^{2\pi} j_0(x)f^{-}(x) \ dx
\end{equation}
Plugging \eqref{combination:3} into \eqref{integral:eq:fixed} and denoting by $\langle h \rangle=\frac{1}{2\pi} \int_{0}^{2\pi} h(x) \ dx$ we have that 
\begin{equation}\label{final:eq:j_0}
j_0(x)=- \displaystyle\sum_{i=1}^{4}(\mathsf{T}_{i}j_{0}(x)-\langle \mathsf{T}_{i}j_{0} \rangle)+\mathsf{G}(x)-\langle \mathsf{G} \rangle-\langle j_0 f^{-}\rangle.
\end{equation}
The problem \eqref{final:eq:j_0} is a fixed point type of equation which will be shown to be equivalent to the solution $(B,p)$. Indeed, after solving equation \eqref{final:eq:j_0}, we can obtain the value of $j(x,y)$ in $\Omega$ using the transport type problem \eqref{transport:problem:nonlinear} and recover the new magnetic field $W$ using the div-curl system \eqref{div:curl:problem:nonlinear}. 
\subsection{A rigorous formulation of the problem}\label{S:33}
The previous computations in Subsections \ref{S:31} and \ref{S:32} are purely formal, since we did not consider in a rigorous manner the convergence of the Fourier series.  In this subsection, we will give a precise meaning of the integral equation \eqref{final:eq:j_0} for $j_0$.
To this end, we
first give a detailed definition of the operators $\mathsf{T}_{1},\ldots, \mathsf{T}_{4}$ in \eqref{tildeT:tilde:g}. We defined the operators $\mathsf{T}_{1}$ and $\mathsf{T}_{2}$ as
\begin{align}
\mathsf{T}_{1}j_{0}(x)&=- \frac{1}{2\pi}\displaystyle \lim_{\epsilon\to 0^{+}}\int_{\mathbb{S}^{1}} \mathfrak{G}_{1,\epsilon}(x-\eta,\eta)j_{0}(\eta) \ d\eta, \label{def:T1} \\
\mathsf{T}_{2}j_{0}(x)&=-\frac{1}{2\pi}\displaystyle \lim_{\epsilon\to 0^{+}}\int_{\mathbb{S}^{1}} \mathfrak{G}_{2,\epsilon}(x-\eta,\eta)j_{0}(\eta) \ d\eta, \label{def:T2}
\end{align}
where
\begin{align}\label{operadorG1G2:frak}
\mathfrak{G}_{1,\epsilon}(x,\eta)&=\displaystyle\sum_{n=-\infty}^{n=\infty}\frac{n\sinh(nL)}{(\cosh(nL)-1)}e^{inx}\int_{\epsilon}^{L} e^{-\abs{n}y}  \left[ \frac{e^{-in\Lambda(\eta,y)}-1}{1+\partial_{\xi} \Theta(X(\eta,y),y))} \right]  dy, \\
\mathfrak{G}_{2,\epsilon}(x,\eta)&=\displaystyle\sum_{n=-\infty}^{n=\infty}\frac{n\sinh(nL)}{(\cosh(nL)-1)}e^{inx}\int_{\epsilon}^{L} e^{-\abs{n}y} \left[ \frac{\partial_{\xi} \Theta(X(\eta,y),y))}{1+\partial_{\xi} \Theta(X(\eta,y),y))} \right]  dy.
\end{align}
On the other hand, the operators $\mathsf{T}_{3},\mathsf{T}_{4}$ are given by
\begin{align}
\mathsf{T}_{3}j_{0}(x)=-\frac{1}{2\pi}\int_{\mathbb{S}^{1}}  \mathfrak{G}_{3}(x-\eta,\eta)j_{0}(\eta) \ d\eta, \label{def:T3} \\
\mathsf{T}_{4}j_{0}(x)=-\frac{1}{2\pi}\int_{\mathbb{S}^{1}} \mathfrak{G}_{4}(x-\eta,\eta)j_{0}(\eta) \ d\eta,\label{def:T4}
\end{align}
where
\begin{align}
\mathfrak{G}_{3}(x,\eta)&=\displaystyle\sum_{n=-\infty}^{n=\infty}\frac{n\sinh(nL)}{(\cosh(nL)-1)}e^{inx}\int_{0}^{L} M(n,y)  \left[ \frac{e^{-in\Lambda(\eta,y)}-1}{1+\partial_{\xi} \Theta(X(\eta,y),y))} \right] dy,  \\
\mathfrak{G}_{4}(x,\eta)&=\displaystyle\sum_{n=-\infty}^{n=\infty} \frac{n\sinh(nL)}{(\cosh(nL)-1)} e^{inx} \int_{0}^{L}M(n,y) \left[ \frac{\partial_{\xi} \Theta(X(\eta,y),y))}{1+\partial_{\xi} \Theta(X(\eta,y),y))} \right]  dy, \end{align}
with $M(n,y)$ as in \eqref{function:M:first}. The operators $\mathsf{T}_{i}$ for $i=1,\ldots, 4$ will act on functions $j_0$ on some suitable H\"older spaces. The fact that the operators $\mathsf{T}_{i}$ for $i=1,\ldots, 4$ in this spaces are well-defined operators will be shown in Section \ref{S4}. For instance the reason why operators $\mathsf{T}_{3}, \mathsf{T}_{4}$ are well defined acting on H\"older functions $j_0$ readily follows from the fact that $\mathfrak{G}_{3}(x,\eta),\mathfrak{G}_{4}(x,\eta)$ are $C^\infty$ in $x$ due with the exponential decay of the function $M(n,y)$ as $\abs{n}\to \infty.$ To deal with operators $\mathsf{T}_{1},\mathsf{T}_{2}$ some refined estimates for perturbations of non-convolution singular integral operators will be required.

We now define in a precise manner the operator $\mathcal{T}^{NL}_{0}$ in \eqref{TNL:0}. On the one hand notice that
\begin{equation}\label{Fourier:expression:T0NL}
\frac{1}{\abs{n}}\frac{\cosh(nL)-1}{\sinh(nL)}=\frac{1}{\abs{n}}+Q_{n} \mbox{ for } n\neq 0, \mbox{ and }  \frac{1}{\abs{n}}\frac{\cosh(nL)-1}{\sinh(nL)}=Q_0
\end{equation}
where 
$$Q_{n}=\frac{1}{n}\left(\frac{\cosh(nL)-1}{\sinh(nL)}-\mbox{sgn}(n)\right), \mbox{ for } n\neq 0, \mbox{ and } Q_0=\frac{L}{2}.$$
We recall that the periodic Hilbert transform denoted by $\mathcal{H}$ is given in Fourier side as 
\begin{equation}
\widehat{\mathcal{H}f}(n)=-i\mbox{sgn}(n)\widehat{f}(n), n\in \mathbb{N}
\end{equation}
and define the linear operator $\partial_{x}^{-1}:L^2(\SS^1)\to H^1(\SS^1)$ by means of 
\begin{equation}
\partial_{x}^{-1}\psi(x)=
\begin{cases}
	\displaystyle\int_{0}^{x} \psi(\xi) \ d\xi-\frac{1}{2\pi}\displaystyle\int_{0}^{2\pi}\psi(\xi) \xi \ d\xi, \mbox{ if } \displaystyle\int_{0}^{2\pi} \psi(x) \ dx=0, \\
	0, \mbox{ if } \psi(x)= 1.
\end{cases}
\end{equation} 

Then it is natural to define $\mathcal{T}^{NL}_{0}$ in \eqref{TNL:0} as
\begin{equation}\label{rig:TNL0}
\mathcal{T}^{NL}_{0}\psi(x)=- \mathcal{H}\partial_{x}^{-1}\psi (x) +\int_{\SS^1} Q(x-\eta)\psi(\eta) \ d\eta,
\end{equation}
where $$Q(x)=\frac{1}{2\pi}\displaystyle\sum_{n=-\infty}^{n=\infty} Q_{n}e^{inx}.$$

On the other hand, notice that the derivative operator $\partial_{x}$ is the inverse of $\partial_{x}^{-1}$, i.e. $\partial_{x}\circ \partial_{x}^{-1}= \mathcal{I}$, where $\mathcal{I}$ denotes the identity operator. Hence, we find that the inverse operator $(\mathcal{T}^{NL}_{0})^{-1}$ is given by
\begin{equation}\label{rig:inverse:TNLO}
(\mathcal{T}^{NL}_{0})^{-1}\psi(x)=
\mathcal{H}\partial_{x}\psi(x)-\int_{\SS^1}\widetilde{Q}(x-\eta) \psi(\eta) \ d\eta,
\end{equation}
where $$\widetilde{Q}(x)=\frac{1}{2\pi}\displaystyle\sum_{n=-\infty}^{n=\infty}\widetilde{Q}_{n}e^{inx}$$ and 
$$\widetilde{Q}_{n}=\left(n\frac{\sinh(nL)}{\cosh(nL)-1}-\abs{n}\right), \mbox{ for } n\neq 0, \quad \widetilde{Q}_{0}=\frac{2}{L} \mbox{ for } n=0.$$
Thus, we have that $(\mathcal{T}^{NL}_{0})^{-1}\circ (\mathcal{T}^{NL}_{0})=(\mathcal{T}^{NL}_{0})\circ (\mathcal{T}^{NL}_{0})^{-1}= \mathcal{I}$. The easiest way to check this identity is to use the Fourier expression for $\mathcal{T}^{NL}_{0}$ on the left hand side in \eqref{Fourier:expression:T0NL}. With these definitions at hand we have that the function $\mathsf{G}(x)$ in \eqref{tildeT:tilde:g} is given by
\begin{equation}\label{func:G:rig}
\mathsf{G}(x)=(\mathcal{T}_{0}^{NL})^{-1}\widetilde{g}(x)=\mathcal{H}\partial_{x}\widetilde{g}(x)-\int_{\SS^1} \widetilde{Q}(x-\eta)\widetilde{g}(\eta) \ d\eta
\end{equation}
where 
$$\widetilde{g}(x)= -g(x)-\mathcal{Z}(x)$$
where $\mathcal{Z}$ is defined in \eqref{B}. 
Roughly, speaking the function $\widetilde{g}$ and hence $\mathsf{G}$ take into account the given boundary value conditions $f$ on $\partial\Omega$, $g$ on  $\partial\Omega_{-}$ satisfied by the magnetic fields.

To conclude, we will define the operators given in \eqref{TNL:j} for $i=1,\ldots,4$ as
\begin{equation}
\mathcal{T}_{i}^{NL}\psi(x)= \mathcal{T}_{0}^{NL}\mathsf{T}_{i}\psi(x)    \end{equation}
where $\mathcal{T}_{0}^{NL}$ is given by \eqref{rig:TNL0} and $\mathsf{T}_{i}$ are given as \eqref{def:T1}, \eqref{def:T2}, \eqref{def:T3} and \eqref{def:T4}.
\begin{remark}
The fact that the operators $\mathcal{T}^{NL}_{i}$ for $i=1,\ldots,4$ can be written as in  \eqref{TNL:j} acting on spaces of H\"older functions will be proved at the end of the paper, cf. Corollary \ref{Cor:71} in Subsection \ref{S:71}.
\end{remark}

\section{H\"older estimates for non-convolution singular integral operators }\label{S4}
In order to show the existence and uniqueness of solutions of equation \eqref{final:eq:j_0}, we will need to derive bounds for the operators $\mathsf{T}_{1},\ldots,\mathsf{T}_{4}$ in the functional spaces $C^{1,\alpha}$ and $C^{\alpha}$. To that purpose, we will first derive in this section some general lemmas showing $C^{1,\alpha}$ and $C^{\alpha}$  H\"older estimates for non-convolution singular integral operators. These operators differ from convolutions because they contain a function $\Lambda:\Omega \to \mathbb{S}^{1}$. Estimates for these operators in H\"older norms will be shown assuming a suitable smallness condition on $\Lambda$ which will be used repeatedly in the rest of the paper. More precisely, the assumptions reads 

\begin{Assumption}\label{assumption:Lambda}
Let us assume that the function $\Lambda:\Omega \to \mathbb{S}^{1}$ has $C^{2,\alpha}(\Omega)$ regularity and satisfies that $\Lambda(\eta,0)=0$. Moreover, there exists $\delta_0\in (0,\frac{1}{2})$ such that
\begin{equation*}
	\norm{\Lambda}_{C^{2,\alpha}(\Omega)}\leq \delta_0.
\end{equation*}
\end{Assumption}

Let us start with the following calculus lemma that will be used throughout this section.
\begin{lemma}\label{calculus:lemma}
Then there exists a numerical constant $c_0>0$ such that for any $\Lambda$ satisfying Assumption \ref{assumption:Lambda} the following inequality holds
\begin{equation}\label{est:calcu} 
	\abs{1-e^{i(x-\eta)-y-i\Lambda(\eta,y)}}\geq \displaystyle\min \{\frac{1}{2\sqrt{2}} \sqrt{(x-\eta)^2+y^{2}},c_{0}\}
\end{equation}
for $x\in [-\pi,\pi]$, $\eta\in[x-\pi,x+\pi]$ and $y\in[0,L]$.
\end{lemma}
\begin{proof}
Denoting by $z=i(x-\eta)-y-i\Lambda(\eta,y)$, we have that for $\abs{z}\leq \frac{1}{2}$
\begin{equation}\label{4.2}
	\abs{e^{z}-1}\geq \frac{1}{2}\abs{z}, \mbox{ for } z\in \mathbb{C}.
\end{equation}
Indeed, a straightforward calculation shows that $$\abs{e^{z}-z-1}=\abs{\int_{0}^{z} (e^{\xi}-1) \ d\xi}\leq \frac{\sqrt{e}}{2}\abs{z}^2, \mbox{ for } \abs{z}\leq \frac{1}{2}$$ and hence 
\begin{equation}\label{eq:1}
	\abs{e^{z}-1}\geq \abs{z}- \int_{0}^{z} (e^{\xi}-1) d\xi \geq (1- \frac{\sqrt{e}}{4})\abs{z}, \mbox{ for } \abs{z}\leq \frac{1}{2}.
\end{equation}
Therefore,  \eqref{4.2} follows.
Furthermore,
\begin{align}
	\abs{i(x-\eta)-y-i\Lambda(\eta,y)}=\sqrt{(x-\eta)^2-2\Lambda (x-\eta)+y^2+\Lambda^2}. \label{equality:abs}
\end{align}
Since by assumption $\Lambda(\eta,0)=0$ and $\norm{\Lambda}_{C^{1}(\Omega)}\leq \frac{1}{2}$ we find that
\begin{equation}
	\abs{\Lambda(\eta,y)}\leq  \frac{1}{2} y, \mbox{ for } 0\leq y \leq L.
\end{equation} 
Applying Young's inequality in \eqref{equality:abs} yields
\begin{align}
	\abs{i(x-\eta)-y-i\Lambda(\eta,y)}\geq \sqrt{(x-\eta)^2-\frac{1}{2}(x-\eta)^{2}-\frac{1}{2}y^{2}+y^{2}},
\end{align}
and hence
\begin{equation}\label{eq:2}
	\abs{i(x-\eta)-y-i\Lambda(\eta,y)}\geq \frac{1}{\sqrt{2}} \sqrt{(x-\eta)^2+y^{2}}
\end{equation}
$\mbox{ for }x\in [-\pi,\pi], \eta\in[x-\pi,x+\pi]\mbox{ and } y\in[0,L].$
On the other hand for $\abs{z}\geq \frac{1}{2}$ one can readily check that for $\norm{\Lambda}_{C^{1}(\Omega)}\leq \delta_0$ we have that
\begin{equation}\label{eq:3}
	\displaystyle\min_{\mathcal{A}} \abs{1-e^{z}}\equiv c_{0}, \mbox{ for } \mathcal{A}=\{z\in\mathbb{C}:\abs{z}\geq \frac{1}{2}, \abs{Im(z)}\leq \pi\}
\end{equation}
where the constant $c_0$ is independent of $\Lambda$. Combining \eqref{eq:1}, \eqref{eq:2} and \eqref{eq:3} we conclude that estimate \eqref{est:calcu} follows.
\end{proof}
\begin{remark}
Notice that in Assumption \ref{assumption:Lambda} we imposed that $\Lambda\in C^{2,\alpha}(\Omega)$, however the proof of Lemma \ref{calculus:lemma} can be shown only assuming $\Lambda\in C^{1}(\Omega)$. However, later in the application we will use this stronger regularity assumption and therefore we prefer to already state the calculus lemma for $\Lambda\in C^{2,\alpha}(\Omega)$.
\end{remark}
\subsection{\texorpdfstring{$C^{\alpha}$}{Lg} H\"older estimates}
In this subsection, we provide a $C^\alpha$ H\"older estimates for a type of non-convolution singular integral operators.
\begin{proposition}[$C^\alpha$ estimate]\label{Holder:singular:integral:alpha}
Let $H(\eta, y)\in C^{\alpha}(\Omega)$ and let Assumption \ref{assumption:Lambda} hold. Then for any $x\in \mathbb{S}^{1}$ the following limit exists
\begin{equation} \label{int:sin:1}
	\displaystyle\lim_{\epsilon\to 0^{+}} \Xi_{\epsilon}(x)=\Xi(x),
\end{equation} 
where
\begin{equation} \label{int:sin:1:eps}
	\Xi_{\epsilon}(x)= \int_{\epsilon}^{L}dy\int_{\mathbb{S}^{1}}d\eta \  \p_{x}\mathcal{A}(x,y,\eta) H( \eta, y), \ \epsilon>0
\end{equation}
with
\begin{equation}\label{function:A:1}
	\mathcal{A}(x,y,\eta)=\left(\frac{e^{i(x-\eta)-y}y}{(1-e^{i(x-\eta)-y})(1-e^{i(x-\eta)-y-i\Lambda(\eta,y)})}\right).
\end{equation}
Moreover, we have that
\begin{align}
	\norm{\Xi}_{C^{\alpha}(\mathbb{S}^{1})}&\leq C \norm{H}_{C^{\alpha}(\Omega)} \label{cota:singular:1} 
\end{align}
with $C>0$.
\end{proposition}

\begin{proof}[Proof of Proposition \ref{Holder:singular:integral:alpha}]
In order to check that the left hand side in \eqref{int:sin:1} exists, we first notice that
\begin{align}
	\p_{x} \mathcal{A}(x,y,\eta)&= -\p_{\eta} \mathcal{A}(x,y,\eta)- \mathcal{R}(x,y,\eta) \label{can:expl}
\end{align}
where 
\begin{equation}\label{rest}
	\mathcal{R}(x,y,\eta)= iy\frac{e^{i(x-\eta)-y}\p_{\eta}\Lambda(\eta,y)e^{i(x-\eta)-y-i\Lambda(\eta,y)}}{(1-e^{i(x-\eta)-y})(1-e^{i(x-\eta)-y-i\Lambda(\eta,y)})^2}.
\end{equation}
Therefore, using the fact that
$\int_{\mathbb{S}^{1}} d\eta \ \p_{\eta}(\ldots)=0
$
we can rewrite \eqref{int:sin:1} as  $$\Xi_{\epsilon}(x)= \Xi_{1,\epsilon}(x)+\Xi_{2,\epsilon}(x),$$ where
\begin{align}
	\Xi_{1,\epsilon}(x)&= -\int_{\epsilon}^{L}dy\int_{\mathbb{S}^{1}}d\eta \  \p_{\eta} \mathcal{A}(x,y,\eta)\left(H( \eta, y)-H(x,0)\right),   \label{Xi:1:eps} \\
	\Xi_{2,\epsilon}(x)&= - \int_{\epsilon}^{L}dy\int_{\mathbb{S}^{1}}d\eta \ \mathcal{R}(x,y,\eta) H( \eta, y). \label{Xi:2:eps}
\end{align}
Expanding the derivative in \eqref{Xi:1:eps} and manipulating the corresponding expression, we have that
\begin{equation} 
	\Xi_{1,\epsilon}(x)=\displaystyle\sum_{j=1}^{4}I_{j,\epsilon}\label{splitting:1}
\end{equation}
where
\begin{equation}
	I_{j,\epsilon}= \int_{\epsilon}^{L}dy\int_{\mathbb{S}^{1}}d\eta \ i_{j,\epsilon}(x,\eta,y), \mbox{  for } j=1,\ldots,4
\end{equation}
where
\begin{align*}
	i_{1,\epsilon}(x,\eta,y) &=  \frac{iye^{i(x-\eta)-y}}{(1-e^{i(x-\eta)-y})(1-e^{i(x-\eta)-y-i\Lambda(\eta,y)})} \left(H( \eta, y)-H(x,0)\right), \\
	i_{2,\epsilon}(x,\eta,y) &= \frac{iy(e^{i(x-\eta)-y})^2}{(1-e^{i(x-\eta)-y})^2(1-e^{i(x-\eta)-y-i\Lambda(\eta,y)})} \left(H( \eta, y)-H(x,0)\right), \\
	i_{3,\epsilon}(x,\eta,y) &= \frac{iye^{i(x-\eta)-y}e^{i(x-\eta)-y-i\Lambda(\eta,y)}}{(1-e^{i(x-\eta)-y})(1-e^{i(x-\eta)-y-i\Lambda(\eta,y)})^2} \left(H( \eta, y)-H(x,0)\right), \\
	i_{4,\epsilon}(x,\eta,y) &=  \frac{iye^{i(x-\eta)-y}e^{i(x-\eta)-y-i\Lambda(\eta,y)}\p_{\eta}\Lambda(\eta,y)}{(1-e^{i(x-\eta)-y})(1-e^{i(x-\eta)-y-i\Lambda(\eta,y)})^2} \left(H( \eta, y)-H(x,0)\right). 
\end{align*}
Identifying $\SS^1$ with $I_{x}=[x-\pi,x+\pi]$ for $x\in [-\pi,\pi]$  and using the bound \eqref{est:calcu} as well as the H\"older regularity for $H$ we obtain
\begin{align}
	\abs{i_{1,\epsilon}} &\leq C \norm{H}_{C^{\alpha}(\Omega)}  \frac{y(\abs{x-\eta}^{\alpha}+y^{\alpha})}{(x-\eta)^{2}+y^{2}}, \\
	\abs{i_{2,\epsilon}}+\abs{i_{3,\epsilon}} &\leq C \norm{H}_{C^{\alpha}(\Omega)}  \frac{y(\abs{x-\eta}^{\alpha}+y^{\alpha})}{\left((x-\eta)^{2}+y^{2}\right)^{3/2}}
\end{align}
where $C>0$ is independent on $\epsilon$. 
Therefore, using the dominated convergence it follows $\displaystyle\lim_{\epsilon\to 0^{+}}I_{j,\epsilon}$ exists, for $j=1,\ldots,3.$
On the other hand, we can combine $I_{4,\epsilon}$ and $\Xi_{2,\epsilon}$ as 
\begin{equation}\label{combiningI4+Xi2}
	I_{4,\epsilon}+\Xi_{2,\epsilon}=\int_{\epsilon}^{L}dy\int_{\mathbb{S}^{1}}d\eta \ i_{4,\epsilon}(x,\eta,y), 
\end{equation}
where 
\begin{equation}
	i_{4,\epsilon}(x,\eta,y)=-H(x,0)\frac{iye^{i(x-\eta)-y}e^{i(x-\eta)-y-i\Lambda(\eta,y)}\p_{\eta}\Lambda(\eta,y)}{(1-e^{i(x-\eta)-y})(1-e^{i(x-\eta)-y-i\Lambda(\eta,y)})^2}.   
\end{equation}
Hence using again \eqref{est:calcu} and Assumption \ref{assumption:Lambda}, we have that
\begin{align}\label{boundI4+Xi12}
	\abs{i_{4,\epsilon}}\leq C \norm{H}_{L^\infty(\Omega)}\frac{y^2}{((x-\eta)^{2}+y^{2})^{3/2}}.
\end{align}
Similarly using dominated convergence it follows that the limit
$\lim_{\epsilon\to 0^{+}}\left(I_{4,\epsilon}+\Xi_{2,\epsilon}\right)$ exists.
Therefore, the limit on the left hand side of \eqref{int:sin:1} exists and the function $\Xi(x)$ is well-defined. Moreover, we have the pointwise bounds
\begin{align}
	\abs{\Xi_{\epsilon}(x)}&\leq C\norm{H}_{C^{\alpha}(\Omega)}, \quad \mbox{for } x\in \mathbb{S}^{1}, \epsilon>0,\label{Linfty:Xi:epsilon:bound} \\
	\abs{\Xi(x)}&\leq C\norm{H}_{C^{\alpha}(\Omega)}, \quad \mbox{for } x\in \mathbb{S}^{1}. \label{Linfty:Xi:bound}
\end{align}
We now proceed with the $\alpha$-H\"older semi-norm. More precisely, we will show that
\begin{equation*}
	\abs{\Xi(x_1)-\Xi(x_2)}\leq C \abs{x_1-x_2}^{\alpha} \norm{H}_{C^{\alpha}(\Omega)}, \quad \mbox{for } x_{1}, x_{2}\in \mathbb{S}^{1}.
\end{equation*}
Due to the translation invariance of the estimate it suffices to check, without loss of generality, that the bound holds for $x_2=0$ and $x_{1}=x$, namely
\begin{equation}
	\abs{\Xi(x)-\Xi(0)}\leq C\abs{x}^{\alpha}  \norm{H}_{C^{\alpha}(\Omega)},  \quad \mbox{for } x \in \mathbb{S}^{1}.
\end{equation}

To that purpose, by means of equation \eqref{Xi:1:eps}-\eqref{Xi:2:eps}, we compute the difference 
\begin{equation}\label{Xi:epsilon:difference}
	\Xi_{\epsilon}(x)-\Xi_{\epsilon}(0)=\bigg[\Xi_{1,\epsilon}(x)-\Xi_{1,\epsilon}(0) \bigg]+\bigg[\Xi_{2,\epsilon}(x)-\Xi_{2,\epsilon}(0) \bigg]=J_{1,\epsilon}+J_{2,\epsilon}
\end{equation}
where
\begin{align*}
	J_{1,\epsilon}=&-\int_{\epsilon}^{L}dy\int_{\mathbb{S}^{1}}d\eta \bigg[ \p_{\eta} \left(\frac{e^{i(x-\eta)-y}y}{(1-e^{i(x-\eta)-y})(1-e^{i(x-\eta)-y-i\Lambda(\eta,y)})}\right) \left(H( \eta, y)-H(x,0)\right) \nonumber \\
	&\quad \hspace{2cm} -\p_{\eta} \left(\frac{e^{-i\eta-y}y}{(1-e^{-i\eta-y})(1-e^{-i\eta-y-i\Lambda(\eta,y)})}\right) \left(H( \eta, y)-H(0,0)\right) \bigg]\\
	J_{2,\epsilon}=&- \int_{\epsilon}^{L}dy\int_{\mathbb{S}^{1}}d\eta  \bigg[\frac{e^{i(x-\eta)-y}y\p_{\eta}\Lambda(\eta,y)e^{i(x-\eta)-y-i\Lambda(\eta,y)}}{(1-e^{i(x-\eta)-y})(1-e^{i(x-\eta)-y-i\Lambda(\eta,y)})^2} H( \eta, y)\\
	&\quad \hspace{3cm}    -\frac{e^{-i\eta-y}y\p_{\eta}\Lambda(\eta,y)e^{-i\eta-y-i\Lambda(\eta,y)}}{(1-e^{-i\eta-y})(1-e^{-i\eta-y-i\Lambda(\eta,y)})^2} H( \eta, y) \bigg].
\end{align*}
Notice that the functions $J_{1,\epsilon}$ and $J_{2,\epsilon}$ depend on $x$, but do not write it explicitly for the sake of simplicity. Moreover, recall that using the arguments above we have that the limits $\lim_{\epsilon\to 0^{+}}J_{1,\epsilon}$ and $\lim_{\epsilon\to 0^{+}}J_{2,\epsilon}$ exist.
Expanding the derivative we can split the integral in the following manner
$$J_{1,\epsilon}+J_{2,\epsilon}=J_{11,\epsilon}+J_{12,\epsilon}+J_{13,\epsilon}+J_{14,\epsilon}$$
where
\begin{eqnarray*}
	J_{11,\epsilon}&=& \int_{\epsilon}^{L}dy\int_{\mathbb{S}^{1}}d\eta \bigg[ \frac{iye^{i(x-\eta)-y}}{(1-e^{i(x-\eta)-y})(1-e^{i(x-\eta)-y-i\Lambda(\eta,y)})} \left(H( \eta, y)-H(x,0)\right)\\
	&&\quad \hspace{3cm}-\frac{iye^{-i\eta-y}}{(1-e^{-i\eta-y})(1-e^{-i\eta-y-i\Lambda(\eta,y)})} \left(H( \eta, y)-H(0,0)\right) \bigg] \\
	J_{12,\epsilon}&=& \int_{\epsilon}^{L}dy\int_{\mathbb{S}^{1}}d\eta \bigg[\frac{iy(e^{i(x-\eta)-y})^2}{(1-e^{i(x-\eta)-y})^2(1-e^{i(x-\eta)-y-i\Lambda(\eta,y)})} \left(H( \eta, y)-H(x,0)\right)\\
	&&\quad \hspace{3cm}-\frac{iy(e^{-i\eta-y})^2}{(1-e^{-i\eta-y})^2(1-e^{-i\eta-y-i\Lambda(\eta,y)})} \left(H( \eta, y)-H(0,0)\right) \bigg] \\
	J_{13,\epsilon}&=& \int_{\epsilon}^{L}dy\int_{\mathbb{S}^{1}}d\eta \bigg[\frac{iye^{i(x-\eta)-y}e^{i(x-\eta)-y-i\Lambda(\eta,y)}}{(1-e^{i(x-\eta)-y})(1-e^{i(x-\eta)-y-i\Lambda(\eta,y)})^2} \left(H( \eta, y)-H(x,0)\right)\\
	&&\quad \hspace{3cm}-\frac{iye^{-i\eta-y}e^{-i\eta)-y-i\Lambda(\eta,y)}}{(1-e^{-i\eta-y})(1-e^{-i\eta-y-i\Lambda(\eta,y)})^2} \left(H( \eta, y)-H(0,0)\right) \bigg] \\
	J_{14,\epsilon}&=&-\int_{\epsilon}^{L}dy\int_{\mathbb{S}^{1}}d\eta \bigg[\frac{iye^{i(x-\eta)-y}e^{i(x-\eta)-y-i\Lambda(\eta,y)}\p_{\eta}\Lambda(\eta,y)}{(1-e^{i(x-\eta)-y})(1-e^{i(x-\eta)-y-i\Lambda(\eta,y)})^2}H(x,0)\\
	&\quad& \hspace{3cm}-\frac{iye^{-i\eta-y}e^{-i\eta-y-i\Lambda(\eta,y)}\p_{\eta}\Lambda(\eta,y)}{(1-e^{-i\eta-y})(1-e^{-i\eta-y-i\Lambda(\eta,y)})^2} H(0,0)\bigg].
\end{eqnarray*}
Notice that in $J_{14,\epsilon}$ we have combined one of the terms resulting in $J_{1,\epsilon}$ with $J_{2,\epsilon}$ in the same way that we combined the term $I_{4,\epsilon}+\Xi_{2,\epsilon}$ in \eqref{combiningI4+Xi2}.

We divide the region of integration $\{ (\eta,y))\in \mathbb{S}^1 \times [\epsilon,L]\}$ into sets of the form 
\begin{align*}
	R^{\epsilon}_{\Omega,\leq}&=\{(\eta,y)\in \Omega:\displaystyle\max\{\abs{y},\abs{\eta}\}\leq 2\abs{x} \mbox{ and } \epsilon\leq y\leq L\}, \\
	R^{\epsilon}_{\Omega,>}&=\{(\eta,y)\in \Omega:\displaystyle\max\{\abs{y},\abs{\eta}\}> 2\abs{x} \mbox{ and } \epsilon\leq y\leq L\},
\end{align*} 
for $\epsilon\geq 0$ and estimate each integral in the different sets. For the sake of simplicity we will write $R^{0}_{\Omega,\leq}=R_{\Omega,\leq}$ and $R^{0}_{\Omega,>}=R_{\Omega,>}$.
Therefore, we have
\begin{equation*}
	J_{1k,\epsilon}=\int_{R^{\epsilon}_{\Omega,\leq}} dy d\eta \big[\ldots\big]+ \int_{R^{\epsilon}_{\Omega,>}}dy d\eta\big[\ldots\big]= J_{1k1,\epsilon}+J_{1k2,\epsilon}
\end{equation*}
for $k=1,\ldots,4$. Using Lemma \ref{calculus:lemma} to estimate the denominators in the integrals we obtain
\begin{align}
	\abs{J_{111,\epsilon}} &\leq C \norm{H}_{C^{\alpha}(\Omega)}  \int_{ R_{\Omega,\leq}} dy \ d\eta \bigg[ \frac{y(\abs{x-\eta}^{\alpha}+y^{\alpha})}{(x-\eta)^{2}+y^{2}}+\frac{y(\abs{\eta}^{\alpha}+y^{\alpha})}{\eta^{2}+y^{2}}\bigg] \leq C \norm{H}_{C^{\alpha}(\Omega)}\abs{x}^{\alpha},\label{bound:J111}\\
	\abs{J_{121,\epsilon}}+\abs{J_{131,\epsilon}} &\leq C \norm{H}_{C^{\alpha}(\Omega)} \int_{ R_{\Omega,\leq}}dy \ d\eta \bigg[\frac{y(\abs{x-\eta}^{\alpha}+y^{\alpha})}{\left((x-\eta)^{2}+y^{2}\right)^{3/2}}+\frac{y(\abs{\eta}^{\alpha}+y^{\alpha})}{\left(\eta^{2}+y^{2}\right)^{3/2}}\bigg] \leq  C\norm{H}_{C^{\alpha}(\Omega)} \abs{x}^{\alpha},  \label{bound:J121-J131} \\
	\abs{J_{141,\epsilon}} &\leq C \norm{\Lambda}_{C^{2}(\Omega)}\norm{H}_{L^{\infty}(\Omega)} \int_{ R_{\Omega,\leq}} dy \ d\eta \bigg[\frac{y^2}{\left((x-\eta)^{2}+y^{2}\right)^{3/2}}+\frac{y^2}{\left(\eta^{2}+y^{2}\right)^{3/2}}\bigg] \nonumber \\
	&\leq C \norm{H}_{L^{\infty}(\Omega)} \abs{x}\leq C \norm{H}_{L^{\infty}(\Omega)} \abs{x}^{\alpha} .\label{bound:J141}
\end{align}
In the region $R^{\epsilon}_{\Omega,>}=\{\displaystyle\max\{\abs{y},\abs{\eta}\}> 2\abs{x}\mbox{ and } \epsilon\leq y\leq L\}$, we rewrite the term $J_{1k2,\epsilon}$ for $k=1,\ldots,4$ in the following way
\begin{align*}
	J_{112,\epsilon}&= \int_{R^{\epsilon}_{\Omega,>}} dy \ d\eta \bigg[ \frac{iye^{i(x-\eta)-y}}{(1-e^{i(x-\eta)-y})(1-e^{i(x-\eta)-y-i\Lambda(\eta,y)})}-\frac{iye^{-i\eta-y}}{(1-e^{-i\eta-y})(1-e^{i(-\eta)-y-i\Lambda(\eta,y)})} \bigg]  \\
	&\quad \hspace{3cm} \times \left(H( \eta, y)-H(x,0)\right) \\
	&- \int_{R^{\epsilon}_{\Omega,>}} dy \ d\eta \frac{iye^{-i\eta-y}}{(1-e^{-i\eta-y})(1-e^{-i\eta-y-i\Lambda(\eta,y)})} \left(H(x, 0)-H(0,0)\right)=K_{1,\epsilon}+K_{2,\epsilon}, \\
	J_{122,\epsilon}&= \int_{R^{\epsilon}_{\Omega,>}} dy \ d\eta \bigg[\frac{iy(e^{i(x-\eta)-y})^2}{(1-e^{i(x-\eta)-y})^2(1-e^{i(x-\eta)-y-i\Lambda(\eta,y)})}-\frac{iy(e^{-i\eta-y})^2}{(1-e^{-i\eta-y})^2(1-e^{-i\eta-y-i\Lambda(\eta,y)})} \bigg] \\
	&\quad \hspace{3cm} \times\left(H( \eta, y)-H(x,0)\right)\\
	&-\int_{R^{\epsilon}_{\Omega,>}} dy \ d\eta\frac{iy(e^{-i\eta-y})^2}{(1-e^{-i\eta-y})^2(1-e^{-i\eta-y-i\Lambda(\eta,y)})} \left(H( x, 0)-H(0,0)\right) =K_{3,\epsilon}+K_{4,\epsilon}, \\
	J_{132,\epsilon}&= \int_{R^{\epsilon}_{\Omega,>}} dy \ d\eta \bigg[\frac{iye^{i(x-\eta)-y}e^{i(x-\eta)-y-i\Lambda(\eta,y)}}{(1-e^{i(x-\eta)-y})(1-e^{i(x-\eta)-y-i\Lambda(\eta,y)})^2}-\frac{iye^{-i\eta-y}e^{-i\eta-y-i\Lambda(\eta,y)}}{(1-e^{-i\eta-y})(1-e^{-i\eta-y-i\Lambda(\eta,y)})^2} \bigg] \\
	&\quad \hspace{3cm} \times \left(H( \eta, y)-H(x,0)\right) \\
	&-\int_{R^{\epsilon}_{\Omega,>}} dy \ d\eta\frac{iye^{-i\eta-y}e^{i(-\eta)-y-i\Lambda(\eta,y)}}{(1-e^{-i\eta-y})(1-e^{-i\eta-y-i\Lambda(\eta,y)})^2} \left(H( x, 0)-H(0,0)\right)=K_{5,\epsilon}+K_{6,\epsilon},\\
	J_{142,\epsilon}&=-\int_{R^{\epsilon}_{\Omega,>}} dy \ d\eta \bigg[\frac{iye^{i(x-\eta)-y}e^{i(x-\eta)-y-i\Lambda(\eta,y)}\p_{\eta}\Lambda(\eta,y)}{(1-e^{i(x-\eta)-y})(1-e^{i(x-\eta)-y-i\Lambda(\eta,y)})^2}-\frac{iye^{-i\eta-y}e^{-i\eta-y-i\Lambda(\eta,y)}\p_{\eta}\Lambda(\eta,y)}{(1-e^{-i\eta-y})(1-e^{-i\eta-y-i\Lambda(\eta,y)})^2} \bigg]  \\
	&\quad \hspace{4cm}  \times H(x,0) \\
	&-\int_{R^{\epsilon}_{\Omega,>}} dy \ d\eta\frac{iye^{-i\eta-y}e^{-i\eta-y-i\Lambda(\eta,y)}\p_{\eta}\Lambda(\eta,y)}{(1-e^{-i\eta-y})(1-e^{-i\eta-y-i\Lambda(\eta,y)})^2} \left(H( x, 0)-H(0,0)\right)=K_{7,\epsilon}+K_{8,\epsilon}. 
\end{align*}
The integrands of the terms $K_{1,\epsilon},K_{3,\epsilon},K_{5,\epsilon}$ and $K_{7,\epsilon}$ can be bounded in the region $R^{\epsilon}_{\Omega,>}$ using the mean value theorem as well as Lemma \ref{calculus:lemma} and the H\"older regularity of $H$. Thus
\begin{align}
	\abs{K_{1,\epsilon}} &\leq C \norm{H}_{C^{\alpha}(\Omega)} \int_{R_{\Omega,>}} dy \ d\eta \abs{x}\frac{y(\abs{x-\eta}^{\alpha}+y^{\alpha})}{\left((x-\eta)^{2}+y^{2}\right)^{3/2}}\leq \norm{H}_{C^{\alpha}(\Omega)}\abs{x}^{\alpha} \label{bound:K1}\\
	\abs{K_{3,\epsilon}}+\abs{K_{5,\epsilon}}&\leq  C \norm{H}_{C^{\alpha}(\Omega)} \int_{R_{\Omega,>}} dy \ d\eta \abs{x}\frac{y(\abs{x-\eta}^{\alpha}+y^{\alpha})}{\left((x-\eta)^{2}+y^{2}\right)^{2}}\leq \norm{H}_{C^{\alpha}(\Omega)}\abs{x}^{\alpha}\label{bound:K3-K5}\\
	\abs{K_{7,\epsilon}} &\leq C \norm{\Lambda}_{C^{2}(\Omega)}\norm{H}_{L^{\infty}(\Omega)}  \int_{R_{\Omega,>}} dy \ d\eta \abs{x}\frac{y^2}{\left((x-\eta)^{2}+y^{2}\right)^{2}} \leq C \norm{H}_{L^{\infty}(\Omega)} \abs{x}^{\alpha}.\label{bound:K7}
\end{align}
Furthermore, a direct computation using Lemma \ref{calculus:lemma} shows that 
\begin{align}
	\abs{K_{2,\epsilon}}&\leq C \norm{H}_{C^\alpha(\Omega)}\abs{x}^{\alpha} \int_{R_{\Omega,>}} dy \ d\eta \frac{y}{(\eta^2+y^2)}\leq C \norm{H}_{C^\alpha(\Omega)}\abs{x}^{\alpha}. \label{bound:K2}
\end{align}
Similarly, using the fact that $\p_{\eta}\Lambda(\eta,0)=0$ (cf. Assumption \ref{assumption:Lambda}) we have that 
$$\abs{\p_{\eta}\Lambda(\eta,y)}\leq y\norm{\Lambda}_{C^2(\Omega)},$$
and hence
\begin{align}
	\abs{K_{8,\epsilon}}&\leq C \abs{x}^{\alpha} \norm{\Lambda}_{C^2(\Omega)}\norm{H}_{C^\alpha(\Omega)} \int_{R_{\Omega,>}} dy \ d\eta  \frac{y^2}{(\eta^2+y^2)^{3/2}}\leq C\abs{x}^{\alpha} \norm{H}_{C^\alpha(\Omega)}.\label{bound:K8}
\end{align}
To conclude the proof of the $C^\alpha$ bound \eqref{cota:singular:1}, it only remains to estimate the more singular terms, namely $K_{4,\epsilon}$ and $K_{6,\epsilon}$. In these terms we can not just estimate the integrands by the absolute value because this will result on the onset of a logarithmically divergent term. To that purpose, we further simplify the integrand until arriving to an expression in which the integral of the most singular term in the $y$ variable can be explicitly computed. First, we decompose 
\begin{equation}\label{decomposition:Lambda}
	\Lambda(\eta,y)=yA(\eta)+[\Lambda(\eta,y)-A(\eta)y]
\end{equation}
where $A(\eta)=\partial_{y}\Lambda(\eta,0)$. Notice that
\begin{equation}\label{bound:square}
	\abs{\Lambda(\eta,y)-A(\eta)y}\leq C y^{2}.
\end{equation}
Using this decomposition, $K_{4,\epsilon}$ can be written as 
\begin{align}
	K_{4,\epsilon}&=  -\left(H(x, 0)-H(0,0)\right) \int_{R^{\epsilon}_{\Omega,>}} dy \ d\eta \frac{iy(e^{-i\eta-y})^2}{(1-e^{-i\eta-y})^2(1-e^{-i\eta-y-i\Lambda(\eta,y)})} \label{K:4:eps:firstpart} \\
	&= -\left(H(x, 0)-H(0,0)\right) \int_{R^{\epsilon}_{\Omega,>}} dy \ d\eta\frac{iy(e^{-i\eta-y})^2}{(1-e^{-i\eta-y})^2(1-e^{-i\eta-y-iA(\eta)y}\left(1+r_{1}(\eta,y)\right))}  \nonumber
\end{align}
where the remainder term $r_{1}(\eta,y)$ can be bounded using \eqref{bound:square} by $\abs{r_1(\eta,y)}\leq Cy^{2}$. Using Taylor expansion we obtain that
\begin{align*}
	K_{4,\epsilon}&= -\left(H(x, 0)-H(0,0)\right) \int_{R^{\epsilon}_{\Omega,>}} dy \ d\eta \frac{iy(e^{-i\eta-y})^2}{(1-e^{-i\eta-y})^2(1-e^{-i\eta-y-iA(\eta)y})}\left(1+r_{2}(\eta,y)\right)\\
	&= -\left(H(x, 0)-H(0,0)\right) \left(\int_{R^{\epsilon}_{\Omega,>}} dy \ d\eta \big[\ldots \big]+ \int_{R^{\epsilon}_{\Omega,>}} dy \ d\eta \big[\ldots \big] r_{2}(\eta,y)\right)=K_{41,\epsilon}+K_{42,\epsilon},
\end{align*}
where the new remainder $r_2(\eta,y)$ is bounded by $\abs{r_2(\eta,y)}\leq C \abs{y}.$ The integrand in $K_{42,\epsilon}$ is integrable and can be bounded by using Lemma \ref{calculus:lemma} for $\Lambda(\eta,y)=yA(\eta)$ as
\begin{equation}\label{estimate:K42:eps}
	\abs{K_{42,\epsilon}}\leq C \abs{x}^\alpha \norm{H}_{C^{\alpha}(\Omega)}  \int_{R_{\Omega,>}} dy \ d\eta \frac{y^2}{\left(\eta^2+y^{2}\right)^{3/2}}\leq C \abs{x}^\alpha \norm{H}_{C^{\alpha}(\Omega)}. 
\end{equation}
To most delicate term is $K_{41,\epsilon}$. Using again Taylor expansion we find that
\begin{align}
	K_{41,\epsilon}&= -\left(H(x, 0)-H(0,0)\right) \int_{R^{\epsilon}_{\Omega,>}} dy \ d\eta \frac{iy}{(-i\eta-y)^2(-i\eta-y-iA(\eta)y)} (1+r_{3}(\eta,y)) \nonumber \\
	&=K_{411,\epsilon}+K_{412,\epsilon} \label{estimate:K41:eps}
\end{align}
with  $\abs{r_{3}(\eta,y)}\leq C (\eta-y)^2$. Then
\begin{equation}\label{bound:K412}
	\abs{K_{412,\epsilon}}\leq C \abs{x}^\alpha \norm{H}_{C^{\alpha}(\Omega)}. 
\end{equation}
To estimate the remaining term $K_{411,\epsilon}$, we recall that $A(\eta)=\p_{y}\Lambda(\eta,0)$ and that by Assumption \ref{assumption:Lambda}, it follows that $\norm{\Lambda}_{C^{2,\alpha}(\Omega)}\leq \delta_{0}$ for $\delta_{0}\in (0,\frac{1}{2})$. Hence we can write 
$
A(\eta)=A(0)+[A(\eta)-A(0)]
$
where 
\begin{equation}\label{A:bound}
	\abs{A(\eta)-A(0)}\leq \delta_{0} \abs{\eta}^{\alpha}.
\end{equation} 
Therefore, 
\begin{align*}
	K_{411,\epsilon}= -\left(H(x, 0)-H(0,0)\right) \int_{R^{\epsilon}_{\Omega,>}} dy \ d\eta \frac{iy}{(-i\eta-y)^2(-i\eta-y-iA(0)y)}+ K_{4112,\epsilon}
\end{align*}
with $\abs{K_{4112,\epsilon}}\leq C \abs{x}^\alpha \norm{H}_{C^{\alpha}(\Omega)}$. Doing the rescaling variables $\eta=y\zeta$ and recalling that 
$$R^{\epsilon}_{\Omega,>}=\{(\eta,y)\in \Omega:\displaystyle\max\{y,\abs{\eta}\}> 2\abs{x} \mbox{ and } \epsilon\leq y\leq L\},
$$ 
we infer that the above integral can be expressed as 
\begin{align*}
	\int_{R^{\epsilon}_{\Omega,>}} dy \ d\eta \frac{iy}{(-i\eta-y)^2(-i\eta-y-iA(0)y)}&=\int_{\epsilon}^{L} \frac{dy}{y}\int_{\Sigma(x,y)} \frac{d\zeta} {(i\zeta+1)^2(i\zeta+1+iA(0))}
\end{align*}
with 
\begin{equation}\label{sigma:def}
	\Sigma(x,y)=\{\zeta\in \mathbb{R}: \max \{1,\abs{\zeta}\}> \frac{2\abs{x}}{y}, \ \abs{\zeta}\leq \frac{\pi}{y} \}.
\end{equation}
In order to estimate this integral, we consider two different cases, the case for $L\geq y\geq2\abs{x}$ and $\epsilon\leq y<2\abs{x}$, namely
\begin{align}\label{integral:L}
	\int_{\epsilon}^{L} \frac{dy}{y}\int_{\max\{1,\abs{\zeta}\}> \frac{2\abs{x}}{y}} \frac{d\zeta} {(i\zeta+1)^2(i\zeta+1+iA(0))}&=\int_{2\abs{x}}^{L} \frac{dy}{y}\int_{ \Sigma(x,y)} \ldots+ \int_{\epsilon}^{2\abs{x}} \frac{dy}{y}\int_{ \Sigma(x,y)} \ldots  \nonumber \\ 
	&=L_{1}+L_{2}.
\end{align}
In the case of the integral $L_1$, this is for $L\geq y\geq2\abs{x}$,  the domain of integration $\Sigma$ reduces to $\Sigma(x,y)=\{\zeta\in \mathbb{R}: \ \abs{\zeta}\leq \frac{\pi}{y} \}$. Therefore we can extend the domain of integration of $\zeta$ to the whole space $\RR$ just adding a remainder term that can be estimated by $C\abs{y}^2$ for small $y$. This follows from the fact that the integrand in $\zeta$ can be estimated by $\frac{C}{\abs{\zeta}^3}$ for $\abs{\zeta}\geq 1$.
Thus, we have that
\begin{equation}\label{estimate:L1:nocal:1}
	L_{1}=\int_{2\abs{x}}^{L} \frac{dy}{y}\int_{\mathbb{R}} \frac{d\zeta}{(i\zeta+1)^2(i\zeta+1+iA(0))}+L_{12}
\end{equation}
with $\abs{L_{12}}\leq  C$. To deal with the first integral in \eqref{estimate:L1:nocal:1} we use contour integrating using residues yields
\begin{equation}\label{estimate:L11}
	L_{11}=\int_{2\abs{x}}^{L} \frac{dy}{y}\int_{\mathbb{R}} \frac{d\zeta}{(i\zeta+1)^2(i\zeta+1+iA(0))}=0
\end{equation}
using the fact that the only poles are in $\zeta=-i, \zeta=-i+A(0)$ and $A(0)\in \RR$.
We now estimate $L_2$. Since $\epsilon\leq y<2\abs{x}$, we have that $ \frac{2\abs{x}}{y}>1$ and therefore
$$ \Big\{ \max \{1,\abs{\zeta}\}> \frac{2\abs{x}}{y} \Big\} \subset \Big\{ \abs{\zeta}> \frac{2\abs{x}}{y}\Big\}.$$
Hence applying Fubini's theorem we obtain
\begin{align}
	\abs{L_{2}}\leq \int_{0}^{2\abs{x}} \frac{dy}{y}\int_{\abs{\zeta}>\frac{2\abs{x}}{y}}\abs{ \frac{1}{(i\zeta+1)^2(i\zeta+1+iA(0))}}d\zeta&\leq C
	\int_{\abs{\zeta}\geq 1}  \frac{d\zeta}{(1+\abs{\zeta}^3)} \int_{\frac{2\abs{x}}{\abs{\zeta}}}^{2\abs{x}} \frac{dy}{y} \nonumber  \\
	&=C
	\int_{\abs{\zeta}\geq 1}  \frac{d\zeta}{(1+\abs{\zeta}^3)} \log(\abs{\zeta}) 
	\nonumber \\
	&\leq C.  \label{estimate:L2:nocal}
\end{align}
Combining \eqref{estimate:L1:nocal:1}-\eqref{estimate:L2:nocal} we have shown that
\begin{equation}\label{bound:K411}
	\abs{K_{411,\epsilon}}\leq C \abs{x}^\alpha \norm{H}_{C^{\alpha}(\Omega)}
\end{equation}
as desired.  Collecting \eqref{estimate:K42:eps}, \eqref{bound:K412} and \eqref{bound:K411} we find that
\begin{equation}\label{K4:eps:final}
	\abs{K_{4,\epsilon}}\leq C \abs{x}^\alpha \norm{H}_{C^{\alpha}(\Omega)}.
\end{equation}
We can estimate the term $K_{6,\epsilon}$ in a similar manner. We recall that 
\begin{align*}
	K_{6,\epsilon}&= -\left(H( x, 0)-H(0,0)\right)\int_{R_{\Omega,>}} dy \ d\eta\frac{iye^{-i\eta-y}e^{-i\eta-y-i\Lambda(\eta,y)}}{(1-e^{i\eta-y})(1-e^{-i\eta-y-i\Lambda(\eta,y)})^2}.
\end{align*}
Indeed, using again the decomposition \eqref{decomposition:Lambda} and the estimate \eqref{bound:square} we find that
\begin{align*}
	K_{6,\epsilon}&= -\left(H( x, 0)-H(0,0)\right)\int_{R^{\epsilon}_{\Omega,>}} dy \ d\eta\frac{iy}{(1-e^{i\eta-y})(1-e^{-i\eta-y-iA(\eta)y})^2}(1+r_{5}(\eta,y)) \\
	&=-\left(H(x, 0)-H(0,0)\right) \left(\int_{R^{\epsilon}_{\Omega,>}} dy \ d\eta \big[\ldots \big]+ \int_{R^{\epsilon}_{\Omega,>}} dy \ d\eta \big[\ldots \big] r_{5}(\eta,y)\right)=K_{61,\epsilon}+K_{62,\epsilon}
\end{align*}
with $\abs{r_{5}(\eta,y)}\leq C \abs{y}$. Therefore, $K_{62,\epsilon}$ can be easily bounded using Lemma \ref{calculus:lemma} for $\Lambda(\eta,y)=yA(\eta)$ by
\begin{equation}\label{estimate:K62:eps}
	\abs{K_{62,\epsilon}}\leq C \abs{x}^{\alpha}\norm{H}_{C^{\alpha}(\Omega)}\int_{R_{\Omega,>}}  dy \ d\eta \frac{y^2}{\left(\eta^2+y^{2}\right)^{3/2}}\leq C \abs{x}^\alpha \norm{H}_{C^{\alpha}}
\end{equation}
To deal with $K_{61,\epsilon}$, we argue as in the estimate of $K_{411,\epsilon}$. Then
\begin{align}\label{estimate:K61:eps}
	K_{61,\epsilon}= -\left(H(x, 0)-H(0,0)\right) \int_{R^{\epsilon}_{\Omega,>}} dy \ d\eta \frac{iy}{(-i\eta-y)(-i\eta-y-iA(0)y)^2}+ K_{612,\epsilon}
\end{align}
where $\abs{K_{612,\epsilon}}\leq C \abs{x}^\alpha \norm{H}_{C^{\alpha}(\Omega)}$. To estimate the remaining term, we perform the change of variables $\eta=y\zeta$ and readily check that the resulting integral 
\begin{align*}
	K_{611,\epsilon}&= -\left(H(x, 0)-H(0,0)\right) \int_{R_{\Omega,>}} dy \ d\eta \frac{iy}{(-i\eta-y)(-i\eta-y-iA(0)y)^{2}} \\
	&=-\left(H(x, 0)-H(0,0)\right)\int_{\epsilon}^{L} \frac{dy}{y}\int_{\Sigma(x,y)} \frac{d\zeta}{(i\zeta+1)(i\zeta+1+iA(0))^2}
\end{align*}
where $\Sigma(x,y)$ is defined in \eqref{sigma:def} can be bounded similarly as we estimated integral \eqref{integral:L}, namely,
\begin{equation}\label{bound:K611}
	\abs{K_{611,\epsilon}}\leq  C \abs{x}^\alpha \norm{H}_{C^{\alpha}(\Omega)}.
\end{equation}
Hence, combining \eqref{estimate:K62:eps}, \eqref{estimate:K61:eps} and \eqref{bound:K611} we conclude
\begin{equation}\label{K6:eps:final}
	\abs{K_{6,\epsilon}}\leq  C \abs{x}^\alpha \norm{H}_{C^{\alpha}(\Omega)}.
\end{equation}

Therefore, by means of \eqref{int:sin:1} and \eqref{Xi:epsilon:difference} and collecting estimates \eqref{bound:J111}-\eqref{bound:J141}, \eqref{bound:K1}-\eqref{bound:K7}, \eqref{bound:K2},\eqref{bound:K8} and bounds \eqref{K4:eps:final}, \eqref{K6:eps:final} we have shown that
\begin{equation}
	\abs{\Xi(x)-\Xi(0)}\leq  \displaystyle\lim_{\epsilon\to 0^{+}}\abs{\Xi_{1,\epsilon}(x)-\Xi_{1,\epsilon}(0)}+\displaystyle\lim_{\epsilon\to 0^{+}}\abs{\Xi_{2,\epsilon}(x)-\Xi_{2,\epsilon}(0)}\leq C \abs{x}^\alpha \norm{H}_{C^{\alpha}(\Omega)}
\end{equation}
which shows the desired $\alpha$-H\"older semi-norm estimate. The later estimate combined with the pointwise bound \eqref{Linfty:Xi:bound} yields the estimate for the $C^{\alpha}(\SS^1)$ norm. This concludes the proof of the proposition.
\end{proof}

\subsection{\texorpdfstring{$C^{1,\alpha}$}{Lg} H\"older estimate}
We now derived the following H\"older estimates for the derivative of $\Xi$.
\begin{proposition}[$C^{1,\alpha}$ estimate]\label{Holder:singular:integral:1alpha}
Let $H(\eta, y)\in C^{1,\alpha}(\Omega)$ and suppose that Assumption \ref{assumption:Lambda} holds. For every $x\in \SS^1$, define the function $\Xi(x)$ as in \eqref{int:sin:1}. Then we have that
\begin{align}
	\norm{\Xi}_{C^{1,\alpha}(\mathbb{S}^{1})}&\leq C \norm{H}_{C^{1,\alpha}(\Omega)} \label{cota:singular:2}
\end{align}
with $C>0$.
\end{proposition}

\begin{proof}[Proof of Proposition \ref{Holder:singular:integral:1alpha}]
By Proposition \ref{Holder:singular:integral:alpha}, it is clear that the function $\Xi(x)$ defined in \eqref{int:sin:1} exists and it is well defined. Moreover, we also showed in the previous lemma that the pointwise bound 
\begin{equation}\label{pointwise:bound:version2}
	\abs{\Xi(x)}\leq C \norm{H}_{C^{\alpha}(\Omega)}, \mbox{ for } x\in \SS^1
\end{equation}
holds. We will see at the end of the proof that estimate \eqref{cota:singular:2} would be a consequence of \eqref{pointwise:bound:version2} and the following bound
\begin{equation*}
	\abs{\p_{x}\Xi_{\epsilon}(x_1)-\p_{x}\Xi_{\epsilon}(x_2)}\leq C \abs{x_1-x_2}^{\alpha} \norm{H}_{C^{1,\alpha}(\Omega)}, \mbox{ for } x_{1},x_{2}\in \SS^1, \epsilon>0
\end{equation*}
where $\Xi_{\epsilon}$ is defined in \eqref{int:sin:1:eps}.
As before, due to the translation invariance of the estimate it suffices to check without loss of generality that the bound holds for $x_2=0$ and $x_{1}=x$, namely
\begin{equation}\label{principal:estimate:1alpha:eps}
	\abs{\p_{x}\Xi_{\epsilon}(x)-\p_{x}\Xi_{\epsilon}(0)}\leq C\abs{x}^{\alpha}\norm{H}_{C^{1,\alpha}(\Omega)}.
\end{equation}

To that purpose, using the definition of $\mathcal{A}(x,y,\eta)$ in \eqref{function:A:1} we have that
\begin{align}
	\p^2_{x}\mathcal{A}(x,y,\eta) &= -\p_{x}\p_{\eta}\mathcal{A}(x,y,\eta)- \p_{x}\mathcal{R}(x,y,\eta) \label{can:expl2}
\end{align}
where $\mathcal{R}(x,y,\eta)$ is given in \eqref{rest}. Therefore, recalling the definitions \eqref{int:sin:1:eps}, \eqref{Xi:1:eps},\eqref{Xi:2:eps}, using \eqref{can:expl2} and integrating by parts we obtain
$$\p_{x}\Xi_{\epsilon}(x)=\p_{x}\Xi_{1,\epsilon}(x)+\p_{x}\Xi_{2,\epsilon}(x)$$ where
\begin{align}
	\p_{x}\Xi_{1,\epsilon}(x)&= \int_{\epsilon}^{L}dy\int_{\mathbb{S}^{1}}d\eta \  \p_{x} \mathcal{A}(x,y,\eta) \p_{\eta}H( \eta, y)   \label{Xi:1alpha:eps} \\
	\p_{x}\Xi_{2,\epsilon}(x)&= - \int_{\epsilon}^{L}dy\int_{\mathbb{S}^{1}}d\eta \ \p_{x} \mathcal{R}(x,y,\eta) \ H( \eta, y). \label{Xi:2alpha:eps}
\end{align}
Notice that the term \eqref{Xi:1alpha:eps} has exactly the same form as \eqref{int:sin:1:eps} with $\p_{\eta}H(\eta,y)$ replaced by $H(\eta,y)$. As a consequence, mimicking the estimate \eqref{Linfty:Xi:epsilon:bound} in Lemma \ref{Holder:singular:integral:alpha} we have that 
\begin{align}
	\abs{\p_{x}\Xi_{1,\epsilon}(x)}&\leq C \norm{H}_{C^{1,\alpha}(\Omega)},   \mbox{ for } x\in\SS^1  \mbox{ and } \epsilon>0. \label{Linfty:der:Xi1:epsilon:bound} 
\end{align}
Moreover, arguing as in the proof of the previous lemma using dominated convergence it follows that
\begin{equation}\label{existence:Xi1:eps:limit}
	\lim_{\epsilon\to 0}\p_{x}\Xi_{1,\epsilon}:= Q_{1}(x)
\end{equation}
exists. A direct application of estimate \eqref{cota:singular:1} with $H(\eta,y)$ replaced by  $\p
_{\eta}H(\eta,y)$ yields 

\begin{equation}\label{Xi1:estimate:final:1alpha}
	\abs{Q_{1}(x)-Q_{1}(0)}\leq C\abs{x}^{\alpha}\norm{H}_{C^{1,\alpha}(\Omega)}, \mbox{ for } x\in\SS^1.
\end{equation}
We next show that limit in \eqref{Xi:2alpha:eps} as $\epsilon$ tends to zero exists. To that purpose we write
\begin{equation}\label{rest:change}
	\p_{x}\mathcal{R}(x,y,\eta)=-\p_{\eta}\mathcal{R}_{0}+\mathcal{R}_{1}+\mathcal{R}_{2}+\mathcal{R}_{3}
\end{equation}
where 
\begin{align}
	\mathcal{R}_{0}&=\p_{\eta}\Lambda(\eta,y)\frac{iye^{i(x-\eta)-y}e^{i(x-\eta)-y-i\Lambda(\eta,y)}}{(1-e^{i(x-\eta)-y})(1-e^{i(x-\eta)-y-i\Lambda(\eta,y)})^2}, \label{R0:last:align}\\
	\mathcal{R}_{1}&=\p^2_{\eta}\Lambda(\eta,y)\frac{iye^{i(x-\eta)-y}e^{i(x-\eta)-y-i\Lambda(\eta,y)}}{(1-e^{i(x-\eta)-y})(1-e^{i(x-\eta)-y-i\Lambda(\eta,y)})^2},\\
	\mathcal{R}_{2}&=(\p_{\eta}\Lambda(\eta,y))^2\frac{iye^{i(x-\eta)-y}e^{i(x-\eta)-y-i\Lambda(\eta,y)}}{(1-e^{i(x-\eta)-y})(1-e^{i(x-\eta)-y-i\Lambda(\eta,y)})^2},\\
	\mathcal{R}_{3}&=\p_{\eta}\Lambda(\eta,y)\frac{ye^{i(x-\eta)-y}(e^{i(x-\eta)-y-i\Lambda(\eta,y)})^2}{(1-e^{i(x-\eta)-y})(1-e^{i(x-\eta)-y-i\Lambda(\eta,y)})^3}. \label{R3:last:align}
\end{align}
Plugging \eqref{rest:change} in \eqref{Xi:2alpha:eps}, we infer that
\begin{equation}\label{decomposition:xi:2}
	\p_{x}\Xi_{2,\epsilon}(x)= \p_{x}\Xi_{20,\epsilon}(x) + \p_{x}\Xi_{21,\epsilon}(x)+ \p_{x}\Xi_{22,\epsilon}(x)+ \p_{x}\Xi_{23,\epsilon}(x)
\end{equation}
with
\begin{align*}
	\p_{x}\Xi_{20,\epsilon}(x)&=  \int_{\epsilon}^{L}dy\int_{\mathbb{S}^{1}}d\eta \  \mathcal{R}_{0}(x,y,\eta) \left(\p_{\eta}H( \eta, y)-\p_{\eta}H( x, 0)\right) \\
	\p_{x}\Xi_{2j,\epsilon}(x)&= - \int_{\epsilon}^{L}dy\int_{\mathbb{S}^{1}}d\eta \  \mathcal{R}_{j}(x,y,\eta) H( \eta, y) , \mbox{ for } j=1,2,3,
\end{align*}
where in the first term we have applied integration by parts. Identifying $\SS^1$ with the symmetric interval $I_{x}=[x-\pi,x+\pi]$ for $x\in [-\pi,\pi]$, using bound \eqref{est:calcu} and the H\"older regularity for $H$ we obtain
\begin{align*}
	\abs{ \p_{x}\Xi_{20,\epsilon}(x)}&\leq C \norm{\Lambda}_{C^2(\Omega)} \norm{H}_{C^{1,\alpha}(\Omega)} \int_{\epsilon}^{L}dy\int_{I_{x}}d\eta \frac{y^{2}(\abs{x-\eta}^{\alpha}+y^{\alpha})}{((x-\eta)^{2}+y^{2})^{3/2}}\leq C\norm{H}_{C^{1,\alpha}(\Omega)}, \\
	\abs{ \p_{x}\Xi_{21,\epsilon}(x)}&\leq C \norm{\Lambda}_{C^{2,\alpha}(\Omega)} \norm{H}_{L^{\infty}(\Omega)} \int_{\epsilon}^{L}dy\int_{I_{x}}d\eta \frac{y^{1+\alpha}}{((x-\eta)^{2}+y^{2})^{3/2}}\leq C\norm{H}_{L^{\infty}(\Omega)}, \\
	\abs{ \p_{x}\Xi_{22,\epsilon}(x)}&\leq C \norm{\Lambda}^2_{C^{2}(\Omega)}\norm{H}_{L^{\infty}(\Omega)} \int_{\epsilon}^{L}dy\int_{I_{x}}d\eta \frac{y^{3}}{((x-\eta)^{2}+y^{2})^{3/2}}\leq C\norm{H}_{L^{\infty}(\Omega)},
\end{align*}
for $x\in \SS^1$. Similarly arguing as in the proof of Lemma \ref{Holder:singular:integral:alpha}, by dominated convergence we have that the limits 
\begin{equation}\label{existence:Xi2j:eps:limit} 
	\displaystyle\lim_{\epsilon\to 0^{+}}\p_{x}\Xi_{2j,\epsilon}(x):=Q_{2j}(x), \mbox{ for } j=0,1,2
\end{equation}  
exist. The most singular term in \eqref{decomposition:xi:2} is $\p_{x}\Xi_{23,\epsilon}$. This term can be written as 
\begin{align}
	\p_{x}\Xi_{23,\epsilon}(x)&=- \int_{\epsilon}^{L}dy\int_{\mathbb{S}^{1}}d\eta \mathcal{R}_{3}(x,y,\eta) \left(H( \eta, y)-H(x,0) \right) -H(x,0)\int_{\epsilon}^{L}dy\int_{\mathbb{S}^{1}}d\eta  \ \mathcal{R}_{3}(x,y,\eta) \nonumber  \\
	&= \p_{x}\Xi_{231,\epsilon}(x)+\p_{x}\Xi_{232,\epsilon}(x). \label{decomposition:xi24}
\end{align}
Identifying again $\SS^1$ with the interval $I_{x}=[x-\pi,x+\pi]$ for $x\in [-\pi,\pi]$, using  the bound \eqref{est:calcu} and the H\"older regularity for $H$ we obtain
\begin{align}\label{Xi231:eps:bound}
	\abs{\p_{x}\Xi_{231,\epsilon}(x)}&\leq C \norm{\Lambda}_{C^{2}(\Omega)}\norm{H}_{C^{\alpha}(\Omega)}\int_{\epsilon}^{L}dy\int_{I_x}d\eta \frac{y^{2}(\abs{x-\eta}^{\alpha}+y^{\alpha})}{((x-\eta)^{2}+y^{2})^{2}} \leq C\norm{H}_{C^{\alpha}(\Omega)}, \mbox{ for } x\in\SS^1.
\end{align}
Using the decomposition
\begin{equation}\label{decomp:Lambda:2}
	\p_{ \eta }\Lambda(\eta,y)=y\p^2_{y \eta }\Lambda(\eta,0)+ \big[ \p_{\eta}\Lambda(\eta,y)-y\p^2_{y \eta }\Lambda(\eta,0) \big]
\end{equation}
we infer that 
\begin{align*}
	&\p_{x}\Xi_{232,\epsilon}(x)=H(x,0)\int_{\epsilon}^{L}dy\int_{\mathbb{S}^{1}}d\eta \ \p^2_{y \eta }\Lambda(\eta,0)\frac{y^2e^{i(x-\eta)-y}(e^{i(x-\eta)-y-i\Lambda(\eta,y)})^2}{(1-e^{i(x-\eta)-y})(1-e^{i(x-\eta)-y-i\Lambda(\eta,y)})^3} \\
	&\quad +H(x,0)\int_{\epsilon}^{L}dy\int_{\mathbb{S}^{1}}d\eta  \ y  \big[ \p_{\eta}\Lambda(\eta,y)-y\p^2_{y \eta }\Lambda(\eta,0) \big]\frac{e^{i(x-\eta)-y}(e^{i(x-\eta)-y-i\Lambda(\eta,y)})^2}{(1-e^{i(x-\eta)-y})(1-e^{i(x-\eta)-y-i\Lambda(\eta,y)})^3}\\
	&\quad \quad \quad=N_{1,\epsilon}+N_{2,\epsilon}
\end{align*}
Using that
\begin{equation}\label{bound:Lambda:2}
	\abs{ \big[\p_{\eta}\Lambda(\eta,y)-y\p^2_{y \eta }\Lambda(\eta,0) \big]} \leq C \abs{y}^{1+\alpha} \norm{\Lambda}_{C^{2,\alpha}(\Omega)}\leq C \abs{y}^{1+\alpha},
\end{equation}
identifying once again $\SS^1$ with $I_{x}=[x-\pi,x+\pi]$ for $x\in [-\pi,\pi]$ and invoking Lemma \ref{calculus:lemma} we find that
\begin{equation}\label{bound:N_2}
	\abs{N_{2,\epsilon}}\leq C\norm{H}_{L^\infty} \int_{\epsilon}^{L}dy \int_{I_{x}}d\eta  \frac{y^{2+\alpha}}{((x-\eta)^2+y^2)^2}\leq C\norm{H}_{L^\infty}.
\end{equation}
On the other hand, using Taylor's expansion and decomposition \eqref{decomposition:Lambda},  we have that
\begin{align}\label{N1eps}
	N_{1,\epsilon}(x)=H(x,0) \int_{\epsilon}^{L}dy\int_{I_{x}}d\eta  \frac{y^2 \p_{\eta }A(\eta)}{(i(x-\eta)-y)(i(x-\eta)-y-iyA(\eta))^3}+N_{12,\epsilon}(x)
\end{align}
with $A(\eta)=\p_{y}\Lambda(\eta,0)$. The remainder $N_{12,\epsilon}$ has an integrable singularity and can be easily estimated using \eqref{bound:square}, namely
\begin{equation}\label{bound:N21:trivial}
	\abs{N_{12,\epsilon}(x)}\leq C \norm{H}_{L^\infty}, \mbox{ for } x\in \SS^1.
\end{equation}

To estimate the first term on the right hand side of \eqref{N1eps}  we further use the decomposition 
\begin{equation}
	\p_{\eta}A(\eta)=\p_{\eta}A(x)+[\p_{\eta}A(\eta)-\p_{\eta}A(x)], \mbox{ where } \abs{\p_{\eta}A(\eta)-\p_{\eta}A(x)}\leq \delta_{0} \abs{x-\eta}^{\alpha},
\end{equation} 
to write
\begin{align}\label{exp:N11}
	N_{11,\epsilon}(x)&=H(x,0)  \p_{\eta }A(x)\int_{\epsilon}^{L}dy\int_{I_{x}}d\eta  \frac{y^2}{(i(x-\eta)-y)(i(x-\eta)-y-iyA(x))^3} \\
	&\quad +N_{112,\epsilon}(x) \nonumber
\end{align}
where $\abs{N_{112,\epsilon}}\leq C \norm{H}_{L^\infty}.$ After the change of variables $(x-\eta)=-y\zeta$ we find that the first term on the right hand side in \eqref{exp:N11} is given by 
\begin{align}
	N_{11,\epsilon}(x)=H(x,0)  \p_{\eta }A(x)\int_{\epsilon}^{L}\frac{dy}{y}\int_{-\frac{\pi}{y}}^{\frac{\pi}{y}} \frac{d\zeta}{(i\zeta+1)(i\zeta+1+iA(x))^3}. \label{treat:N11:eps}
\end{align}
Extending the value $\zeta$ to the whole space $\RR$  we have that
\begin{equation}\label{estimate:N11}
	N_{11,\epsilon}=H(x,0) \p_{\eta }A(x) \int_{\epsilon}^{L}\frac{dy}{y}\int_{\RR} \frac{d\zeta}{(i\zeta+1)(i\zeta+1+iA(x))^3}+N_{111,\epsilon}
\end{equation}
where the remaining term $\abs{N_{111,\epsilon}}\leq  C$. Since the only poles are in $\zeta=-i, \zeta=-i+A(x)$ and $A(x)\in \RR$, computing the integral using residues yields
\begin{equation}\label{estimate:N111}
	\int_{\epsilon}^{L}\frac{dy}{y}\int_{\RR} \frac{d\zeta}{(i\zeta+1)(i\zeta+1+iA(x))^3}=0.
\end{equation}
Hence, we have that bounds \eqref{bound:N_2},\eqref{bound:N21:trivial},\eqref{exp:N11} and \eqref{estimate:N111} yield
\begin{equation}\label{Xi:232:final}
	\abs{\p_{x}\Xi_{232,\epsilon}(x)}\leq \abs{N_{1,\epsilon}+N_{2,\epsilon}}\leq C \norm{H}_{L^{\infty}(\Omega)}, \mbox{ for } x\in\SS^1.
\end{equation}
Therefore, collecting the previous estimates \eqref{Xi231:eps:bound} and \eqref{Xi:232:final} we have shown that 
\begin{equation*}
	\abs{\p_{x}\Xi_{23,\epsilon}(x)}=\abs{ \p_{x}\Xi_{231,\epsilon}(x)+\p_{x}\Xi_{232,\epsilon}(x)}\leq C\norm{H}_{C^{1,\alpha}(\Omega)}, \mbox{ for } x\in \SS^1.
\end{equation*}
Application of dominated convergence as well as the fact that in the previous estimates the integrands where estimated by an integrable function independent of $\epsilon$ shows that the 
\begin{equation}\label{existence:Xi23:eps:limit}
	\lim_{\epsilon\to 0} \p_{x}\Xi_{23,\epsilon}(x):=Q_{23}(x)
\end{equation}  
exists. Therefore, recalling that
$$ \p_{x}\Xi_{\epsilon}(x)= \p_{x}\Xi_{1,\epsilon}(x)+ \p_{x}\Xi_{2,\epsilon}(x)= \p_{x}\Xi_{1,\epsilon}(x)+ \displaystyle\sum_{j=0}^{3} \p_{x}\Xi_{2j,\epsilon}(x)$$
and definitions \eqref{existence:Xi1:eps:limit}, \eqref{existence:Xi2j:eps:limit} and \eqref{existence:Xi23:eps:limit} we infer that
\begin{equation}
	\lim_{\epsilon\to 0} \p_{x} \Xi_{\epsilon}(x):= Q_{1}(x)+\displaystyle\sum_{j=0}^{3}Q_{2j}(x):=\mathcal{Q}(x).
\end{equation}
By the fundamental theorem of calculus, one can readily see that
\begin{equation}
	\Xi_{\epsilon}(x)-\Xi_{\epsilon}(0)= \int_{0}^{x}  \p_{x} \Xi_{\epsilon}(\xi) \ d \xi, \mbox{ for } \epsilon>0.
\end{equation}
We remark that in Lemma \ref{Holder:singular:integral:alpha} we already showed that $\lim_{\epsilon\to 0} \Xi_{\epsilon}(x)=\Xi(x)$ for $x\in\SS^1$ we have by uniqueness that
\begin{equation}
	\lim_{\epsilon\to 0} \p_{x} \Xi_{\epsilon}(x):=\mathcal{Q}(x)=\p_{x} \Xi(x).
\end{equation}
Noticing that we proved in \eqref{Xi1:estimate:final:1alpha} the H\"older semi-norm bound
$$ \abs{\p_{x}\Xi_{1,\epsilon}(x)-\p_{x}\Xi_{1,\epsilon}(0)}\leq C\abs{x}^{\alpha}\norm{H}_{C^{1,\alpha}(\Omega)}, \mbox{ for } x\in\SS^1, $$
and recalling that $ \p_{x}\Xi_{\epsilon}= \p_{x}\Xi_{1,\epsilon}+ \p_{x}\Xi_{2,\epsilon}$ it remains to prove that
\begin{equation}\label{estimate:left:Xi_2}
	\abs{\p_{x}\Xi_{2,\epsilon}(x)-\p_{x}\Xi_{2,\epsilon}(0)}\leq C\abs{x}^{\alpha}\norm{H}_{C^{1,\alpha}(\Omega)}, \mbox{ for } x\in\SS^1.
\end{equation}
Indeed, combining the last two estimates we conclude the $C^{1,\alpha}$ semi-norm estimate \eqref{principal:estimate:1alpha:eps}. 
To that purpose, recalling \eqref{rest:change}-\eqref{R3:last:align} and the decomposition \eqref{decomposition:xi:2} we compute the difference
\begin{align*}
	\p_{x}\Xi_{2,\epsilon}(x)- \p_{x}\Xi_{2,\epsilon}(0) &= \displaystyle\sum_{j=0}^{3}\left( \p_{x}\Xi_{2j,\epsilon}(x)- \p_{x}\Xi_{2j,\epsilon}(0) \right)=M_{0,\epsilon}+M_{1,\epsilon}+M_{2,\epsilon}+M_{3,\epsilon},
\end{align*}
where
\begin{align*}
	M_{0,\epsilon}&= - \int_{\epsilon}^{L}dy\int_{\mathbb{S}^{1}} d\eta \ \big[ \mathcal{R}_{0}(x,y,\eta) \left(\p_{\eta}H( \eta, y)-\p_{\eta}H( x, 0)\right) -\mathcal{R}_{0}(0,y,\eta)  \left(\p_{\eta}H( \eta, y)-\p_{\eta}H( 0, 0)\right) \big], \\
	M_{1,\epsilon}&=- \int_{\epsilon}^{L}dy\int_{\mathbb{S}^{1}}d\eta \ \big[ \mathcal{R}_{1}(x,y,\eta)- \mathcal{R}_{1}(0,y,\eta)\big] H( \eta, y), \\
	M_{2,\epsilon}&=- \int_{\epsilon}^{L}dy\int_{\mathbb{S}^{1}}d\eta \ \big[ \mathcal{R}_{2}(x,y,\eta)- \mathcal{R}_{1}(0,y,\eta)\big] H( \eta, y), \\
	M_{3,\epsilon}&=- \int_{\epsilon}^{L}dy\int_{\mathbb{S}^{1}}d\eta \ \big[ \mathcal{R}_{3}(x,y,\eta)- \mathcal{R}_{3}(0,y,\eta)\big] H( \eta, y).
\end{align*}
with $\mathcal{R}_{j}$ for $j=0,\ldots,3$ defined as in \eqref{R0:last:align}-\eqref{R3:last:align}.

We divide as in the previous proposition, the region of integration $\{ (\eta,y)\in \mathbb{S}^1 \times [\epsilon,L] \}$ into sets of the form 
\begin{equation}\label{regions:integration:R}
	\begin{aligned}
		R^{\epsilon}_{\Omega,\leq}&=\{(\eta,y)\in \Omega:\displaystyle\max\{\abs{y},\abs{\eta}\}\leq 2\abs{x} \mbox{ and } \epsilon\leq y\leq L\}, \\
		R^{\epsilon}_{\Omega,>}&=\{(\eta,y)\in \Omega:\displaystyle\max\{\abs{y},\abs{\eta}\}> 2\abs{x} \mbox{ and } \epsilon\leq y\leq L\}, 
	\end{aligned} 
\end{equation}
for $\epsilon\geq 0$ and estimate each integral in the different sets. For the sake of simplicity we will write $R^{0}_{\Omega,\leq}=R_{\Omega,\leq}$ and $R^{0}_{\Omega,>}=R_{\Omega,>}$.
Therefore, we have
\begin{equation}\label{notation:out:in}
	M_{k,\epsilon}=\int_{R^{\epsilon}_{\Omega,\leq}} dy d\eta \big[\ldots\big]+ \int_{R^{\epsilon}_{\Omega,>}}dy d\eta\big[\ldots\big]= M_{k1,\epsilon}+M_{k2,\epsilon}, \quad \mbox{ for } k=0,\ldots,2. 
\end{equation}
Notice that we did not include above the most singular term, namely $M_{3,\epsilon}$. We will prove the required estimates for that quantity later on. Let us now show how to bound the other integral quantities $M_{0,\epsilon}$ to $M_{2,\epsilon}$ in the different regions of integration $R^{\epsilon}_{\Omega,\leq}$ and $R^{\epsilon}_{\Omega,>}$.

In the inner region $R^{\epsilon}_{\Omega,\leq}$, using the H\"older regularity for $H$ we readily see that
\begin{align}
	\abs{M_{01,\epsilon}} &\leq  C\norm{\Lambda}_{C^{2}(\Omega)} \norm{H}_{C^{1,\alpha}(\Omega)}  \int_{R_{\Omega,\leq}} dy \ d\eta \bigg[ \frac{y^{2}(\abs{x-\eta}^{\alpha}+y^{\alpha})}{\left((x-\eta)^{2}+y^{2}\right)^{3/2}}+\frac{y^2(\abs{\eta}^{\alpha}+y^{\alpha})}{\left(\eta^{2}+y^{2}\right)^{3/2}}\bigg] \nonumber \\
	&\leq C\norm{H}_{C^{1,\alpha}(\Omega)}\abs{x}^{\alpha},\label{estimate:M01:eps} \\
	\abs{M_{11,\epsilon}} &\leq  C\norm{\Lambda}_{C^{2,\alpha}(\Omega)}  \norm{H}_{L^{\infty}(\Omega)} \int_{R_{\Omega,\leq}} dy \ d\eta  \bigg[ \frac{y^{1+\alpha}}{\left((x-\eta)^{2}+y^{2}\right)^{3/2}}+\frac{y^{1+\alpha}}{\left(\eta^{2}+y^{2}\right)^{3/2}}\bigg] \nonumber \\
	&\leq C\norm{H}_{L^{\infty}(\Omega)}\abs{x}^{\alpha}, \\
	\abs{M_{21,\epsilon}} &\leq  C\norm{\Lambda}^{2}_{C^{2}(\Omega)}  \norm{H}_{L^{\infty}(\Omega)} \int_{R_{\Omega,\leq}} dy \ d\eta \bigg[ \frac{y^{3}}{\left((x-\eta)^{2}+y^{2}\right)^{3/2}}+\frac{y^{3}}{\left(\eta^{2}+y^{2}\right)^{3/2}}\bigg] \nonumber \\
	&\leq C \norm{H}_{L^{\infty}(\Omega)}\abs{x}^{\alpha}.
\end{align}

In the outer region $R^{\epsilon}_{\Omega,>}=\{\displaystyle\max\{\abs{y},\abs{\eta}\}> 2\abs{x}\mbox{ and } \epsilon\leq y\leq L\}$, we can estimate $M_{12,\epsilon},M_{22,\epsilon}$ by applying the mean value theorem and using Lemma \ref{calculus:lemma} as
\begin{align}
	\abs{M_{12,\epsilon}}+ \abs{M_{22,\epsilon}}&\leq C\abs{x}\norm{\Lambda}_{C^{2,\alpha}(\Omega)}\norm{H}_{L^{\infty}(\Omega)} \int_{R_{\Omega,>}} dy \ d\eta  \bigg[ \frac{y^{1+\alpha}}{\left((x-\eta)^{2}+y^{2}\right)^{2}}+\frac{y^{3}}{\left((x-\eta)^{2}+y^{2}\right)^{2}}\bigg] \nonumber \\
	&\leq C\abs{x}^{\alpha}\norm{H}_{L^{\infty}(\Omega)}, \label{bound:N22N32}
\end{align}

To estimate the term $M_{02,\epsilon}$, we rewrite it by adding and subtracting as
\begin{align*}
	M_{02,\epsilon}&=- \int_{R^{\epsilon}_{\Omega,>}} dy d\eta \ (\mathcal{R}_{0}(x,y,\eta)-\mathcal{R}_{0}(0,y,\eta)) \left(\p_{\eta}H( \eta, y)-\p_{\eta}H( x, 0)\right) \\
	&-\left(\p_{\eta}H(x,0)-\p_{\eta}H(0,0)\right) \int_{R^{\epsilon}_{\Omega,>}} dy d\eta \ \mathcal{R}_{0}(0,y,\eta)=M_{021,\epsilon}+M_{022,\epsilon}.
\end{align*}
Recalling the definition of $\mathcal{R}_{0}$ in \eqref{R0:last:align}, using Lemma \ref{calculus:lemma} and applying the mean value theorem we find that
\begin{align*}
	\abs{M_{021,\epsilon}}&\leq C\abs{x}\norm{\Lambda}_{C^{2}(\Omega)}\norm{H}_{C^{1,\alpha}(\Omega)}  \int_{R_{\Omega,>}} dy \ d\eta \frac{y^{2}\left((x-\eta)^{\alpha}+y^{\alpha}\right)}{\left((x-\eta)^{2}+y^{2}\right)^{2}}\\
	&\leq C\abs{x}^{\alpha}\norm{\Lambda}_{C^{2}(\Omega)}\norm{H}_{C^{1,\alpha}(\Omega)}.
\end{align*}
Moreover, similarly we obtain
\begin{align*}
	\abs{M_{022,\epsilon}}&\leq C\abs{x}^{\alpha}\norm{\Lambda}_{C^{2}(\Omega)}\norm{H}_{C^{1,\alpha}(\Omega)} \int_{R_{\Omega,>}} dy \ d\eta\frac{y^{2}}{\left(\eta^{2}+y^{2}\right)^{3/2}}\\
	&\leq C\abs{x}^{\alpha}\norm{\Lambda}_{C^{2}(\Omega)}\norm{H}_{C^{1,\alpha}(\Omega)}.
\end{align*}
Hence, we have that
\begin{equation}\label{estimate:M02:eps}
	\abs{M_{02,\epsilon}}\leq C\abs{x}^{\alpha}\norm{H}_{C^{1,\alpha}(\Omega)}.
\end{equation}
Collecting estimates \eqref{estimate:M01:eps}-\eqref{estimate:M02:eps} we obtain that
\begin{equation}\label{estimate:M0:M2:eps}
	\abs{M_{0,\epsilon}+M_{1,\epsilon}+M_{2,\epsilon}}\leq C\abs{x}^{\alpha}\norm{H}_{C^{1,\alpha}(\Omega)}.
\end{equation}
Let us deal with the most singular term $M_{3,\epsilon}$, given by
\begin{equation}
	M_{3,\epsilon}=- \int_{\epsilon}^{L}dy\int_{\mathbb{S}^{1}}d\eta \ \big[ \mathcal{R}_{3}(x,y,\eta)- \mathcal{R}_{3}(0,y,\eta)\big] H( \eta, y), \\
\end{equation}
with $\mathcal{R}_{3}$ as in \eqref{R3:last:align}. We claim that the following estimate holds
\begin{equation}\label{claim:M3:eps}
	\abs{M_{3,\epsilon}}\leq C\abs{x}^{\alpha}\norm{H}_{C^{1,\alpha}(\Omega)}.
\end{equation}

Defining the auxiliary function 
\begin{equation}\label{def:H:tilde:R3}
	\widetilde{H}(\eta,y):= \p_{\eta} \left(\frac{\Lambda(\eta,y)}{y}\right) H( \eta, y) \mbox{ and } \widetilde{\mathcal{R}}_{3}(x,y,\eta)\p_{\eta}\Lambda(\eta,y)=\mathcal{R}_{3}(x,y,\eta)
\end{equation}
we have that
\begin{align}
	M_{3,\epsilon}=&- \int_{\epsilon}^{L}dy\int_{\mathbb{S}^{1}}d\eta \ y \big[ \widetilde{\mathcal{R}}_{3}(x,y,\eta)- \widetilde{\mathcal{R}}_{3}(0,y,\eta)\big] (\widetilde{H}(\eta,y)-\widetilde{H}(0,0)) \nonumber \\
	&\quad+ \widetilde{H}(0,0) \int_{\epsilon}^{L}dy\int_{\mathbb{S}^{1}}d\eta \ y \big[ \widetilde{\mathcal{R}}_{3}(x,y,\eta)- \widetilde{\mathcal{R}}_{3}(0,y,\eta)\big]=M_{4,\epsilon}+M_{5,\epsilon}. \label{M5:eps:splitting}
\end{align}
We first bound $M_{4,\epsilon}$. Rewriting the term we obtain
\begin{align}
	M_{4,\epsilon}&=  - \int_{\epsilon}^{L}dy\int_{\mathbb{S}^{1}}d\eta \ y  \widetilde{\mathcal{R}}_{3}(x,y,\eta) (\widetilde{H}(\eta,y)-\widetilde{H}(x,0)) \nonumber \\
	&\quad \quad +  \int_{\epsilon}^{L}dy\int_{\mathbb{S}^{1}}d\eta \ y  \widetilde{\mathcal{R}}_{3}(x,y,\eta) \left(\widetilde{H}(x,0)-\widetilde{H}(0,0)\right) \nonumber \\
	&\quad \quad \quad  -\int_{\epsilon}^{L}dy\int_{\mathbb{S}^{1}}d\eta \ y  \widetilde{\mathcal{R}}_{3}(0,y,\eta) \left(\widetilde{H}(\eta,y)-\widetilde{H}(0,0)\right)= M_{41,\epsilon}+ M_{42,\epsilon}+ M_{43,\epsilon}. \label{splitting:M4:eps}
\end{align}
Using the decomposition of the regions of integration in \eqref{regions:integration:R} we find that
\begin{align*}
	M_{41,\epsilon}&=- \int_{R^{\epsilon}_{\Omega,\leq}} dy d\eta \ \big[ \ldots \big] -\int_{R^{\epsilon}_{\Omega,>}} dy d\eta \ \big[ \ldots \big]=M_{411,\epsilon}+M_{412,\epsilon}, \\
	M_{43,\epsilon}&=- \int_{R^{\epsilon}_{\Omega,\leq}} dy d\eta \ \big[ \ldots \big] -\int_{R^{\epsilon}_{\Omega,>}} dy d\eta \ \big[ \ldots \big]=M_{431,\epsilon}+M_{432,\epsilon}.
\end{align*}
Similarly as before, using the bound \eqref{est:calcu} we have that
\begin{align}\label{M411:M413:eps}
	\abs{M_{411,\epsilon}}&\leq  C \norm{\widetilde{H}}_{C^{\alpha}(\Omega)}  \int_{R_{\Omega,\leq}} dy \ d\eta \frac{y^{2}(\abs{x-\eta}^{\alpha}+y^{\alpha})}{\left((x-\eta)^{2}+y^{2}\right)^{2}} 
	\leq C\norm{\widetilde{H}}_{C^{\alpha}(\Omega)}\abs{x}^{\alpha}, \\
	\abs{M_{431,\epsilon}} &\leq  C \norm{\widetilde{H}}_{C^{\alpha}(\Omega)}  \int_{R_{\Omega,\leq}} dy \ d\eta \frac{y^{2}(\abs{\eta}^{\alpha}+y^{\alpha})}{\left(\abs{\eta}^{2}+y^{2}\right)^{2}} 
	\leq C\norm{\widetilde{H}}_{C^{\alpha}(\Omega)}\abs{x}^{\alpha}.
\end{align}
On the exterior region $R^{\epsilon}_{\Omega,>}$ given in \eqref{regions:integration:R}, we obtain after rearranging terms by adding and substracting $\widetilde{H}(0,0)$ that
\begin{align}
	M_{412,\epsilon}+M_{432,\epsilon}&=   \int_{R^{\epsilon}_{\Omega,>}} dy d\eta  \ y  \widetilde{\mathcal{R}}_{3}(x,y,\eta) (\widetilde{H}(x,0)-\widetilde{H}(0,0))\\
	&\quad \quad -\int_{R^{\epsilon}_{\Omega,>}} dy d\eta \ y \left(\widetilde{\mathcal{R}}_{3}(x,y,\eta)-\widetilde{\mathcal{R}}_{3}(0,y,\eta)\right) (\widetilde{H}(\eta,y)-\widetilde{H}(0,0)). \\
	&=M_{6,\epsilon}+M_{7,\epsilon}
\end{align}
The later integral can be estimated using Lemma \ref{calculus:lemma} and the mean value theorem as
	\begin{equation}\label{estimate:M7:eps}
		\abs{M_{7,\epsilon}}\leq C  \norm{\Lambda}_{C^{2}(\Omega)}\norm{\widetilde{H}}_{C^{\alpha}(\Omega)}  \int_{R_{\Omega,>}} dy \ d\eta \abs{x} \frac{y^{2}\left(\abs{\eta}^{\alpha}+y^{\alpha}\right)}{\left((x-\eta)^{2}+y^{2}\right)^{3/2}} \leq C \abs{x}^{
		\alpha}\norm{\widetilde{H}}_{C^{\alpha}(\Omega)}.
	\end{equation}

On the other hand, to bound $M_{6,\epsilon}$, we follow the same ideas as we did to estimate the term $K_{4,\epsilon}$ or $K_{6,\epsilon}$ in Lemma \ref{Holder:singular:integral:alpha}. Using the decomposition \eqref{decomposition:Lambda} and bound \eqref{bound:square}, we can write
\begin{align*}
	M_{6,\epsilon}&=(\widetilde{H}(x,0)-\widetilde{H}(0,0)) \int_{R^{\epsilon}_{\Omega,>}} dy d\eta  \ y^2  \frac{e^{i(x-\eta)-y}(e^{i(x-\eta)-y-i\Lambda(\eta,y)})^2}{(1-e^{i(x-\eta)-y})(1-e^{i(x-\eta)-y-iA(\eta)y}\left(1+r_{1}(\eta,y)\right)))^3}   \\
	&=(\widetilde{H}(x,0)-\widetilde{H}(0,0)) \int_{R^{\epsilon}_{\Omega,>}} dy d\eta \ y^2  \frac{e^{i(x-\eta)-y}(e^{i(x-\eta)-y-i\Lambda(\eta,y)})^2}{(1-e^{i(x-\eta)-y})(1-e^{i(x-\eta)-y-iA(\eta)y})^3}\left(1+r_2(\eta,y)\right)  \\
	&= (\widetilde{H}(x,0)-\widetilde{H}(0,0)) \left( \int_{R^{\epsilon}_{\Omega,>}} dy \ d\eta \big[\ldots \big]+ \int_{R^{\epsilon}_{\Omega,>}} dy \ d\eta \big[\ldots \big] r_{2}(\eta,y)\right)=M_{61,\epsilon}+M_{62,\epsilon}
\end{align*}
where the remainder term $r_{1}(\eta,y)$ can be bounded $\abs{r_1(\eta,y)}\leq Cy^{2}$ and remainder $r_2(\eta,y)$ is bounded by $\abs{r_2(\eta,y)}\leq C \abs{y}.$ The integrand in $M_{62,\epsilon}$ is integrable and can be bounded by means of Lemma \ref{calculus:lemma} as
\begin{equation}\label{M62:bound:eps}
	\abs{M_{62,\epsilon}}\leq C \abs{x}^\alpha \norm{\widetilde{H}}_{C^{\alpha}(\Omega)}  \int_{R_{\Omega,>}} dy \ d\eta \frac{y^3}{\left((x-\eta)^2+y^{2}\right)^{2}}\leq C \abs{x}^\alpha \norm{\widetilde{H}}_{C^{\alpha}(\Omega)}. 
\end{equation}
The most involved term is $M_{61,\epsilon}$.  To deal with it, we argue as in the estimate of $K_{41,\epsilon}$ in \eqref{estimate:K41:eps} or $K_{61,\epsilon}$ in \eqref{estimate:K61:eps}. Using Taylor expansion, recalling decomposition and the bound \eqref{A:bound} we infer that
\begin{equation}\label{estimate:M61:eps}
M_{61,\epsilon}= (\widetilde{H}(x,0)-\widetilde{H}(0,0)) \int_{R^{\epsilon}_{\Omega,>}} dy \ d\eta \frac{y^2}{(i(x-\eta)-y)(-i(x-\eta)-y-iA(0)y)^{3}}+M_{612,\epsilon}
\end{equation}
with $\abs{M_{612,\epsilon}} \leq C \abs{x}^{\alpha} \norm{\widetilde{H}}_{C^{\alpha}(\Omega)}.$ Performing the change of variables $x-\eta=y\zeta$ we have that
\begin{align}
	M_{611,\epsilon}&= (\widetilde{H}(x,0)-\widetilde{H}(0,0)) \int_{R^{\epsilon}_{\Omega,>}} dy \ d\eta \frac{y^2}{(i(x-\eta)-y)(-i(x-\eta)-y-iA(0)y)^{3}} \\
	&=-(\widetilde{H}(x,0)-\widetilde{H}(0,0)) \int_{\epsilon}^{L}\ \frac{dy}{y}\int_{\widetilde{\Sigma}(x,y)} \frac{d\zeta}{(i\zeta+1)(i\zeta+1+iA(0))^3}  
\end{align}
where the domain integration $\widetilde{\Sigma}$ is given by
$$\widetilde{\Sigma}(x,y)= \{\zeta\in \mathbb{R}: \max \{1,\abs{\zeta-\dfrac{x}{y}}\}> \frac{2\abs{x}}{y}, \ \abs{\zeta}\leq \frac{\pi}{y} \}.$$
We recall that the integration domain $\widetilde{\Sigma}$ is just a shifted version of the integration domain $\Sigma$ defined in \eqref{sigma:def}. Therefore, 
similarly as we estimated integral \eqref{integral:L}, we have that
\begin{equation}\label{bound:M611}
	\abs{M_{611,\epsilon}}\leq  C \abs{x}^\alpha \norm{\widetilde{H}}_{C^{\alpha}(\Omega)}.
\end{equation}
and hence, collecting estimates \eqref{M62:bound:eps}, \eqref{estimate:M61:eps}, \eqref{bound:M611} we find that
\begin{equation}\label{M6:eps:final}
	\abs{M_{6,\epsilon}}\leq C \abs{x}^{\alpha} \norm{\widetilde{H}}_{C^{\alpha}(\Omega)}.
\end{equation}
To end the bound for $M_{4,\epsilon}$ we are left to estimate $M_{42,\epsilon}$ in \eqref{splitting:M4:eps}. Identifying $\SS^1$ with $I_{x}=[x-\pi,x+\pi]$ for $x\in[-\pi,\pi]$ and using Taylor's expansion and decomposition \eqref{decomposition:Lambda} we obtain that
\begin{align}
	M_{42,\epsilon}= \left(\widetilde{H}(x,0)-\widetilde{H}(0,0)\right)   \int_{\epsilon}^{L}dy\int_{I_x} d\eta \ \frac{y^2 }{(i(x-\eta)-y)(i(x-\eta)-y-iyA(\eta))^3}+M_{422,\epsilon} \label{splitting:M42:eps}
	\end{align}
where the term $M_{422,\epsilon}$ has an  integrable singularity and can be estimated using \eqref{bound:square}, namely 
\begin{equation} \label{estimate:M422:eps}
	\abs{M_{422,\epsilon}}\leq C\abs{x}^{\alpha} \norm{\widetilde{H}}_{C^{\alpha}(\Omega)}. 
\end{equation}
We further decompose  the function $A(\eta)$ by 
$$	A(\eta)=A(x)+[A(\eta)-A(x)], \mbox{ where } \abs{A(\eta)-A(x)}\leq \delta_{0} \abs{x-\eta}^{\alpha}, $$
to write
\begin{align*}
	M_{422,\epsilon}= \left(\widetilde{H}(x,0)-\widetilde{H}(0,0)\right) \int_{\epsilon}^{L}dy\int_{I_x} d\eta \ \frac{y^2 }{(i(x-\eta)-y)(i(x-\eta)-y-iyA(x))^3}+M_{4222,\epsilon}
	\end{align*}
with $\abs{M_{4222,\epsilon}}\leq C \abs{x}^{\alpha} \norm{\widetilde{H}}_{C^{\alpha}(\Omega)}$. After a change of variable $(x-\eta)=y\zeta$ we find that
\begin{align*}
	\left(\widetilde{H}(x,0)-\widetilde{H}(0,0)\right) & \int_{\epsilon}^{L}dy\int_{I_x} d\eta \ \frac{y^2 }{(i(x-\eta)-y)(i(x-\eta)-y-iyA(x))^3} \\
	&= \left(\widetilde{H}(x,0)-\widetilde{H}(0,0)\right)  \int_{\epsilon}^{L}dy\int_{-\frac{\pi}{y}}^{\frac{\pi}{y}} \frac{d\zeta}{(i\zeta+1)(i\zeta+1+iA(x))^3}.
\end{align*}
Estimating the integral as in \eqref{treat:N11:eps} by extending the value $\zeta$ to the whole space $\mathbb{R}$ and contour integrating via residues yields that
\begin{equation}\label{estimate:M422:eps:fin}
	\abs{M_{422,\epsilon}}\leq  C \abs{x}^{\alpha} \norm{\widetilde{H}}_{C^{\alpha}(\Omega)}
\end{equation}
and hence estimates \eqref{splitting:M42:eps}, \eqref{estimate:M422:eps} and \eqref{estimate:M422:eps:fin} shows that
\begin{equation} \label{M42:est:final}
		\abs{M_{42,\epsilon}}\leq  C \abs{x}^{\alpha} \norm{\widetilde{H}}_{C^{\alpha}(\Omega)}.
	\end{equation}
	Thus, combining bounds \eqref{splitting:M4:eps}, \eqref{M411:M413:eps}, \eqref{estimate:M7:eps}, \eqref{M6:eps:final} and \eqref{M42:est:final} we conclude that
\begin{equation}\label{M4:eps:final}
	\abs{M_{4,\epsilon}}\leq C \abs{x}^{\alpha} \norm{\widetilde{H}}_{C^{\alpha}(\Omega)}.
\end{equation}
We provide now the estimate for $M_{5,\epsilon}$ in \eqref{M5:eps:splitting}. Using the change of variables $\widetilde{\eta}=\eta-x$ and recalling the definition of $\mathcal{R}_{3}$ in \eqref{R3:last:align} we find that
\begin{align*}
	M_{5,\epsilon}&= \widetilde{H}(0,0) \int_{\epsilon}^{L}dy\int_{\mathbb{S}^{1}}d\eta \ \frac{y^2 (e^{-i\eta-y})^3}{(1-e^{-i\eta-y})}\big[ \frac{(e^{-i\Lambda(\eta+x,y)})^2}{(1-e^{-i \eta-y-i\Lambda(\eta+x,y)})^3}-\frac{(e^{-i\Lambda(\eta,y)})^2}{(1-e^{-i \eta-y-i\Lambda(\eta,y)})^3} \big] \\
	& \hspace{-0.5cm} =\widetilde{H}(0,0) \int_{\epsilon}^{L}dy\int_{\mathbb{S}^{1}}d\eta \ \frac{y^2 (e^{-i\eta-y})^3}{(1-e^{-i\eta-y})}\big[ \frac{(e^{-i\Lambda(\eta+x,y)})^2-(e^{-i\Lambda(\eta,y)})^2}{(1-e^{-i \eta-y-i\Lambda(\eta+x,y)})^3} \big] \\
	& +\widetilde{H}(0,0) \int_{\epsilon}^{L}dy\int_{\mathbb{S}^{1}}d\eta \ \frac{y^2 (e^{-i\eta-y})^3}{(1-e^{-i\eta-y})} (e^{-i\Lambda(\eta,y)})^2 \big[\frac{1}{(1-e^{-i \eta-y-i\Lambda(\eta+x,y)})^3}-\frac{1}{(1-e^{-i \eta-y-i\Lambda(\eta,y)})^3}\big] \\
	&\quad \quad = M_{51,\epsilon}+M_{52,\epsilon}
\end{align*}
where by abusing of notation we wrote $\eta$ instead of $\widetilde{\eta}$. Identifying $\SS^1$ with $I_{x}=[x-\pi,x+\pi]$ for $x\in [-\pi,\pi]$, using bound \eqref{est:calcu} and Taylor expansion we have that
\begin{align}\label{M51:eps}
	\abs{M_{51,\epsilon}} &\leq C \norm{\widetilde{H}}_{L^{\infty}(\Omega)} \int_{\epsilon}^{L}dy\int_{I_{x}} \frac{y^3}{(\abs{\eta}^2+y^2)^4}\frac{ \abs{\Lambda(\eta+x,y)-\Lambda(\eta,y)}}{y}  \nonumber \\
	&\leq  C\norm{\Lambda}_{C^{2,\alpha}(\Omega)} \norm{\widetilde{H}}_{L^{\infty}(\Omega)} \int_{\epsilon}^{L}dy\int_{I_{x}} \abs{x} \frac{y^3}{(\abs{\eta}^2+y^2)^2} \leq C \abs{x}^{\alpha} \norm{\widetilde{H}}_{L^{\infty}(\Omega)},
\end{align}
where in the second inequality we used Assumption \ref{assumption:Lambda}
to bound 
\begin{equation}\label{computation:extra:Lambda}
\frac{\abs{\Lambda(\eta+x,y)-\Lambda(\eta,y)}}{y}\leq C \abs{x}  \norm{\Lambda}_{C^{2,\alpha}(\Omega)}.
\end{equation}

To get the desired bound for $M_{52,\epsilon}$ let us first rewrite the term in brackets inside the integral. To that purpose, denoting by $\mathfrak{D}=e^{-i \eta-y-i\Lambda(\eta,y)}$ we have that
\begin{align}\label{rewrite:brackets}
	\frac{1}{(1-e^{-i \eta-y-i\Lambda(\eta+x,y)})^3}-\frac{1}{(1-\mathfrak{D})^3}
&=\frac{1}{(1-\mathfrak{D}e^{-i(\Lambda(\eta+x,y)-\Lambda(\eta,y)})^3}-\frac{1}{(1-\mathfrak{D})^3} \nonumber \\
&=\frac{1}{(1-\mathfrak{D})^{3}}\bigg[\frac{1}{\left(1-\frac{\mathfrak{D}(e^{-i(\Lambda(\eta+x,y)-\Lambda(\eta,y))}-1)}{(1-\mathfrak{D})}\right)^3}-1\bigg].
\end{align}
Therefore, plugging \eqref{rewrite:brackets} into $M_{52,\epsilon}$ and recalling that $\mathfrak{D}=e^{-i \eta-y-i\Lambda(\eta,y)}$
we have that
\begin{equation}
 M_{52,\epsilon}=\widetilde{H}(0,0) \int_{\epsilon}^{L}dy\int_{\mathbb{S}^{1}}d\eta \ \frac{y^2 (e^{-i\eta-y})^3 (e^{-i\Lambda(\eta,y)})^2}{(1-e^{-i\eta-y})(1-e^{-i\eta-y-i\Lambda(\eta,y)})^3} \ \mathfrak{J}_{x}(\eta,y)  
\end{equation}
where 
$$\mathfrak{J}_{x}(\eta,y) =\frac{1}{\left(1-\frac{e^{-i\eta-y-i\Lambda(\eta,y)}(e^{-i(\Lambda(\eta+x,y)-\Lambda(\eta,y))}-1)}{(1-e^{-i\eta-y-i\Lambda(\eta,y)})}\right)^3}-1.$$
Next, we claim that the function $\mathfrak{J}_{x}(\eta,y)$ is $C^\alpha$ H\"older is the variables $(\eta,y)\in \Omega$, this is
\begin{equation}\label{bound:Calpha:mathfrakJ}
\norm{\mathfrak{J}_{x}(\eta,y)}_{C^{\alpha}(\Omega)}\leq C \abs{x}^{\alpha}.
\end{equation}
Indeed, to show the pointwise $L^{\infty}(\Omega)$ bound we readily check by applying Taylor's expansion, Assumption \ref{assumption:Lambda} and bound \eqref{computation:extra:Lambda} that
\begin{equation}\label{Linfty:bound:mathcalJ}
  \norm{\mathfrak{J}_{x}(\eta,y)}_{L^{\infty}(\Omega)}\leq  C\abs{x} \norm{\Lambda}_{C^{2,\alpha}(\Omega)} \leq C\abs{x}.
\end{equation}
To show the $\alpha-$H\"older semi-norm bound, we have prove that
\begin{equation}\label{alpha:Holder:seminorm:bound}
 \abs{\mathfrak{J}_{x}(\eta_1,y_1)-\mathfrak{J}_{x}(\eta_2,y_2)}\leq C\abs{x}^{\alpha}\left(\abs{y_{1}-y_{2}}+\abs{\eta_{1}-\eta_{2}} \right)^{\alpha},\mbox{ for } \eta_1,\eta_2\in\mathbb{S}^{1}, y_{1},y_{2}\in [\epsilon,L].
\end{equation}
By computing the difference and using the notation 
\begin{align*}
\mathfrak{N}_{x}(\eta_{1},y_{1})&=\frac{e^{-i\eta_{1}-y_{1}-i\Lambda(\eta_{1},y_{1})}(e^{-i(\Lambda(\eta_{1}+x,y_{1})-\Lambda(\eta_{1},y_{1}))}-1)}{(1-e^{-i\eta_{1}-y_{1}-i\Lambda(\eta_{1},y_{1})})}, \\
\mathfrak{N}_{x}(\eta_{2},y_{2})&=\frac{e^{-i\eta_{2}-y_{2}-i\Lambda(\eta_{2},y_{2})}(e^{-i(\Lambda(\eta_{2}+x,y_{2})-\Lambda(\eta_{2},y_{2}))}-1)}{(1-e^{-i\eta_{2}-y_{2}-i\Lambda(\eta_{2},y_{2})})}
\end{align*} 
we have that
\begin{align}
\abs{\mathfrak{J}_{x}(\eta_1,y_1)-\mathfrak{J}_{x}(\eta_2,y_2)}&=\abs{ \frac{1}{\left(1-\mathfrak{N}_{x}(\eta_{1},y_{1})\right)^3}-\frac{1}{\left(1-\mathfrak{N}_{x}(\eta_{2},y_{2})\right)^3}} \\
&=\abs{\int_{\mathfrak{N}_{x}(\eta_{2},y_{2})}^{\mathfrak{N}_{x}(\eta_{1},y_{1})}\frac{3}{(1-\zeta)^{4}} d\zeta} \leq C \abs{\mathfrak{N}_{x}(\eta_{1},y_{1})-\mathfrak{N}_{x}(\eta_{2},y_{2})}.
\end{align}
Moreover, further manipulations show that
\begin{equation}
\abs{\mathfrak{N}_{x}(\eta_{1},y_{1})-\mathfrak{N}_{x}(\eta_{2},y_{2})}= \mathfrak{I}_{1}+\mathfrak{I}_{2}
\end{equation}
where
\begin{equation*}
\mathfrak{I}_{1}=\abs{e^{-i\eta_{1}-y_{1}-i\Lambda(\eta_{1},y_{1})}\frac{(e^{-i(\Lambda(\eta_{1}+x,y_{1})-\Lambda(\eta_{1},y_{1}))}-1)}{y_{1}}  \bigg[\frac{y_{1}}{(1-e^{-i\eta_{1}-y_{1}-i\Lambda(\eta_{1},y_{1})})}- \frac{y_{2}}{(1-e^{-i\eta_{2}-y_{2}-i\Lambda(\eta_{2},y_{2})})}\bigg]}
\end{equation*}
\begin{align*}
\mathfrak{I}_{2}=&\abs{\bigg[ e^{-i\eta_{1}-y_{1}-i\Lambda(\eta_{1},y_{1})}\frac{(e^{-i(\Lambda(\eta_{1}+x,y_{1})-\Lambda(\eta_{1},y_{1}))}-1)}{y_{1}}- e^{-i\eta_{2}-y_{2}-i\Lambda(\eta_{2},y_{2})}\frac{(e^{-i(\Lambda(\eta_{2}+x,y_{2})-\Lambda(\eta_{2},y_{2}))}-1)}{y_{2}}\bigg]} \\
& \quad \quad \times \abs{\frac{y_{2}}{(1-e^{-i\eta_{2}-y_{2}-i\Lambda(\eta_{2},y_{2})})}}.
\end{align*}
Let us first bound $\mathfrak{I}_{1}$. Using Taylor's expansion, \eqref{computation:extra:Lambda} and Assumption \ref{assumption:Lambda} we have that
\begin{equation}\label{estimate:mathfrak:I1:final}
    \mathfrak{I}_{1}\leq C \abs{x} \norm{\Lambda}_{C^{2,\alpha}(\Omega)} \abs{\frac{y_{1}}{(1-e^{-i\eta_{1}-y_{1}-i\Lambda(\eta_{1},y_{1})})}- \frac{y_{2}}{(1-e^{-i\eta_{2}-y_{2}-i\Lambda(\eta_{2},y_{2})})}} \leq C\abs{x} \left( \abs{y_{1}-y_{2}}+\abs{\eta_{1}-\eta_{2}}\right).
\end{equation}
To bound $\mathfrak{I}_{2}$, we further split it as  $
\mathfrak{I}_{2}= \mathfrak{I}_{21}+\mathfrak{I}_{22}
$
where
\begin{align*}
\mathfrak{I}_{21}&=\abs{ \left(e^{-i\eta_{1}-y_{1}-i\Lambda(\eta_{1},y_{1})}-e^{-i\eta_{2}-y_{2}-i\Lambda(\eta_{2},y_{2})}\right)
\frac{(e^{-i(\Lambda(\eta_{1}+x,y_{1})-\Lambda(\eta_{1},y_{1}))}-1)}{y_{1}}\left( \frac{y_{2}}{1-e^{-i\eta_{2}-y_{2}-i\Lambda(\eta_{2},y_{2})}}\right)}, \\
\mathfrak{I}_{22}&=\abs{\left(\frac{(e^{-i(\Lambda(\eta_{1}+x,y_{1})-\Lambda(\eta_{1},y_{1}))}-1)}{y_{1}}-\frac{(e^{-i(\Lambda(\eta_{2}+x,y_{2})-\Lambda(\eta_{2},y_{2}))}-1)}{y_{2}} \right)
\frac{e^{-i\eta_{2}-y_{2}-i\Lambda(\eta_{2},y_{2})}y_{2}}{1-e^{-i\eta_{2}-y_{2}-i\Lambda(\eta_{2},y_{2})}}}.
\end{align*}
Similarly as before, Taylor's expansion, \eqref{computation:extra:Lambda} and Assumption \ref{assumption:Lambda} yields 
\begin{equation}\label{estimate:mathfrak:I21}
\mathfrak{I}_{21}\leq C\abs{x} \abs{e^{-i\eta_{1}-y_{1}-i\Lambda(\eta_{1},y_{1})}-e^{-i\eta_{2}-y_{2}-i\Lambda(\eta_{2},y_{2})}} \leq C\abs{x}\left( \abs{y_{1}-y_{2}}+\abs{\eta_{1}-\eta_{2}} \right).
\end{equation}
To bound $\mathfrak{I}_{22}$ we first notice that by Lemma \ref{calculus:lemma} we have that
\begin{equation}\label{mathfrak:I22:estimate}
\mathfrak{I}_{22} \leq C \abs{\left(\frac{(e^{-i(\Lambda(\eta_{1}+x,y_{1})-\Lambda(\eta_{1},y_{1}))}-1)}{y_{1}}-\frac{(e^{-i(\Lambda(\eta_{2}+x,y_{2})-\Lambda(\eta_{2},y_{2}))}-1)}{y_{2}} \right)}
\end{equation}
On the other hand, denoting by 
$$ \mathfrak{L}_{x}(\eta,y)=\frac{\Lambda(\eta+x,y)-\Lambda(\eta,y)}{y} \mbox{ and } \mathfrak{H}(\xi,y)= \frac{e^{-iy\xi}-1}{y} $$
we have that
\begin{align*}
\frac{(e^{-i(\Lambda(\eta_{1}+x,y_{1})-\Lambda(\eta_{1},y_{1}))}-1)}{y_{1}}-\frac{(e^{-i(\Lambda(\eta_{2}+x,y_{2})-\Lambda(\eta_{2},y_{2}))}-1)}{y_{2}}= \int_{(\xi_1,y_1)}^{(\xi_2,y_2)}\bigg[  \frac{\partial \mathfrak{H} }{\partial\xi} (\xi,y) \ d\xi + \frac{\partial \mathfrak{H} }{\partial y} (\xi,y) \ dy \bigg]
\end{align*}
where the right hand side is the line integral connecting the point $(\xi_1,y_1)$ with $(\xi_2,y_2)$ where $\xi_1=\mathfrak{L}_{x}(\eta_1,y_1)$ and $\xi_2=\mathfrak{L}_{x}(\eta_2,y_2)$. The contour of integration consists in a horizontal segment connecting $(\xi_1,y_1)$ with $(\xi_2,y_1)$ plus a vertical segment connecting this point with  $(\xi_2,y_2)$. 
After a direct computation using Taylor's expansion, we can check that
\begin{equation*}
\abs{\frac{\partial \mathfrak{H} }{\partial \xi}}\leq C, \quad \abs{\frac{\partial \mathfrak{H} }{\partial y}}\leq C\abs{\xi},
\end{equation*}
and hence
\begin{align}
\abs{\int_{(\xi_1,y_1)}^{(\xi_2,y_2)}\bigg[  \frac{\partial \mathfrak{H} }{\partial\xi} (\xi,y) \ d\xi + \frac{\partial \mathfrak{H} }{\partial y} (\xi,y) \ dy \bigg]}
& \leq C \left(  \abs{\mathfrak{L}_{x}(\eta_2,y_1)-\mathfrak{L}_{x}(\eta_1,y_1)} + \abs{\mathfrak{L}_{x}(\eta_2,y_2)} \abs{y_{1}-y_{2}} \right)\nonumber \\
& = \mathfrak{I}_{221}+\mathfrak{I}_{222}.
\end{align}
Using estimate \eqref{computation:extra:Lambda} and Assumpion \ref{assumption:Lambda} we can readily check that
\begin{equation}\label{mathfrakI222}
\mathfrak{I}_{222}\leq C \abs{x}\abs{y_{1}-y_{2}}.
\end{equation}
To bound $\mathfrak{I}_{221}$ we proceed as follows. First, notice that for $W(\eta,y)=\frac{\Lambda(\eta,y)}{y}$ we have that
\begin{align*}
\mathfrak{L}_{x}(\eta_1,y_1)-\mathfrak{L}_{x}(\eta_2,y_1) &=\big[W(\eta_{1}+x,y_1)-W(\eta_1,y_1)\big]-\big[W(\eta_{2}+x,y_1)-W(\eta_2,y_1)\big] \\
&=\big[W(\eta_{1}+x,y_1)-W(\eta_1,y_1)\big]-\big[W(\eta_{2}+x,y_1)-W(\eta_2,y_1)\big] \\
&=\int_{\eta_{1}}^{\eta_{1}+x} \frac{\partial W}{\partial \eta}(\xi,y_1) \ d \xi -  \int_{\eta_{2}}^{\eta_{2}+x} \frac{\partial W}{\partial \eta}(\xi,y_1) \ d \xi 
\end{align*}
After changing variables in the second integral, we infer that
\begin{align*}
\mathfrak{L}_{x}(\eta_1,y_1)-\mathfrak{L}_{x}(\eta_2,y_2) &= \bigg[\int_{\eta_{1}}^{\eta_{1}+x} \frac{\partial W}{\partial \eta}(\xi,y_1) \ d \xi -  \int_{\eta_{1}}^{\eta_{1}+x} \frac{\partial W}{\partial \eta}((\eta_{2}-\eta_{1})+\xi,y_1) \ d \xi \bigg] = \mathfrak{M}_1
\end{align*}
Recalling that $W(\eta,y)=\frac{\Lambda(\eta,y)}{y}$ and using Assumption \ref{assumption:Lambda}, we have that $W\in C^{1,\alpha}(\Omega)$. Therefore, we find that
\begin{align*}
\abs{\mathfrak{M}_1} \leq C \int_{\eta_{1}}^{\eta_{1}+x} \abs{\eta_2-\eta_1}^{\alpha} \ d\xi \leq C\abs{x} \abs{\eta_2-\eta_1}^{\alpha} 
\end{align*}
and hence 
\begin{equation}\label{bound:mathfrak:differencia}
\abs{\mathfrak{L}_{x}(\eta_1,y_1)-\mathfrak{L}_{x}(\eta_2,y_1)}\leq C \abs{x}  \abs{\eta_2-\eta_1}^{\alpha}.
\end{equation}
Collecting the  estimates  \eqref{mathfrakI222}, \eqref{bound:mathfrak:differencia} we obtain that
\begin{align*}
\abs{\frac{(e^{-i(\Lambda(\eta_{1}+x,y_{1})-\Lambda(\eta_{1},y_{1}))}-1)}{y_{1}}-\frac{(e^{-i(\Lambda(\eta_{2}+x,y_{2})-\Lambda(\eta_{2},y_{2}))}-1)}{y_{2}}}&\leq \abs{\int_{(\xi_1,y_1)}^{(\xi_2,y_2)}\bigg[  \frac{\partial \mathfrak{H} }{\partial\xi} (\xi,y) \ d\xi + \frac{\partial \mathfrak{H} }{\partial y} (\xi,y) \ dy \bigg]} \\
&\leq C \abs{x}  \abs{\eta_2-\eta_1}^{\alpha},
\end{align*}
which combined with \eqref{mathfrak:I22:estimate} yields the desired estimate for $\mathfrak{I}_{22}$, namely,
\begin{equation}\label{estimate:mathfrak:I22:final}
\mathfrak{I}_{22}\leq C \abs{x}  \abs{\eta_2-\eta_1}^{\alpha}
\end{equation}
Therefore, recalling that $\mathfrak{I}_{2}=\mathfrak{I}_{21}+\mathfrak{I}_{22}$ and bounds \eqref{estimate:mathfrak:I21}, \eqref{estimate:mathfrak:I22:final} we find that
\begin{equation}\label{estimate:mathfrak:I2:final}
\mathfrak{I}_2 \leq C \abs{x} \left(  \abs{y_{1}-y_{2}}+ \abs{\eta_2-\eta_1}^{\alpha}\right).
\end{equation}
The later estimate and \eqref{estimate:mathfrak:I1:final} provide the $\alpha$-H\"older semi-norm bound \eqref{alpha:Holder:seminorm:bound} since
\begin{align*}
\abs{\mathfrak{J}_{x}(\eta_1,y_1)-\mathfrak{J}_{x}(\eta_2,y_2)}\leq C \abs{\mathfrak{N}_{x}(\eta_{1},y_{1})-\mathfrak{N}_{x}(\eta_{2},y_{2})} &= \mathfrak{I}_{1}+ \mathfrak{I}_{2}  \\
&\leq C \abs{x} \left(  \abs{y_{1}-y_{2}}+ \abs{\eta_2-\eta_1}^{\alpha}\right),
\end{align*}
proving bound \eqref{bound:Calpha:mathfrakJ} and concluding the claim.

Hence using the fact that the function $\mathfrak{J}_{x}(\eta,y)$ is $C^\alpha$ H\"older and satisfies bound \eqref{bound:Calpha:mathfrakJ} we can bound the term $M_{52,\epsilon}$ as follows. Adding and subtracting  $\mathfrak{J}_{x}(0,0)$ we have that  
\begin{align*}
 M_{52,\epsilon}&=\widetilde{H}(0,0) \int_{\epsilon}^{L}dy\int_{\mathbb{S}^{1}}d\eta \ \frac{y^2 (e^{-i\eta-y})^3 (e^{-i\Lambda(\eta,y)})^2}{(1-e^{-i\eta-y})(1-e^{-i\eta-y-i\Lambda(\eta,y)})^3} \left( \mathfrak{J}_{x}(\eta,y)-  \mathfrak{J}_{x}(0,0) \right) \\
 \quad &+\widetilde{H}(0,0) \mathfrak{J}_{x}(0,0) \int_{\epsilon}^{L}dy\int_{\mathbb{S}^{1}}d\eta \ \frac{y^2 (e^{-i\eta-y})^3 (e^{-i\Lambda(\eta,y)})^2}{(1-e^{-i\eta-y})(1-e^{-i\eta-y-i\Lambda(\eta,y)})^3}=  M_{521,\epsilon}+ M_{522,\epsilon}.
\end{align*}
Identifying $\SS^1$ with the symmetric interval $I_{x}=[x-\pi,x+\pi]$ for $x\in [-\pi,\pi]$, using bound \eqref{est:calcu} and the $\alpha$-H\"older bound \eqref{bound:Calpha:mathfrakJ} for $ \mathfrak{J}_{x}(\eta,y)$ we obtain
\begin{equation}
\abs{M_{521,\epsilon}}\leq C\abs{x}^{\alpha} \norm{\widetilde{H}}_{L^{\infty}(\Omega)} \int_{\epsilon}^{L}dy\int_{I_{x}}d\eta \frac{y^{2}(\abs{\eta}^{\alpha}+y^{\alpha})}{(\eta^{2}+y^{2})^{2}}\leq C\abs{x}^{\alpha} \norm{\widetilde{H}}_{L^{\infty}(\Omega)}.
\end{equation}
To deal with the second integral $M_{522,\epsilon}$, we notice that we can mimick the ideas that we used to bound the singular term $K_{4,\epsilon}$ in \eqref{K:4:eps:firstpart} which relies on using the decomposition $\Lambda(\eta,y)=yA(\eta)+[\Lambda(\eta,y)-A(\eta)y]$ and bound \eqref{bound:square}. The main diference is that, for $M_{522,\epsilon}$, the $\abs{x}^{\alpha}$ power is obtained using the fact that $\abs{\mathfrak{J}_{x}(0,0)}\leq C\abs{x}^{\alpha}.$ Combining those elements, and closely following the arguments to estimate $K_{4,\epsilon}$ one can find that
\begin{equation}
\abs{M_{522,\epsilon}}\leq C\abs{x}^{\alpha} \norm{\widetilde{H}}_{L^{\infty}(\Omega)}.
\end{equation}
Therefore, since $ M_{52,\epsilon}=M_{521,\epsilon}+M_{522,\epsilon}$ we conclude that 
\begin{equation}\label{M52:eps}
   \abs{ M_{52,\epsilon}}\leq C \abs{x}^{\alpha} \norm{\widetilde{H}}_{L^{\infty}(\Omega)}.
\end{equation}
Estimate \eqref{M51:eps} and \eqref{M52:eps} shows that
\begin{equation}\label{M5:eps}
    \abs{M_{5,\epsilon}}\leq \abs{M_{51,\epsilon}+M_{52,\epsilon}}\leq C \abs{x}^{\alpha} \norm{\widetilde{H}}_{L^{\infty}(\Omega)}.
\end{equation}

Collecting the previous estimates \eqref{M4:eps:final} and \eqref{M5:eps}, and recalling that $M_{3,\epsilon}=M_{4,\epsilon}+M_{5,\epsilon}$ we infer that
\begin{equation}\label{M3eps:quasi:final}
\abs{M_{3,\epsilon}}\leq C \abs{x}^{\alpha} \norm{\widetilde{H}}_{C^{\alpha}(\Omega)}.
\end{equation}
Moreover, using the definition of $\widetilde{H}$ in \eqref{def:H:tilde:R3} and Assumption \ref{assumption:Lambda} we infer that
\begin{equation}\label{tildeH:estimate}
	\norm{\widetilde{H}}_{C^{\alpha}(\Omega)}\leq \norm{\Lambda}_{C^2(\Omega)}\norm{H}_{C^{\alpha}(\Omega)}\leq C\norm{H}_{C^{\alpha}(\Omega)},
\end{equation}
and hence \eqref{M3eps:quasi:final} and \eqref{tildeH:estimate} yield
\begin{equation*}
\abs{M_{3,\epsilon}}\leq C \abs{x}^{\alpha} \norm{H}_{C^{\alpha}(\Omega)},
\end{equation*}
proving claim \eqref{claim:M3:eps}.

Hence, by means of estimates \eqref{estimate:M0:M2:eps} and \eqref{claim:M3:eps} we conclude that
\begin{equation}\label{Xi2:estimate:final:1alpha:eps}
	\abs{\p_{x}\Xi_{2,\epsilon}(x)- \p_{x}\Xi_{2,\epsilon}(0)} =\abs{M_{0,\epsilon}+M_{1,\epsilon}+M_{2,\epsilon}+M_{3,\epsilon}}\leq C \abs{x}^{\alpha} \norm{H}_{C^{1,\alpha}(\Omega)}, \mbox{ for } x\in \SS^1.
\end{equation}
Recalling that 
$$\p_{x}\Xi_{\epsilon}(x)=\p_{x}\Xi_{1,\epsilon}(x)+\p_{x}\Xi_{2,\epsilon}(x)$$
and bounds \eqref{Xi1:estimate:final:1alpha}, \eqref{Xi2:estimate:final:1alpha:eps} we find that
\begin{equation}\label{Xi:final:estimate:1alpha:eps}
	\abs{\p_{x}\Xi_{\epsilon}(x)- \p_{x}\Xi_{\epsilon}(0)} \leq C\abs{x}^{\alpha}\norm{H}_{C^{1,\alpha}(\Omega)}, \mbox{ for } x\in \SS^1
\end{equation}
which shows the desired $C^{1,\alpha}$ semi-norm bound in \eqref{principal:estimate:1alpha:eps}. 

To conclude the proof of the proposition, we claim that the $C^{1,\alpha}$ semi-norm bound \eqref{Xi:final:estimate:1alpha:eps} and the $L^{\infty}$ pointwise bound \eqref{pointwise:bound:version2} implies the  $C^{1,\alpha}$ H\"older norm \eqref{cota:singular:2}. This is, we have to show that we can control the $L^\infty$ norm of $\p_{x}\Xi_{\epsilon}$. Indeed, by means of the identity
\begin{equation}
	\Xi_{\epsilon}(x_{1})= \Xi_{\epsilon}(x_2)+\p_{x}\Xi_{\epsilon}(x_{2})(x_{1}-x_{2}) -\int_{x_1}^{x_2} \left(\p_{x}\Xi_{\epsilon}(x)-\p_{x}\Xi_{\epsilon}(x_{2})\right) dx, \mbox{ for } x_1,x_2\in \SS^1
\end{equation}
and using the $C^{1,\alpha}$ semi-norm bound \eqref{Xi:final:estimate:1alpha:eps} and the $L^{\infty}$ pointwise bound \eqref{pointwise:bound:version2}
\begin{align}
	\abs{\p_{x}\Xi_{\epsilon}(x_{2})} &\leq \frac{1}{2\pi}2\norm{\Xi}_{L^\infty}+\abs{\p_{x}\Xi_{\epsilon}(x)-\p_{x}\Xi_{\epsilon}(x_{2})}  \leq \frac{C}{\pi} \norm{H}_{C^{\alpha}(\Omega)}
	+C(2\pi)^{\alpha}\norm{H}_{C^{1,\alpha}(\Omega)} \label{new:der:Linfty:Xi}
\end{align}
since $\abs{x_{1}-x_{2}}\leq 2\pi$. Therefore, estimate \eqref{new:der:Linfty:Xi} and the previous $C^{1,\alpha}$ semi-norm bound \eqref{Xi:final:estimate:1alpha:eps} implies the desired $C^{1,\alpha}$ H\"older norm estimate \eqref{cota:singular:2}.
\end{proof}
\subsection{\texorpdfstring{$C^\alpha$}{Lg} and \texorpdfstring{$C^{1,\alpha}$}{Lg} estimates for simplified singular integral operators} 
We provide a similar type of estimate for a simplified operator that do not have the dependence on the function $\Lambda$ (and hence is of convolution type) which reads
\begin{proposition}\label{Holder:singular:integral:2}
Let $H(\eta, y)\in C^{1,\alpha}(\Omega)$ and let Assumption \ref{assumption:Lambda} hold. Then for any $x\in \mathbb{S}^{1}$ the following limit exists
\begin{equation} \label{int:sin:2}
	\displaystyle\lim_{\epsilon\to 0^{+}} \widetilde{\Xi}_{\epsilon}(x)=\widetilde{\Xi}(x),
\end{equation} 
where
\begin{equation} \label{int:sin:2:eps}
	\widetilde{\Xi}_{\epsilon}(x)= \int_{\epsilon}^{L}dy\int_{\mathbb{S}^{1}}d\eta \  \p_{x} \left( \frac{e^{i(x-\eta)-y}}{1-e^{i(x-\eta)-y}} \right) 
	H( \eta, y), \ \epsilon>0.
\end{equation}
Moreover, we have that
\begin{align}
	\norm{\widetilde{\Xi}}_{C^{\alpha}(\mathbb{S}^{1})}&\leq C \norm{H}_{C^{\alpha}(\Omega)}, \label{cota:singular:3} \\
	\norm{\widetilde{\Xi}}_{C^{1,\alpha}(\mathbb{S}^{1})}&\leq C \norm{H}_{C^{1,\alpha}(\Omega)}.\label{cota:singular:4}
\end{align}
with $C>0$.
\end{proposition}
\begin{proof}[Proof of Proposition \ref{Holder:singular:integral:2}]
The proof follows closely the ideas of Proposition \ref{Holder:singular:integral:alpha} and Proposition \ref{Holder:singular:integral:1alpha} and the estimates can be shown mimicking the arguments used there. For the sake of completeness, we include the computations to derive bound \eqref{cota:singular:3}, being the $C^{1,\alpha}$ bound \eqref{cota:singular:4} analogous. To show that the left hand side of \eqref{int:sin:2} exists, we first notice that 
\begin{equation} 
	\p_{x} \left( \frac{e^{ix-y}}{1-e^{i(x-\eta)-y}} \right)=-\p_{\eta} \left( \frac{e^{ix-y}}{1-e^{i(x-\eta)-y}} \right).
\end{equation}
Using the fact that $\int_{\mathbb{S}^{1}} d\eta \ \p_{\eta}(\ldots)=0$, we rewrite \eqref{int:sin:2:eps} as
\begin{equation}\label{int:sin:2:eps:2}
	\widetilde{\Xi}_{\epsilon}(x)=- \int_{\epsilon}^{L}dy\int_{\mathbb{S}^{1}}d\eta \  \p_{\eta} \left( \frac{e^{ix-y}}{1-e^{i(x-\eta)-y}} \right) \left( H( \eta, y)-H(x,0) \right).
\end{equation}
Expanding the derivative and using Lemma \ref{calculus:lemma} (for the trivial case of $\Lambda\equiv 0$), where for $x\in [-\pi,\pi]$ we identify $\eta\in \SS^1$ with the symmetric interval $I_{x}=[x-\pi,x+\pi]$ we find that
\begin{align}
	\abs{\widetilde{\Xi}_{\epsilon}(x)} &\leq C \norm{H}_{C^{\alpha}(\Omega)} \displaystyle\lim_{\epsilon\to 0^{+}} \int_{\epsilon}^{L}dy\int_{I_{x}}d\eta \frac{(\abs{x-\eta}^{\alpha}+y^{\alpha})}{(x-\eta)^{2}+y^{2}}\leq C\norm{H}_{C^{\alpha}(\Omega)}. 
\end{align}
Since we are estimating the integrands of \eqref{int:sin:2:eps:2} by integrable functions that are independent on $\epsilon$ it follows from Lebesgue Dominated Convergence Theorem that 
$\lim_{\epsilon\to 0^{+}}\widetilde{\Xi}_{\epsilon}(x)$ exists. Therefore, the limit on the left hand side of \eqref{int:sin:2} exists and the function $\widetilde{\Xi}$ is well-defined.

To obtain the H\"older estimate, we compute the difference $\widetilde{\Xi}_{\epsilon}(x)-\widetilde{\Xi}_{\epsilon}(0)$ and divide the region of integration in the integrals $\int_{\epsilon}^{L} dy \int_{\mathbb{S}^{1}} d\eta (\ldots)$ into sets of the form 
\begin{align*}
	R^{\epsilon}_{\Omega,\leq}&=\{(\eta,y)\in \Omega:\displaystyle\max\{\abs{y},\abs{\eta}\}\leq 2\abs{x} \mbox{ and } \epsilon\leq y\leq L\}, \\
	R^{\epsilon}_{\Omega,>}&=\{(\eta,y)\in \Omega:\displaystyle\max\{\abs{y},\abs{\eta}\}> 2\abs{x} \mbox{ and } \epsilon\leq y\leq L\},
\end{align*} 
as
\begin{align*}
	\widetilde{\Xi}_{\epsilon}(x)-\widetilde{\Xi}_{\epsilon}(0)&= \int_{R^{\epsilon}_{\Omega,\leq}}dy\ d\eta \bigg[\left(\frac{ie^{i(x-\eta)-y}}{(1-e^{i(x-\eta)-y})^2}\right) \left(H( \eta, y)-H(x,0)\right)\\
	&\hspace{1cm} -\left(\frac{ie^{-i\eta-y}}{(1-e^{-i\eta-y})^2}\right) \left(H( \eta, y)-H(0,0)\right) \bigg] \\
	&+  \int_{R^{\epsilon}_{\Omega,\leq}}dy\ d\eta  \bigg[\left(\frac{ie^{i(x-\eta)-y}}{(1-e^{i(x-\eta)-y})^2}\right) \left(H( \eta, y)-H(x,0)\right)\\
	&\hspace{1cm} -\left(\frac{ie^{-i\eta-y}}{(1-e^{-i\eta-y})^2}\right) \left(H( \eta, y)-H(0,0)\right)=D_{1,\epsilon}+D_{2,\epsilon}.
\end{align*}
In the inner region, we have that using Lemma \ref{calculus:lemma} (again in the trivial case of $\Lambda\equiv 0$),
\begin{align}
	\abs{D_{1,\epsilon}} &\leq C \norm{H}_{C^{\alpha}(\Omega)}\int_{R_{\Omega,\leq}}dy\ d\eta \bigg[ \frac{(\abs{x-\eta}^{\alpha}+y^{\alpha})}{(x-\eta)^{2}+y^{2}}+\frac{(\abs{\eta}^{\alpha}+y^{\alpha})}{\eta^{2}+y^{2}}\bigg] \nonumber  \\
	&\leq C \norm{H}_{C^{\alpha}(\Omega)}\abs{x}^{\alpha} \label{estimate:M_1}.
\end{align}
In the outer region we rewrite the term $D_{2,\epsilon}$ as follows
\begin{align*}
	D_{2,\epsilon}&=\int_{R^{\epsilon}_{\Omega,>}} dy \ d\eta \bigg[\left(\frac{ie^{i(x-\eta)-y}}{(1-e^{i(x-\eta)-y})^2}\right)-\left(\frac{ie^{-i\eta-y}}{(1-e^{-i\eta-y})^2}\right)\bigg] \left(H( \eta, y)-H(0,0)\right) \\
	& \quad -\left(H(x, 0)-H(0,0)\right)\int_{R{\epsilon}_{\Omega,>}} dy \ d\eta \left(\frac{ie^{i(x-\eta)-y}}{(1-e^{i(x-\eta)-y})^2}\right)=D_{21,\epsilon}+D_{22,\epsilon}.
\end{align*}
The first integral $D_{21,\epsilon}$ can be bounded using the mean value theorem as
\begin{align}
	\abs{D_{21,\epsilon}} &\leq C\abs{x} \norm{H}_{C^{\alpha}(\Omega)}\int_{R_{\Omega,>}} dy \ d\eta \bigg[  \frac{(\abs{\eta}^{\alpha}+\abs{y}^{\alpha})}{\left((x-\eta)^{2}+y^{2}\right)^{3/2}}\bigg] \leq C\abs{x}^{\alpha} \norm{H}_{C^{\alpha}(\Omega)}. \label{estimate:M_21}
\end{align}
To deal with $D_{22,\epsilon}$, we make the following change of variables $\widetilde{\eta}=x-\eta$ to find that
\begin{align*}
	D_{22,\epsilon}&=-i\left(H(x, 0)-H(0,0)\right)\displaystyle\lim_{\epsilon\to 0^{+}}\bigg[\int_{-\frac{1}{2}}^{\frac{1}{2}} d\widetilde{\eta} \int_{2\abs{x}}^{L} dy \frac{e^{i\widetilde{\eta}-y}}{(1-e^{i\widetilde{\eta}-y})^2}+\int_{-\frac{1}{2}}^{-2\abs{x}}d\widetilde{\eta}  \int_{\epsilon}^{L} dy \frac{e^{i\widetilde{\eta}-y}}{(1-e^{i\widetilde{\eta}-y})^2} \\
	&\hspace{6cm} +\int_{2\abs{x}}^{\frac{1}{2}}d\widetilde{\eta}  \int_{\epsilon}^{L} dy\frac{e^{i\widetilde{\eta}-y}}{(1-e^{i\widetilde{\eta}-y})^2}   \bigg] \\
	&=-i\left(H(x, 0)-H(0,0)\right)\displaystyle\lim_{\epsilon\to 0^{+}}\bigg[ \int_{-\frac{1}{2}}^{\frac{1}{2}}d\widetilde{\eta} \left( \frac{e^{i\widetilde{\eta}}}{e^{i\widetilde{\eta}}-e^{L}}-\frac{e^{i\widetilde{\eta}}}{e^{i\widetilde{\eta}}-e^{2\abs{x}}} \right) +\int_{-\frac{1}{2}}^{-2\abs{x}}d\widetilde{\eta} \left( \frac{e^{i\widetilde{\eta}}}{e^{i\widetilde{\eta}}-e^{L}}-\frac{e^{i\widetilde{\eta}}}{e^{i\widetilde{\eta}}-e^{\epsilon}} \right) \\
	&\hspace{6cm}+ \int_{2\abs{x}}^{\frac{1}{2}}d\widetilde{\eta} \left( \frac{e^{i\widetilde{\eta}}}{e^{i\widetilde{\eta}}-e^{L}}-\frac{e^{i\widetilde{\eta}}}{e^{i\widetilde{\eta}}-e^{\epsilon}} \right)  \bigg] \\
	&=\left(H(x, 0)-H(0,0)\right) \displaystyle\lim_{\epsilon\to 0^{+}} \left( \bigg[\displaystyle\log(\frac{e^{2\abs{x}}-e^{i\widetilde{\eta}}}{e^{L}-e^{i\widetilde{\eta}}})\bigg]^{\widetilde{\eta}=\frac{1}{2}}_{\widetilde{\eta}=-\frac{1}{2}}+\bigg[\displaystyle\log(\frac{e^{\epsilon}-e^{i\widetilde{\eta}}}{e^{L}-e^{i\widetilde{\eta}}})\bigg]^{\widetilde{\eta}=-2\abs{x}}_{\widetilde{\eta}=-\frac{1}{2}}+\bigg[\displaystyle\log(\frac{e^{\epsilon}-e^{i\widetilde{\eta}}}{e^{L}-e^{i\widetilde{\eta}}})\bigg]^{\widetilde{\eta}=2\abs{x}}_{\widetilde{\eta}=\frac{1}{2}} \right)
\end{align*}
Therefore we have that
\begin{equation}\label{estimate:M_22}
	\abs{D_{22,\epsilon}}\leq C  \abs{x}^{\alpha}\norm{H}_{C^{\alpha}(\Omega)}
\end{equation}
which together with \eqref{estimate:M_21} yields
\begin{equation}\label{estimate:M_2:final}
	\abs{D_{2,\epsilon}}\leq C  \abs{x}^{\alpha}\norm{H}_{C^{\alpha}(\Omega)}
\end{equation}
Combining estimates \eqref{estimate:M_1}-\eqref{estimate:M_2:final} we infer that
\begin{equation}
	\abs{\widetilde{\Xi}(x)-\widetilde{\Xi}(0)} \leq \displaystyle\lim_{\epsilon\to 0^{+}} \abs{D_{1,\epsilon}+D_{2,\epsilon}} \leq C \abs{x}^{\alpha}\norm{H}_{C^{\alpha}(\Omega)} 
\end{equation}
which shows the desired $C^\alpha$ semi-norm bound \eqref{cota:singular:3}.
\end{proof}

To conclude this subsection, let us state the following result which will be needed later to show the contracting estimates in Subsection \ref{sec:4:3}.
\begin{proposition}[Difference $C^\alpha$ estimate]\label{Holder:lemma:difference}
Let $H(\eta, y)\in C^{\alpha}(\Omega)$ and let $\Lambda^1,\Lambda^2$ satisfy Assumption \ref{assumption:Lambda}. Then for any $x\in \mathbb{S}^{1}$ the following limit exists
\begin{equation} \label{int:sin:1:diff}
	\displaystyle\lim_{\epsilon\to 0^{+}} \Xi^{d}_{\epsilon}(x)=\Xi^{d}(x),
\end{equation} 
where
\begin{equation} \label{int:sin:diff:eps}
	\Xi^{d}_{\epsilon}(x)=\displaystyle\lim_{\epsilon\to 0^{+}} \int_{\epsilon}^{L}dy\int_{\mathbb{S}^{1}}d\eta \  \p_{x} \left(\frac{e^{i(x-\eta)-y}y}{(1-e^{i(x-\eta)-y-i\Lambda^1(\eta,y)})(1-e^{i(x-\eta)-y-i\Lambda^2(\eta,y)})}\right) H( \eta, y), \ \epsilon>0.
\end{equation}
Moreover, we have that
\begin{align}\label{cota:singular:diff}
	\norm{\Xi^{d}}_{C^{\alpha}(\mathbb{S}^{1})}&\leq C \norm{H}_{C^{\alpha}(\Omega)}  
\end{align}
with $C>0$.
\end{proposition}
\begin{proof}[Proof of Proposition \ref{Holder:lemma:difference}]
This result can be proved by a means an elementary adaption of the Proposition \ref{Holder:singular:integral:alpha}. Notice that the only difference between this result and Proposition \ref{Holder:singular:integral:alpha} is that instead of a single function $\Lambda$ appearing as a perturbation in the denominator we have 
two different functions $\Lambda^1,\Lambda^2$ affecting the denominator in \eqref{int:sin:diff:eps}. 
Actually, Proposition \ref{Holder:singular:integral:alpha} is a particular case in which we take $\Lambda^1=0 $ and $\Lambda^2=\Lambda$. The proof of this Proposition \ref{Holder:lemma:difference} can be showed around similar lines of Proposition \ref{Holder:singular:integral:alpha} just estimating the corrective terms due to $\Lambda^1,\Lambda^2$  as it was made in the proof of Proposition \ref{Holder:singular:integral:alpha}. We will not provide the details here.
\end{proof}

\section{The a priori estimates for the operators \texorpdfstring{$\mathsf{T}_{1},\ldots,\mathsf{T}_{4}$}{Lg}.}\label{sec:4:2}
In this section, we will provide the $C^\alpha$ and $C^{1,\alpha}$ H\"older estimates for the operators $\mathsf{T}_{1},\ldots,\mathsf{T}_{4}$ and the function $\mathsf{G}$. The key point towards the estimates relies on the $C^\alpha$ and $C^{1,\alpha}$ H\"older bounds showed in the previous Section \ref{S4}. 

In order to do so, let us introduce the following new assumption:

\begin{Assumption}\label{assumption:Theta}
We assume that the function $\vartheta:\Omega \to \mathbb{S}^{1}$ belongs to $C^{1,\alpha}(\Omega)$ and also that there exists $\delta_1\in(0,\frac{1}{2})$ such that
\begin{equation*}
	\norm{\vartheta}_{C^{1,\alpha}(\Omega)}\leq \delta_1.
\end{equation*}
\end{Assumption}
\begin{remark}
The new function $\vartheta(\eta,y)$ will play the role of the function $\partial_{\xi} \Theta(X(\eta,y),y))$ in the operators $\mathsf{T}_{1},\ldots,\mathsf{T}_{4}$. However, using this notation reduces the length of the formulas.
\end{remark}

\begin{remark}\label{remark:operator:T}
In the following we will define the operators $\mathsf{T}_{1},\ldots,\mathsf{T}_{4}$ by means of certain integral expressions which represent operators from $C^\alpha$ to $C^\alpha$ and from $C^{1,\alpha}$ to $C^{1,\alpha}$. For the sake of simplicity, we will use the same symbol to denote these operators, in spite of the fact that they act in different spaces. The space on which they act will be clear in each particular case from the context.
\end{remark}

\begin{proposition}[Estimates $\mathsf{T}_{1}$] \label{Proposition:estimate:T1}
Let Assumption \ref{assumption:Lambda} and Assumption \ref{assumption:Theta} hold. Then for $j_0\in C^{1,\alpha}(\SS^1)$ we define the operator $\mathsf{T}_{1}$ as follows
\begin{equation}\label{def:T1:second}
	\mathsf{T}_{1}[\Lambda,\vartheta]j_{0}(x)=-\frac{1}{2\pi} \displaystyle \lim_{\epsilon\to 0^{+}}\int_{\mathbb{S}^{1}} \mathfrak{G}_{1,\epsilon}(x-\eta,\eta)j_{0}(\eta) \ d\eta, \mbox{ for } x\in \SS^1,
\end{equation}
where
\begin{equation}\label{def:G1:second}
	\mathfrak{G}_{1,\epsilon}(x,\eta)=\displaystyle\sum_{n=-\infty}^{n=\infty}\frac{n\sinh(nL)}{(\cosh(nL)-1)}e^{inx}\int_{\epsilon}^{L} e^{-\abs{n}y}  \left[ \frac{e^{-in\Lambda(\eta,y)}-1}{1+\vartheta(\eta,y)} \right]  dy.
\end{equation}
The limit in \eqref{def:T1:second} exists, moreover we have that 
\begin{align}
	\norm{\mathsf{T}_{1}[\Lambda,\vartheta]}_{\mathcal{L}(C^{\alpha}(\mathbb{S}^1)) }&\leq C\delta_{0}  \label{estimate:T1:Calpha}, \\ 
	\norm{\mathsf{T}_{1}[\Lambda,\vartheta]}_{\mathcal{L}(C^{1,\alpha}(\mathbb{S}^1))}&\leq C\delta_{0} \label{estimate:T1:C1alpha}.
\end{align}
\end{proposition}

\begin{proof}
To show that the right hand side of \eqref{def:T1:second} exists, let us split the function $\mathfrak{G}_{1,\epsilon}(x,\eta)$ for $n>0$ and $n<0$. We will use the notation $\mathfrak{G}^{+}_{1,\epsilon}$ for $n>0$ and $\mathfrak{G}^{-}_{1,\epsilon}$ otherwise. Therefore, for $n>0$ we have that 
\begin{equation}\label{operadorG1+:n:m}
	\mathfrak{G}^{+}_{1,\epsilon}(x,\eta)=\displaystyle\sum_{n=1}^{n=\infty}n e^{inx}\int_{\epsilon}^{L} e^{-ny}  \left[ \frac{e^{-in\Lambda(\eta,y)}-1}{1+\vartheta(\eta,y)} \right]  dy+ \mathfrak{G}_{12,\epsilon}^{+}(x,\eta)
\end{equation}
where
\begin{equation}\label{operadorG1+:n:s}
	\mathfrak{G}_{12,\epsilon}^{+}(x,\eta)=\displaystyle\sum_{n=1}^{n=\infty}\left(\frac{n\sinh(nL)}{(\cosh(nL)-1)}-n\right)e^{inx}\int_{\epsilon}^{L} e^{-ny}  \left[ \frac{e^{-in\Lambda(\eta,y)}-1}{1+\vartheta(\eta,y)} \right]  dy.
\end{equation}
On the one hand, calculating explicitly the summation via geometric series, we have that the first term in \eqref{operadorG1+:n:m} denoted by $ \mathfrak{G}^{+}_{11,\epsilon}$ is given by 
\begin{align}
	\mathfrak{G}^{+}_{11,\epsilon}(x,\eta)& = \frac{1}{i}\int_{\epsilon}^{L}\p_{x}\left(\frac{e^{ix-y-i\Lambda(\eta,y)}}{1-e^{ix-y-i\Lambda(\eta,y)}}-\frac{e^{ix-y}}{1-e^{ix-y}}\right)\frac{1}{1+\vartheta(\eta,y)}  \  dy \nonumber \\
	&=\frac{1}{i}\int_{\epsilon}^{L} \frac{(e^{-i\Lambda(\eta,y)}-1)}{y} \p_{x} \left(\frac{e^{ix-y}y}{(1-e^{ix-y})(1-e^{ix-y-i\Lambda(\eta,y)})}\right)\frac{1}{1+\vartheta(\eta,y)} \ dy  \nonumber \\
	&=\frac{1}{i}\int_{\epsilon}^{L} \mathsf{F}(x,\eta,y) \frac{(e^{-i\Lambda(\eta,y)}-1)}{y} \frac{1}{1+\vartheta(\eta,y)} \ dy  \label{G11+:formula}
\end{align} 
with 
\begin{equation*}
	\mathsf{F}(x,\eta,y)=\p_{x} \left(\frac{e^{ix-y}y}{(1-e^{ix-y})(1-e^{ix-y-i\Lambda(\eta,y)})}\right).
\end{equation*}
Therefore, by means of \eqref{G11+:formula}, we have that the operator $\mathsf{T}_{1}$ for $n>0$ given in \eqref{def:T1:second} can be written after changing the order of integration as
\begin{align}
	\mathsf{T}^{+}_{1}[\Lambda,\vartheta]j_{0}(x)&=-\frac{1}{2\pi i} \displaystyle \lim_{\epsilon\to 0} \int_{\epsilon}^{L}dy \int_{\mathbb{S}^{1}} d\eta \ \mathsf{F}(x-\eta,\eta,y) \frac{(e^{-i\Lambda(\eta,y)}-1)}{y} \frac{1}{1+\vartheta(\eta,y)}j_{0}(\eta) \nonumber \\
	& - \frac{1}{2\pi}\displaystyle \lim_{\epsilon\to 0}\int_{\mathbb{S}^{1}}\mathfrak{G}^{+}_{12,\epsilon}(x-\eta,\eta)j_{0}(\eta) \ d\eta= \mathsf{T}^{+}_{11}[\Lambda,\vartheta]j_{0}(x)+ \mathsf{T}^{+}_{12}[\Lambda,\vartheta]j_{0}(x) \label{T11:T12}
\end{align}
assuming that the limits exist. To show that the limit of the second integral $\mathsf{T}^{+}_{12}j_{0}(x)$ exists, we notice that the function $\mathfrak{G}^{+}_{12}(x,\eta)$ given in \eqref{operadorG1+:n:s} is a smooth function in $x$ and decays exponentially in $n$. Therefore, using the bound
\begin{equation*}
	\abs{e^{-in\Lambda(\eta,y)}-1}\leq C \abs{n}\norm{\Lambda}_{L^\infty(\Omega)}, 
\end{equation*}
we find that
\begin{equation*}
	\abs{\mathfrak{G}^{+}_{12,\epsilon}(x,\eta)}\leq C\norm{\frac{1}{1+\vartheta}}_{L^{\infty}(\Omega)}  \norm{\Lambda}_{L^{\infty}(\Omega)}\leq C, \mbox{ for } x\in \SS^1,\eta\in \SS^1.
\end{equation*}
Since the estimate is independent of $\epsilon$ using the Lebesgue Dominated Convergence Theorem, we can ensure that the limit exists and that the operator $\mathsf{T}^{+}_{12}[\Lambda,\vartheta]j_{0}(x)$ is well defined. Moreover, we also have the pointwise bound
\begin{equation}
	\abs{\mathsf{T}^{+}_{12}[\Lambda,\vartheta]j_{0}(x)}\leq C\norm{\frac{1}{1+\vartheta}}_{L^{\infty}(\Omega)}  \norm{\Lambda}_{L^{\infty}(\Omega)}\norm{j_0}_{L^{\infty}(\SS^1)} \leq C\delta_{0}\norm{j_0}_{L^{\infty}(\SS^1)}, \mbox{ for } x\in \SS^1.
\end{equation}
On the other hand, to ensure that $\mathsf{T}^{+}_{11}j_{0}(x)$ in \eqref{T11:T12} is well defined we make use of Proposition \ref{Holder:singular:integral:alpha} where
\begin{equation}\label{selection:H1}
	H(\eta,y)= \frac{(e^{-i\Lambda(\eta,y)}-1)}{y} \frac{1}{1+\vartheta(\eta,y)} j_{0}(\eta).
\end{equation}
It is straightforward to check that choosing $H(\eta,y)$ as in \eqref{selection:H1}, we have that $H(\eta,y)\in C^{\alpha}(\Omega)$ since
\begin{align*}
	\norm{H}_{C^{\alpha}(\Omega)} &\leq C \norm{\frac{1}{1+\vartheta}}_{C^{\alpha}(\Omega)}\norm{\Lambda}_{C^{1,\alpha}(\Omega)}\norm{j_0}_{C^{\alpha}(\mathbb{S}^1)} \leq C\delta_{0} \norm{j_0}_{C^{\alpha}(\mathbb{S}^1)}
\end{align*}
where in the last inequality we have used Assumption \ref{assumption:Lambda}.
To show the $C^{\alpha}$ and $C^{1,\alpha}$ semi-norm estimate \eqref{estimate:T1:Calpha} and \eqref{estimate:T1:C1alpha}, we check that using estimate
\begin{align}
	\abs{e^{i nx_{1}}-e^{in x_2}} &\leq C \abs{n}^\alpha \abs{x_{1}-x_{2}}^{\alpha}, \mbox{ for } \alpha\in(0,1) \mbox{ and } x_{1},x_{2}\in \mathbb{S}^{1} \label{formula:holder:exp}
\end{align}
we find that
\begin{align*}
	\abs{\mathfrak{G}^{+}_{12,\epsilon}(x_{1},\eta)-\mathfrak{G}^{+}_{12,\epsilon}(x_{2},\eta)}&\leq C\norm{\frac{1}{1+\vartheta}}_{L^{\infty}(\Omega)}  \norm{\Lambda}_{L^{\infty}(\Omega)} \abs{x_{1}-x_{2}}^{\alpha} \\
	\abs{\p_{x}\mathfrak{G}^{+}_{12,\epsilon}(x_{1},\eta)-\p_{x}\mathfrak{G}^{+}_{12,\epsilon}(x_{2},\eta)}&\leq C\norm{\frac{1}{1+\vartheta}}_{L^{\infty}(\Omega)}  \norm{\Lambda}_{L^{\infty}(\Omega)} \abs{x_{1}-x_{2}}^{\alpha}.
\end{align*}
Thus
\begin{align}
	\norm{\mathsf{T}^{+}_{12}[\Lambda,\vartheta]}_{\mathcal{L}(L^{\infty}(\mathbb{S}^{1}),C^{\alpha}(\mathbb{S}^1))} &\leq C\norm{\frac{1}{1+\vartheta}}_{L^{\infty}(\Omega)}\norm{\Lambda}_{L^{\infty}(\Omega)}\leq C\delta_{0},\label{bounds:T1+sa} \\
	\norm{\mathsf{T}^{+}_{12}[\Lambda,\vartheta]}_{\mathcal{L}(L^{\infty}(\mathbb{S}^{1}),C^{1,\alpha}(\mathbb{S}^1))} &\leq C\norm{\frac{1}{1+\vartheta}}_{L^{\infty}(\Omega)}\norm{\Lambda}_{L^{\infty}(\Omega)}\leq C\delta_{0}.\label{bounds:T1+s1a}
\end{align}

Similarly, as before a direct application of Proposition \ref{Holder:singular:integral:alpha} and Proposition \ref{Holder:singular:integral:1alpha} with $H(\eta,y)$ as in \eqref{selection:H1} yields
\begin{align}
	\norm{\mathsf{T}^{+}_{11}[\Lambda,\vartheta]j_{0}}_{\mathcal{L}(C^{\alpha}(\mathbb{S}^1))}  &\leq C \norm{\frac{1}{1+\vartheta}}_{C^{\alpha}(\Omega)}\norm{\Lambda}_{C^{1,\alpha}(\Omega)} \leq C\delta_{0}, \label{estimate:T11+m:a} \\
	\norm{\mathsf{T}^{+}_{11}[\Lambda,\vartheta]j_{0}}_{\mathcal{L}(C^{1,\alpha}(\mathbb{S}^1))}  &\leq C \norm{\frac{1}{1+\vartheta}}_{C^{1,\alpha}(\Omega)}\norm{\Lambda}_{C^{2,\alpha}(\Omega)} \leq C\delta_{0}, \label{estimate:T11+m:1a}
\end{align}
since
\begin{align*}
	\norm{H}_{C^{\alpha}(\Omega)} &\leq C \norm{\frac{1}{1+\vartheta}}_{C^{\alpha}(\Omega)}\norm{\Lambda}_{C^{1,\alpha}(\Omega)}\norm{j_0}_{C^{\alpha}(\mathbb{S}^1)} \leq C \delta_{0}\norm{j_0}_{C^{\alpha}(\mathbb{S}^1)}, \\
	\norm{H}_{C^{1,\alpha}(\Omega)} &\leq C \norm{\frac{1}{1+\vartheta}}_{C^{1,\alpha}(\Omega)}\norm{\Lambda}_{C^{2,\alpha}(\Omega)}\norm{j_0}_{C^{1,\alpha}(\mathbb{S}^1)} \leq C\delta_{0} \norm{j_0}_{C^{1,\alpha}(\mathbb{S}^1)}.
\end{align*}
Since the estimates for $n<0$ follow identically, we omit a detailed proof here. 
\end{proof}

\begin{proposition}[Estimates $\mathsf{T}_{2}$]\label{Proposition:estimate:T2}
Let Assumption \ref{assumption:Theta} hold. Then for $j_0\in C^{1,\alpha}(\SS^1)$ we define the operator $\mathsf{T}_{2}$ as follows
\begin{equation}\label{def:T2:second}
	\mathsf{T}_{2}[\vartheta]j_{0}(x)=-\frac{1}{2\pi}\displaystyle\lim_{\epsilon\to 0^{+}}\int_{\mathbb{S}^{1}} \mathfrak{G}_{2,\epsilon}(x-\eta,\eta)j_{0}(\eta) \ d\eta,
\end{equation}
where 
\begin{equation}\label{operadorG2:frak}
	\mathfrak{G}_{2,\epsilon}(x,\eta)=\displaystyle\sum_{n=-\infty}^{n=\infty}\frac{n\sinh(nL)}{(\cosh(nL)-1)}e^{inx}\int_{\epsilon}^{L} e^{-\abs{n}y} \left[ \frac{\vartheta(\eta,y)}{1+\vartheta(\eta,y)} \right]  dy.
\end{equation}
The limit \eqref{def:T2:second} exists and in addition the following estimates hold
\begin{align}
	\norm{\mathsf{T}_{2}[\vartheta]}_{\mathcal{L}(C^{\alpha}(\mathbb{S}^1)) }&\leq C\delta_{1}, \label{estimate:T2:Calpha}  \\
	\norm{\mathsf{T}_{2}[\vartheta]}_{\mathcal{L}(C^{1,\alpha}(\mathbb{S}^1))}&\leq C\delta_{1}. \label{estimate:T2:C1alpha}
\end{align}
\end{proposition}


\begin{proof}
The proof follows the same lines of Proposition \ref{Proposition:estimate:T1}. We first, show that the right hand side of \eqref{def:T2:second} exists, and afterwards we provide the H\"older bounds \eqref{estimate:T2:Calpha} and \eqref{estimate:T2:C1alpha}. To that purpose, we split $\mathfrak{G}_{2,\epsilon}(x,\eta)$ for $n>0$ and $n<0$ and just show the estimates for $n>0$, being the case for $n<0$ identical. Therefore, for $n>0$ we have that
\begin{equation}\label{operadorG2+:n:m}
	\mathfrak{G}^{+}_{2,\epsilon}(x,\eta)=\sum_{n=1}^{n=\infty}n e^{inx}\int_{\epsilon}^{L} e^{-ny} \left[ \frac{\vartheta(\eta,y)}{1+\vartheta(\eta,y)} \right]  dy+ \mathfrak{G}_{22,\epsilon}^{+}(x,\eta)
\end{equation}
where
\begin{equation}\label{operadorG2+:n:s}
	\mathfrak{G}_{22,\epsilon}^{+}(x,\eta)=\sum_{n=1}^{n=\infty}\left(\frac{n\sinh(nL)}{(\cosh(nL)-1)}-n\right)e^{inx}\int_{\epsilon}^{L} e^{-ny} \left[ \frac{\vartheta(\eta,y)}{1+\vartheta(\eta,y)} \right]  dy.
\end{equation}
Computing the sum in the first term in \eqref{operadorG2+:n:m} we obtain    
\begin{equation}
	\mathfrak{G}^{+}_{21,\epsilon}(x,\eta)= \frac{1}{i}\int_{\epsilon}^{L}\p_{x}\left(\frac{1}{1-e^{ix-y}}\right)\left[ \frac{\vartheta(\eta,y)}{1+\vartheta(\eta,y)} \right]  dy,   \label{finalG+2} 
\end{equation}
and changing the order of integration the operator $\mathsf{T}_{2}[\vartheta]j_{0}$ in \eqref{def:T2:second} can be written for $n>0$ as 
\begin{align*}
	\mathsf{T}^{+}_{2}[\vartheta]j_{0}(x)&=-\frac{1}{2\pi i}\displaystyle\lim_{\epsilon\to 0} \int_{\epsilon}^{L}dy \int_{\mathbb{S}^{1}}d\eta \  \p_{x}\left(\frac{e^{i(x-\eta)-y}}{1-e^{i(x-\eta)-y}}\right) \left[ \frac{\vartheta(\eta,y)}{1+\vartheta(\eta,y)} \right]  j_{0}(\eta) \\
	& \quad - \frac{1}{2\pi}\displaystyle\lim_{\epsilon\to 0}\int_{\mathbb{S}^{1}}\mathfrak{G}^{+}_{22,\epsilon}(x-\eta,\eta)j_{0}(\eta) \ d\eta:=\mathsf{T}^{+}_{21}[\vartheta]j_0+\mathsf{T}^{+}_{22}[\vartheta]j_0.
\end{align*}
As before, the remainder smoothing term $\mathfrak{G}^{+}_{22}(x,\eta)$ \eqref{operadorG2+:n:s} is a smooth function in $x$ and in particular
\begin{equation}\label{L:inftyG22}
	\abs{\mathfrak{G}^{+}_{22,\epsilon}(x,\eta)}\leq C\norm{\frac{\vartheta}{1+\vartheta} }_{L^{\infty}(\Omega)}\leq C\delta_{1}, \mbox{ for } x\in\SS^1, \eta\in \SS^1. 
\end{equation}
Hence, since estimate  \eqref{L:inftyG22} is independent of $\epsilon$, Lebesgue Dominated Convergence Theorem shows that the limit exists and that the associated operator $\mathsf{T}^{+}_{22}[\vartheta]j_0$ is well-defined. Choosing 
\begin{equation}\label{selection:H2}
	H(\eta,y)=\frac{\vartheta(\eta,y)}{1+\vartheta(\eta,y)}j_{0}(\eta),
\end{equation}
and noticing that 
\begin{equation}
	\norm{H}_{C^{\alpha}(\Omega)} \leq C \norm{\frac{\vartheta}{1+\vartheta}}_{C^{\alpha}(\Omega)}\norm{j_0}_{C^{\alpha}(\mathbb{S}^1)} \leq C\delta_{1} \norm{j_0}_{C^{\alpha}(\mathbb{S}^1)}.
\end{equation}
we can apply Lemma \ref{Holder:singular:integral:2} to obtain that  $\mathsf{T}^{+}_{21}[\vartheta]j_0$ is well defined. We are left to show the $C^\alpha$ and $C^{1,\alpha}$ semi-norm estimates \eqref{estimate:T2:Calpha}-\eqref{estimate:T2:C1alpha}. Making use of the bound
\begin{align*}
	\abs{e^{i nx_{1}}-e^{in x_2}} &\leq C \abs{n}^\alpha \abs{x_{1}-x_{2}}^{\alpha}, \mbox{ for } \alpha\in(0,1) \mbox{ and } x_{1},x_{2}\in \mathbb{S}^{1}.
\end{align*}
we infer that for $\mathfrak{G}^{+}_{22}$ defined in \eqref{operadorG2+:n:s} 
\begin{align*}
	\abs{\mathfrak{G}^{+}_{22,\epsilon}(x_{1},\eta)-\mathfrak{G}^{+}_{22,\epsilon}(x_{2},\eta)}&\leq C\norm{\frac{\vartheta}{1+\vartheta} }_{L^{\infty}(\Omega)} \abs{x_{1}-x_{2}}^{\alpha}, \\ 
	\abs{\p_{x}\mathfrak{G}^{+}_{22,\epsilon}(x_{1},\eta)-\p_{x}\mathfrak{G}^{+}_{22,\epsilon}(x_{2},\eta)}&\leq C\norm{\frac{\vartheta}{1+\vartheta} }_{L^{\infty}(\Omega)} \abs{x_{1}-x_{2}}^{\alpha}
\end{align*}
holds. Thus, due to Assumption \ref{assumption:Theta} we obtain
\begin{align}
	\norm{\mathsf{T}^{+}_{22}[\vartheta]}_{\mathcal{L}(L^{\infty}(\mathbb{S}^{1}),C^{\alpha}(\mathbb{S}^1))} &\leq C\norm{\frac{\vartheta}{1+\vartheta}}_{L^{\infty}(\Omega)}\leq C\delta_{1},\label{bounds:T2+s:a} \\
	\norm{\mathsf{T}^{+}_{22}[\vartheta]}_{\mathcal{L}(L^{\infty}(\mathbb{S}^{1}),C^{1,\alpha}(\mathbb{S}^1))} &\leq C\norm{\frac{\vartheta}{1+\vartheta}}_{L^{\infty}(\Omega)}\leq C\delta_{1}. \label{bounds:T2+s:1:a}
\end{align}
To deal with the most singular operator $\mathsf{T}_{21}^{+}[\vartheta]$ we make use of Proposition \ref{Holder:singular:integral:2}. Indeed, applying estimates \eqref{cota:singular:3} and \eqref{cota:singular:4} to the function $H(\eta,y)$ as in \eqref{selection:H2} we find that
\begin{align}
	\norm{\mathsf{T}^{+}_{21}[\vartheta]}_{\mathcal{L}(C^{\alpha}(\mathbb{S}^1))}  &\leq C \norm{\frac{\vartheta}{1+\vartheta}}_{C^{\alpha}(\Omega)} \leq C\delta_{1}, \label{estimate:T21+:a} \\
	\norm{\mathsf{T}^{+}_{21}[\vartheta]}_{\mathcal{L}(C^{1,\alpha}(\mathbb{S}^1))}  &\leq C \norm{\frac{\vartheta}{1+\vartheta}}_{C^{1,\alpha}(\Omega)}\leq C\delta_{1},  \label{estimate:T21+:1a}
\end{align}
since
\begin{align*}
	\norm{H}_{C^{\alpha}(\Omega)} &\leq C \norm{\frac{\vartheta}{1+\vartheta}}_{C^{\alpha}(\Omega)}\norm{j_0}_{C^{\alpha}(\mathbb{S}^1)} \leq C\delta_{1} \norm{j_0}_{C^{\alpha}(\mathbb{S}^1)}, \\
	\norm{H}_{C^{1,\alpha}(\Omega)} &\leq C \norm{\frac{\vartheta}{1+\vartheta}}_{C^{1,\alpha}(\Omega)}\norm{j_0}_{C^{1,\alpha}(\mathbb{S}^1)} \leq C\delta_{1} \norm{j_0}_{C^{1,\alpha}(\mathbb{S}^1)}.
\end{align*}
Therefore, using \eqref{bounds:T2+s:a} and \eqref{estimate:T21+:a} we conclude that
\begin{equation}\label{estimate:T2:alpha}
	\norm{\mathsf{T}^{+}_{2}[\vartheta}_{\mathcal{L}(C^{\alpha}(\mathbb{S}^1))} \leq C\delta_{1},
\end{equation}
and similarly invoking \eqref{bounds:T2+s:1:a} and \eqref{estimate:T21+:1a} we find that
\begin{equation}\label{estimate:T2:1alpha}
	\norm{\mathsf{T}^{+}_{2}[\vartheta]}_{\mathcal{L}(C^{1,\alpha}(\mathbb{S}^1))}  \leq C\delta_{1}.
\end{equation}
\end{proof}


\begin{proposition}[Estimates $\mathsf{T}_{3}$ and $\mathsf{T}_{4}$]\label{prop:T3T3}
Let Assumption \ref{assumption:Lambda} and Assumption \ref{assumption:Theta} hold. Then for $j_0\in C^{1,\alpha}(\SS^1)$ we define the operators $\mathsf{T}_{3}$ and $\mathsf{T}_{4}$ as follows
\begin{align}
	\mathsf{T}_{3}[\Lambda,\vartheta]j_{0}(x)=-\frac{1}{2\pi}\displaystyle\int_{\mathbb{S}^{1}} \mathfrak{G}_{3}(x-\eta,\eta)j_{0}(\eta) \ d\eta,  \label{def:T3:sec}\\
	\mathsf{T}_{4}[\Lambda,\vartheta]j_{0}(x)=-\frac{1}{2\pi}
	\int_{\mathbb{S}^{1}} \mathfrak{G}_{4}(x-\eta,\eta)j_{0}(\eta) \ d\eta.\label{def:T4:sec}  
\end{align}
with 
\begin{align}
	\mathfrak{G}_{3}(x,\eta)&=\displaystyle\sum_{n=-\infty}^{n=\infty}\frac{n\sinh(nL)}{(\cosh(nL)-1)}e^{inx}\int_{0}^{L} M(n,y) \frac{(e^{-in\Lambda(\eta,y)}-1)}{1+\vartheta(\eta,y)}  dy,  \\
	\mathfrak{G}_{4}(x,\eta)&=\displaystyle\sum_{n=-\infty}^{n=\infty} \frac{n\sinh(nL)}{(\cosh(nL)-1)} e^{inx} \int_{0}^{L} M(n,y)  \frac{\vartheta(\eta,y)}{1+\vartheta(\eta,y)} dy, 
\end{align}
and 
\begin{equation}\label{function:M}
	M(n,y)=\frac{e^{-2nL}\left(e^{ny}-e^{-ny} \right)}{(1-e^{-2nL})}.
\end{equation}
Then we have that
\begin{align}
	\norm{\mathsf{T}_{3}[\Lambda,\vartheta]}_{\mathcal{L}(C^{\alpha}(\mathbb{S}^1),L^\infty(\SS^1))}\leq C\delta_{0},  \label{estimate:prop:T3:Calpha} \\
	\norm{\mathsf{T}_{4}[\Lambda,\vartheta]}_{\mathcal{L}(C^{\alpha}(\mathbb{S}^1),L^\infty(\SS^1))}\leq C\delta_{1},\label{estimate:prop:T4:Calpha}
\end{align}
and
\begin{align}
	\norm{\mathsf{T}_{3}[\Lambda,\vartheta]}_{\mathcal{L}(C^{1,\alpha}(\mathbb{S}^1),L^\infty(\SS^1))}\leq C, \label{estimate:prop:T3:C1alpha}  \\
	\norm{\mathsf{T}_{4}[\Lambda,\vartheta]}_{\mathcal{L}(C^{1,\alpha}(\mathbb{S}^1),L^\infty(\SS^1))}\leq C, \label{estimate:prop:T4:C1alpha}   
\end{align}
\end{proposition}
\begin{remark}
As we can see from the estimates the operators, $\mathsf{T}_{3}$ and $\mathsf{T}_{4}$ are smoothing operators that transform functions from $L^\infty(\SS^1)$ to functions in $C^{1,\alpha}(\SS^1)$. Moreover, the series in $\mathfrak{G}_{3}$ and $\mathfrak{G}_{4}$ are uniformly convergent and therefore we do not need to define the operators $\mathsf{T}_{3}$ and $\mathsf{T}_{4}$ in \eqref{def:T3:sec}, \eqref{def:T4:sec} as a limit since the integrals are well defined.
\end{remark}
\begin{proof}
Notice that $M(n,y)$ defined in \eqref{function:M} is a smooth function in $y$ and decays exponential in $n$. Therefore, the simple bounds
\begin{align*}
	\abs{M(n,y)}\leq C e^{-nL}, \quad \abs{e^{-in\Lambda(\eta,y)}-1}\leq C \abs{n} \norm{\Lambda}_{L^{\infty}(\Omega)},
\end{align*}
yield
\begin{align*}
	\abs{\mathfrak{G}_{3}(x,\eta)}&\leq C \norm{\Lambda}_{L^{\infty}(\Omega)} \norm{\frac{1}{1+\vartheta} }_{L^{\infty}(\Omega)}  \mbox{ and }
	\abs{\mathfrak{G}_{4}(x,\eta)}\leq C \norm{\frac{\vartheta}{1+\vartheta}}_{L^{\infty}}, \mbox{ for } x\in \SS^1,\eta\in \SS^1.
\end{align*}
Moreover using that $\abs{e^{i nx_{1}}-e^{in x_2}} \leq C \abs{n}^\alpha \abs{x_{1}-x_{2}}^{\alpha}$ for $\alpha\in(0,1)$ and $x_{1},x_{2}\in \mathbb{S}^{1}$ we have that for $\eta\in \SS^1$
\begin{align*}
	\abs{\mathfrak{G}_{3}(x_{1},\eta)-\mathfrak{G}_{3}(x_{2},\eta)}&\leq C \norm{\Lambda}_{L^{\infty}(\Omega)}\norm{\frac{1}{1+\vartheta}}_{L^{\infty}(\Omega)}\abs{x_{1}-x_{2}}^{\alpha}\leq C\delta_{0}\abs{x_{1}-x_{2}}^{\alpha}, \\
	\abs{\mathfrak{G}_{4}(x_{1},\eta)-\mathfrak{G}_{4}(x_{2},\eta)}&\leq C \norm{\frac{\vartheta}{1+\vartheta}}_{L^{\infty}(\Omega)}\abs{x_{1}-x_{2}}^{\alpha}\leq C\delta_{1}\abs{x_{1}-x_{2}}^{\alpha}, \\
	\abs{\partial_{x}\mathfrak{G}_{3}(x_{1},\eta)-\partial_{x}\mathfrak{G}_{3}(x_{2},\eta)}&\leq C \norm{\Lambda}_{L^{\infty}(\Omega)} \norm{\frac{1}{1+\vartheta}}_{L^{\infty}(\Omega)}\abs{x_{1}-x_{2}}^{\alpha}\leq C\delta_{0}\abs{x_{1}-x_{2}}^{\alpha}, \\
	\abs{\partial_{x}\mathfrak{G}_{4}(x_{1},\eta)-\partial_{x}\mathfrak{G}_{4}(x_{2},\eta)}&\leq C \norm{\frac{\vartheta}{1+\vartheta}}_{L^{\infty}(\Omega)}\abs{x_{1}-x_{2}}^{\alpha}\leq C\delta_{1}\abs{x_{1}-x_{2}}^{\alpha}.
\end{align*}
Therefore using the above pointwise estimates we conclude that the H\"older semi-norm of $\mathsf{T}_{3}$, $\mathsf{T}_{4}$ defined in \eqref{def:T3:sec}-\eqref{def:T4:sec} is bounded as 
\begin{align*}
	\norm{\mathsf{T}_{3}[\Lambda,\vartheta]}_{\mathcal{L}(L^\infty(\SS^1),C^{\alpha}(\mathbb{S}^1))}\leq  C\delta_{0},  \\
	\norm{\mathsf{T}_{4}[\Lambda,\vartheta]}_{\mathcal{L}(L^\infty(\SS^1),C^{\alpha}(\mathbb{S}^1))}\leq C\delta_{1}, \\
	\norm{\mathsf{T}_{3}[\Lambda,\vartheta]}_{\mathcal{L}(L^\infty(\SS^1),C^{1,\alpha}(\mathbb{S}^1))}\leq C\delta_{0},  \\
	\norm{\mathsf{T}_{4}[\Lambda,\vartheta]}_{\mathcal{L}(L^\infty(\SS^1),C^{1,\alpha}(\mathbb{S}^1))}\leq C\delta_{1},  
\end{align*}
concluding the proof.
\end{proof}

\subsection{Estimates for the differences of \texorpdfstring{$\mathsf{T}_{j}$}{Lg}}\label{sec:4:3}
In this subsection, we will derive estimates for the difference operators. This will be needed in order to show the contraction estimate in the general fixed point argument (see Section \ref{sec:6}) in the lower order H\"older space $C^{\alpha}$. The proof follows the same lines as in the previous subsection  but some extra computations are needed in order to get the desired contraction type estimate. More precisely, our first result reads
\begin{proposition}\label{prop:diff:T1}
Let $\Lambda^1,\Lambda^2$ satisfy Assumption \ref{assumption:Lambda} and $\vartheta^1,\vartheta^2$ satisfy Assumption \ref{assumption:Theta}. Let $j_{0}\in C^{\alpha}(\SS^1)$ and define 
\begin{align}
	\mathsf{T}^{d}_{1}j_{0}(x)=\left(\mathsf{T}_{1}[\Lambda^1,\vartheta^1]-\mathsf{T}_{1}[\Lambda^2,\vartheta^2]\right)j_{0}(x), \label{difference:operator:T1} 
\end{align}
where $\mathsf{T}_{1}[\cdot,\cdot]$ is given in \eqref{def:T1:second}. Then we have that
\begin{align}
	\norm{\mathsf{T}_{1}^{d}}_{\mathcal{L}(C^{\alpha}(\mathbb{S}^1))}\leq C\left( \norm{\Lambda^1-\Lambda^2}_{C^{1,\alpha}(\Omega)}+ \norm{\vartheta^2-\vartheta^1}_{C^{\alpha}(\Omega)}\right). \label{estimate:difference:T1:Calpha} 
\end{align}
\end{proposition}
\begin{proof}
Let us first introduce some notation. We define by $\mathfrak{G}^{d}_{1,\epsilon}$ the difference function
\begin{align}
	\mathfrak{G}^{d}_{1,\epsilon}(x,\eta)=\mathfrak{G}_{1,\epsilon}[\Lambda^1,\vartheta^1](x,\eta)-\mathfrak{G}_{1,\epsilon}[\Lambda^2,\vartheta^2](x,\eta)  \label{difference:fucntion:G1}    
\end{align}
where $\mathfrak{G}_{1,\epsilon}[\cdot,\cdot]$ is given in \eqref{def:G1:second}. Following the lines of Proposition \ref{Proposition:estimate:T1}, we split the function \eqref{def:G1:second} for $n>0$ and $n<0$. Similarly as in the previous propositions, we will use the notation $\mathfrak{G}^{d,+}_{1,\epsilon}$ for $n>0$ and $\mathfrak{G}^{d,-}_{1,\epsilon}$ otherwise.Then, for $n>0$ we find that
\begin{align}
	\mathfrak{G}^{d,+}_{1,\epsilon}(x,\eta)&=\displaystyle\sum_{n=1}^{n=\infty}n e^{inx}\int_{\epsilon}^{L} e^{-ny}  \mathsf{L}^{n}(\eta,y) dy \nonumber + \displaystyle\sum_{n=1}^{n=\infty}\left(\frac{n\sinh(nL)}{(\cosh(nL)-1)}-n\right)e^{inx}\int_{\epsilon}^{L} e^{-ny}   \mathsf{L}^{n}(\eta,y)  dy \nonumber \\
	&=\mathfrak{G}^{d,+}_{11,\epsilon}(x,\eta)+\mathfrak{G}^{d,+}_{12,\epsilon}(x,\eta) \label{splitting:difference:T1}
\end{align}
with
\begin{equation}
	\mathsf{L}^{n}(\eta,y) =\left[ \frac{e^{-in\Lambda^{1}(\eta,y)}-1}{1+\vartheta^1(\eta,y)} - \frac{e^{-in\Lambda^2(\eta,y)}-1}{1+\vartheta^2(\eta,y)} \right].
\end{equation}
It will be convenient to decompose $\mathsf{L}^{n}(\eta,y)$ as follows
\begin{align}
	\mathsf{L}^{n}(\eta,y) &=\frac{1}{1+\vartheta^1(\eta,y)}\left(e^{-in\Lambda^1(\eta,y)}-e^{-in\Lambda^2(\eta,y)}\right) + \left(e^{-in\Lambda^2(\eta,y)}-1\right) \left(\frac{1}{1+\vartheta^1(\eta,y)}-\frac{1}{1+\vartheta^2(\eta,y)}\right) \nonumber \\
	&= \mathsf{L}^{n}_{1}+\mathsf{L}^{n}_{2}. \label{definition:L1:L2}
\end{align}
Plugging the decomposition $\mathsf{L}^{n}=\mathsf{L}^{n}_1+\mathsf{L}^{n}_2$  into \eqref{splitting:difference:T1} and calculating the sum in $n$ for the first term $\mathfrak{G}^{d}_{11,\epsilon}(x,\eta)$ we obtain
\begin{align*}
	\mathfrak{G}^{d,+}_{11,\epsilon}(x,\eta)&= \frac{1}{i}\int_{\epsilon}^{L}\p_{x}\left(\frac{e^{ix-y-i\Lambda^1(\eta,y)}}{1-e^{ix-y-i\Lambda^1(\eta,y)}}-\frac{e^{ix-y-i\Lambda^2(\eta,y)}}{1-e^{ix-y-i\Lambda^2(\eta,y)}}\right)\frac{1}{1+\vartheta^1(\eta,y)}  \  dy \nonumber \\
	&\quad + \frac{1}{i}\int_{\epsilon}^{L}\p_{x}\left(\frac{e^{ix-y-i\Lambda^2(\eta,y)}}{1-e^{ix-y-i\Lambda^2(\eta,y)}}-\frac{e^{ix-y}}{1-e^{ix-y}}\right)\times \left(\frac{1}{1+\vartheta^1(\eta,y)}-\frac{1}{1+\vartheta^2(\eta,y)}\right)  dy \\
	&= \mathfrak{G}^{d,+}_{111,\epsilon}(x,\eta)+\mathfrak{G}^{d,+}_{112,\epsilon}(x,\eta).
\end{align*}
Denoting by
\begin{align*}
	\mathsf{F}_{1}(x,\eta,y)&= \p_{x}\left(\frac{ye^{ix-y}}{\left(1-e^{ix-y-i\Lambda^1(\eta,y)}\right)\left(1-e^{ix-y-i\Lambda^2(\eta,y)}\right)} \right), \\
	\mathsf{F}_{2}(x,\eta,y)&=\p_{x}\left(\frac{ye^{ix-y}}{\left(1-e^{ix-y-i\Lambda^2(\eta,y)}\right)\left(1-e^{ix-y}\right)} \right),
\end{align*}
we obtain that 
\begin{align*}
	\mathfrak{G}^{d,+}_{111,\epsilon}(x,\eta)= \frac{1}{i}\int_{\epsilon}^{L} \frac{\mathsf{L}^{1}_{1}(\eta,y)}{y}\mathsf{F}_{1}(x,\eta,y) dy \mbox{ and } \mathfrak{G}^{d}_{112,\epsilon}(x,\eta)= \frac{1}{i}\int_{\epsilon}^{L} \frac{\mathsf{L}^{1}_{2}(\eta,y)}{y}\mathsf{F}_{2}(x,\eta,y)dy.
\end{align*}
Therefore, recalling the definition of the operator $\mathsf{T}^{d}_{1}$ in \eqref{difference:operator:T1} and the previous computations we find, after changing the order of integration that
\begin{align*}
	\mathsf{T}^{d,+}_{1}j_{0}(x)&= \mathsf{T}^{d,+}_{11}j_{0}(x)+ \mathsf{T}^{d,+}_{12}j_{0}(x)+ \mathsf{T}^{d,+}_{13}j_{0}(x)
\end{align*} 
where
\begin{align}
	\mathsf{T}^{d,+}_{11}j_{0}(x)&=  -\frac{1}{2\pi i} \displaystyle \lim_{\epsilon\to 0} \int_{\epsilon}^{L}dy \int_{\mathbb{S}^{1}} d\eta \frac{\mathsf{L}^{1}_{1}(\eta,y)}{y} \mathsf{F}_{1}(x-\eta,\eta,y) j_{0}(\eta), \label{def:Td:11} \\
	\mathsf{T}^{d,+}_{12}j_{0}(x)&=  -\frac{1}{2\pi i} \displaystyle \lim_{\epsilon\to 0} \int_{\epsilon}^{L}dy \int_{\mathbb{S}^{1}} d\eta \frac{\mathsf{L}^{1}_{2}(\eta,y)}{y}\mathsf{F}_{2}(x-\eta,\eta,y) j_{0}(\eta), \label{def:Td:12} \\
	\mathsf{T}^{d,+}_{13}j_{0}(x)&= -\frac{1}{2\pi} \displaystyle \lim_{\epsilon\to 0}\int_{\mathbb{S}^{1}} \mathfrak{G}^{d,+}_{12,\epsilon}(x-\eta,\eta)j_{0}(\eta) \ d\eta. \label{def:Td:13}
\end{align} 
We start estimating the third term $\mathsf{T}^{d}_{13}j_{0}(x)$. Notice that the function $\mathfrak{G}^{d}_{12,\epsilon}$ in \eqref{splitting:difference:T1} is a smooth function in $x$ due to the exponential decay in $n$ of the terms that define the function \eqref{splitting:difference:T1}. Furthermore, we also have that the functions $\mathsf{L}^{n}_{1}(\eta,y),\mathsf{L}^{n}_{2}(\eta,y)$ defined in \eqref{definition:L1:L2} are bounded by
\begin{align}
	\abs{\mathsf{L}^{n}_{1}(\eta,y)}&\leq C \norm{\frac{1}{1+\vartheta^1}}_{L^{\infty}(\Omega)} \abs{n} \norm{\Lambda^1-\Lambda^2}_{L^\infty(\Omega)}, \mbox{ for } (\eta,y)\in\Omega, \label{estimate:L1} \\
	\abs{\mathsf{L}^{n}_{2}(\eta,y)}&\leq C \norm{\frac{\vartheta^2-\vartheta^1}{(1+\vartheta^2)(1+\vartheta^1)}}_{L^{\infty}(\Omega)} \abs{n}\norm{\Lambda^2}_{L^\infty(\Omega)}, \mbox{ for } (\eta,y)\in\Omega. \label{estimate:L2}
\end{align}
and thus
\begin{align}
	\abs{\mathfrak{G}^{d,+}_{12,\epsilon}(x,\eta)}&\leq C \left(\norm{\frac{1}{1+\vartheta^1}}_{L^{\infty}(\Omega)}\norm{\Lambda^1-\Lambda^2}_{L^\infty(\Omega)}+\norm{\frac{\vartheta^2-\vartheta^1}{(1+\vartheta^2)(1+\vartheta^1)}}_{L^{\infty}(\Omega)}\norm{\Lambda^2}_{L^\infty(\Omega)} \right)\nonumber \\
	&\leq C \norm{\Lambda^1-\Lambda^2}_{L^\infty(\Omega)}+ \norm{\vartheta^2-\vartheta^1}_{L^{\infty}(\Omega)}, \mbox{ for } x\in\SS^1,\eta\in\SS^1. \label{estimate:G12d:Linfty}
\end{align}
Since the estimate \eqref{estimate:G12d:Linfty} is independent of $\epsilon$, Lebesgue Dominated Convergence Theorem shows that the limit exists and that the associated operator $\mathsf{T}^{d}_{13}j_0$ in \eqref{def:Td:13} is well-defined. Moreover, we have the pointwise estimate
\begin{equation}
	\abs{\mathsf{T}^{d,+}_{13}j_{0}(x)}\leq C  \left(\norm{\Lambda^1-\Lambda^2}_{L^\infty(\Omega)}+ \norm{\vartheta^2-\vartheta^1}_{L^{\infty}(\Omega)}\right)\norm{j_0}_{L^{\infty}(\Omega)}, \quad x\in\SS^1. \label{estimate:Td:13:Linf}
\end{equation}
The $C^{\alpha}$ semi-norm estimate for $\mathsf{T}^{d,+}_{13}j_0$ follows directly by using estimate \eqref{formula:holder:exp} which combined with \eqref{estimate:Td:13:Linf} shows 
\begin{align}
	\norm{\mathsf{T}^{d,+}_{13}}_{\mathcal{L}(L^\infty(\SS^1),C^{\alpha}(\mathbb{S}^1))}\leq C\left(\norm{\Lambda^1-\Lambda^2}_{L^\infty(\Omega)}+\norm{\vartheta^2-\vartheta^1}_{L^{\infty}(\Omega)}\right).  \label{bound:T+13:diff}
\end{align}
To deal with $\mathsf{T}^{d,+}_{11}j_0$ in \eqref{def:Td:11} and $\mathsf{T}^{d,+}_{12}j_0$  in \eqref{def:Td:12}, we can invoke Proposition \ref{Holder:lemma:difference} and Proposition \ref{Holder:singular:integral:alpha}, respectively. Notice that a consequence of the before mentioned lemmas is that the limits of the integrals \eqref{def:Td:11} and \eqref{def:Td:12} are well-defined.

To that purpose, we first notice using the expression \eqref{definition:L1:L2} that
\begin{align}
	\frac{\mathsf{L}^{1}_{1}(\eta,y)}{y}=\frac{e^{-i\Lambda^{1}(\eta,y)}}{1+\vartheta^1(\eta,y)}\frac{\left(1-e^{-(i\Lambda^{2}(\eta,y)-i\Lambda^{1}(\eta,y))} \right)}{y}=\frac{e^{-i\Lambda^{1}(\eta,y)}}{1+\vartheta^1(\eta,y)}i\int_{0}^{1}B(\eta,y)e^{iysB(\eta,y)} ds \label{exp:L1:integral}
\end{align}
since by Assumption \ref{assumption:Lambda} we can write $\Lambda^{2}(\eta,y)-\Lambda^{1}(\eta,y)=yB(\eta,y)$ for $B(\eta,y)\in C^{1,\alpha}(\Omega).$
Thus, choosing
\begin{equation*}
	H(\eta,y)=\frac{\mathsf{L}^{1}_{1}(\eta,y)}{y} j_{0}(\eta)\in C^{\alpha}(\Omega)
\end{equation*}
we apply estimate \eqref{cota:singular:diff} in Proposition \ref{Holder:lemma:difference} combined with \eqref{exp:L1:integral} to obtain that
\begin{equation}\label{estimate:Td11}
	\norm{\mathsf{T}^{d,+}_{11}}_{\mathcal{L}(C^{\alpha}(\mathbb{S}^1))}\leq  C \norm{\Lambda^1-\Lambda^2}_{C^{1,\alpha}(\Omega)}.
\end{equation}
On the other hand, proceeding in a similar way and choosing
\begin{equation*}
	H(\eta,y)=\frac{\mathsf{L}^{1}_{2}(\eta,y)}{y} j_{0}(\eta)\in C^{\alpha}(\Omega)
\end{equation*}
we use Lemma \ref{Holder:singular:integral:alpha} to show that
\begin{equation}\label{estimate:Td12}
	\norm{\mathsf{T}^{d,+}_{12}}_{\mathcal{L}(C^{\alpha}(\mathbb{S}^1))}\leq C \norm{\vartheta^2-\vartheta^1}_{C^{\alpha}(\Omega)}.
\end{equation}

Collecting bounds \eqref{bound:T+13:diff}-\eqref{estimate:Td12} we deduce estimate \eqref{estimate:difference:T1:Calpha} for $n>0$. Repeating the same arguments for $n<0$ concludes the proof.
\end{proof}

\begin{proposition}\label{prop:diff:T2}
Let $\vartheta^1,\vartheta^2$ satisfy Assumption \ref{assumption:Theta}. Let $j_0\in C^{\alpha}(\SS^1)$ and define 
\begin{align}
	\mathsf{T}^{d}_{2}j_{0}(x)=\left(\mathsf{T}_{2}[\vartheta^1]-\mathsf{T}_{2}[\vartheta^2]\right)j_{0}(x), \label{difference:operator:T2} 
\end{align}
where $\mathsf{T}_{2}[\cdot]$ is given in \eqref{def:T2:second}. Then we have that
\begin{align}
	\norm{\mathsf{T}_{2}^{d}}_{\mathcal{L}(C^{\alpha}(\mathbb{S}^1))}\leq C \norm{\vartheta^2-\vartheta^1}_{C^{\alpha}(\Omega)}. \label{estimate:difference:T2:Calpha} 
\end{align}
\end{proposition}

\begin{proof}
Following the same notation as in the last proposition and the arguments of the proof of Proposition \ref{Proposition:estimate:T2}, we introduce the difference function $\mathfrak{G}^{d}_{2,\epsilon}$
\begin{align}
	\mathfrak{G}^{d}_{2,\epsilon}(x,\eta)=\mathfrak{G}_{2,\epsilon}[\vartheta^1](x,\eta)-\mathfrak{G}_{2,\epsilon}[\vartheta^2](x,\eta).  \label{difference:fucntion:G2}    
\end{align}
where $\mathfrak{G}_{2,\epsilon}[\cdot]$ is given in \eqref{operadorG2:frak}. The using the same notation by means of the superscripts $\pm$ as before, we find that for $n>0$ we have that
\begin{align}
	\mathfrak{G}^{d,+}_{2,\epsilon}(x,\eta)&=\displaystyle\sum_{n=1}^{n=\infty}n e^{inx}\int_{\epsilon}^{L} e^{-ny}  \mathsf{L}_{3}(\eta,y) dy \nonumber + \displaystyle\sum_{n=1}^{n=\infty}\left(\frac{n\sinh(nL)}{(\cosh(nL)-1)}-n\right)e^{inx}\int_{\epsilon}^{L} e^{-ny}  \mathsf{L}_{3}(\eta,y)  dy \nonumber \\
	&=\mathfrak{G}^{d,+}_{21,\epsilon}(x,\eta)+\mathfrak{G}^{d,+}_{22,\epsilon}(x,\eta) \label{splitting:difference:T2}
\end{align}
with
\begin{equation*}
	\mathsf{L}_3(\eta,y) =\frac{1}{1+\vartheta^1(\eta,y)} (\vartheta^1(\eta,y)-\vartheta^2(\eta,y))+\vartheta^{2}(\eta,y)\left( \frac{1}{1+\vartheta^1(\eta,y)}-\frac{1}{1+\vartheta^2(\eta,y)}\right).  
\end{equation*}
Therefore, calculating the sum in $n$ in $\mathfrak{G}^{d}_{21,\epsilon}(x,\eta)$ we find that 
\begin{align*}
	\mathfrak{G}^{d,+}_{21,\epsilon}(x,\eta)= \frac{1}{i}\int_{\epsilon}^{L} \p_{x}\left(\frac{e^{ix-y}}{1-e^{ix-y}}\right)\mathsf{L}_3(\eta,y) dy 
\end{align*}
and thus
\begin{align*}
	\mathsf{T}^{d,+}_{2}j_{0}(x)&= \mathsf{T}^{d,+}_{21}j_{0}(x)+ \mathsf{T}^{d,+}_{22}j_{0}(x)
\end{align*} 
where
\begin{align}
	\mathsf{T}^{d,+}_{21}j_{0}(x)&=  -\frac{1}{2\pi} \displaystyle \lim_{\epsilon\to 0^{+}} \int_{\epsilon}^{L}dy \int_{\mathbb{S}^{1}} d\eta \p_{x}\left(\frac{e^{ix-y}}{1-e^{i(x-\eta)-y}}\right)\mathsf{L}_3(\eta,y) j_{0}(\eta), \label{def:Td:21} \\
	\mathsf{T}^{d,+}_{22}j_{0}(x)&= -\frac{1}{2\pi} \displaystyle \lim_{\epsilon\to 0^{+}}\int_{\mathbb{S}^{1}} \mathfrak{G}^{d,+}_{22}(x-\eta,\eta)j_{0}(\eta) \ d\eta. \label{def:Td:22}
\end{align} 
Since the function $\mathfrak{G}^{d,+}_{22}$ is smooth in $x$ since the series that defines it decays exponentially in $n$, we infer using the expression on $\mathsf{L}_3(\eta,y)$  that
\begin{equation}
	\abs{\mathfrak{G}^{d,+}_{22,\epsilon}(x,\eta)}\leq C \norm{\mathsf{L}_{3}}_{L^\infty(\Omega)}\leq C \norm{\vartheta^1-\vartheta^2}_{L^{\infty}(\Omega)}, \mbox{ for } x\in\SS^1, \eta\in \SS^1.
\end{equation}
Since estimate is independent of $\epsilon$, Lebesgue Dominated Convergence Theorem shows that the limit exists and that the associated operator $\mathsf{T}^{d,+}_{22}j_0$ in \eqref{def:Td:22} is well-defined. Moreover, we have the pointwise estimate
\begin{equation}
	\abs{\mathsf{T}^{d,+}_{22}j_{0}(x)}\leq C   \norm{\vartheta^1-\vartheta^2}_{L^{\infty}(\Omega)}\norm{j_0}_{L^{\infty}(\Omega)}, \quad x\in\SS^1. \label{estimate:Td:22:Linf}
\end{equation}
The H\"older semi-norm estimate follows in a similar way as in \eqref{formula:holder:exp}. Therefore, 
\begin{equation}\label{bound:T21:diff}
	\norm{\mathsf{T}^{d,+}_{22}}_{\mathcal{L}(L^\infty(\SS^1),C^{\alpha}(\mathbb{S}^1))}\leq  C \norm{\vartheta^2-\vartheta^1}_{C^\alpha(\Omega)}.
\end{equation}
To estimate the remainder operator $\mathsf{T}^{d}_{21}j_0$ we apply Proposition \ref{Holder:singular:integral:2} with $H(\eta,y)=\mathsf{L}_{3}j_{0}(\eta)\in C^{\alpha}(\Omega)$ and hence
\begin{equation}\label{bound:T22:diff}
	\norm{\mathsf{T}^{d,+}_{21}}_{\mathcal{L}(C^{\alpha}(\mathbb{S}^1))}\leq  C \norm{\vartheta^2-\vartheta^1}_{C^\alpha(\Omega)}.
\end{equation}
By means of \eqref{bound:T21:diff}-\eqref{bound:T22:diff} we deduce that \eqref{estimate:difference:T2:Calpha} concluding the proof.
\end{proof}

\begin{proposition}\label{proposition:T2:diff}
Let $\Lambda^1,\Lambda^2$ satisfy Assumption \ref{assumption:Lambda} and let $\vartheta^1,\vartheta^2$ satisfy Assumption \ref{assumption:Theta}. Let $j_0\in C^{\alpha}(\SS^1)$ and define 
\begin{align}
	\mathsf{T}^{d}_{3}j_{0}(x)&=\left(\mathsf{T}_{3}[\Lambda^1,\vartheta^1]-\mathsf{T}_{2}[\Lambda^2,\vartheta^2]\right)j_{0}(x), \label{difference:operator:T3}  \\
	\mathsf{T}^{d}_{4}j_{0}(x)&=\left(\mathsf{T}_{4}[\vartheta^1]-\mathsf{T}_{4}[\vartheta^2]\right)j_{0}(x),
\end{align}
where $\mathsf{T}_{3}[\cdot,\cdot]$ and $\mathsf{T}_{4}[\cdot]$ are given in \eqref{def:T3:sec}-\eqref{def:T4:sec}. Then we have that
\begin{align}
	\norm{\mathsf{T}_{3}^{d}}_{\mathcal{L}(L^\infty(\Omega),C^{\alpha}(\mathbb{S}^1))}&\leq C\left(\norm{\Lambda^1-\Lambda^2}_{C^{1,\alpha}(\Omega)}+ \norm{\vartheta^2-\vartheta^1}_{C^{\alpha}(\Omega)}\right), \label{estimate:difference:T3:Calpha}  \\
	\norm{\mathsf{T}_{4}^{d}}_{\mathcal{L}(L^\infty(\Omega),C^{\alpha}(\mathbb{S}^1))}&\leq C  \norm{\vartheta^2-\vartheta^1}_{C^{\alpha}(\Omega)} \label{estimate:difference:T4:Calpha}  \end{align}
\end{proposition}

\begin{proof}
The proof follows simply by combining the bounds derived in Proposition \ref{prop:diff:T1}, Proposition \ref{prop:diff:T2}  and Proposition \ref{prop:T3T3}.
\end{proof}

\subsection{Estimates for \texorpdfstring{$\mathsf{G}$}{Lg}}\label{Sec:44}
In this subsection, we provide the $C^{1,\alpha}$ H\"older estimate for the term $\mathsf{G}$ defined in \eqref{func:G:rig}. To that purpose, let us start with the following lemma:
\begin{lemma}\label{Lemma:G:1}
Let $f\in C^{2,\alpha}(\partial\Omega)$ and define the functions
\begin{align*}
	h^{+}(x)=-J+\int_{0}^{x} (f(\xi,L)-A) \ d\xi, \quad 
	h^{-}(x)=\int_{0}^{x} (f(\xi,0)-A) \ d\xi, \mbox{ for }  x\in \SS^1.
\end{align*}
Denote by $\widehat{h^{+}}(n),\widehat{h^{-}}(n)$ the Fourier coefficients of $h^{+},h^{-}$ respectively. We define
\begin{equation}\label{definition:widetilde:f}
	\mathcal{Z}(x)=\frac{1}{2\pi} \displaystyle\sum_{n=-\infty}^{n=\infty}\left( \widehat{h^{+}}(n) \frac{\abs{n}}{\sinh (\abs{n}L)}-
	\widehat{h^{-}}(n)\frac{\abs{n}}{\tanh (\abs{n}L)}\right)  e^{inx}.
\end{equation}
Then the function $\mathcal{Z}\in C^{2,\alpha}(\SS^1)$. Moreover, we have that
\begin{equation}
	\norm{\mathcal{Z}}_{C^{2,\alpha}(\SS^1)}\leq C \norm{f}_{C^{2,\alpha}(\SS^1)}.
\end{equation}
\end{lemma}
\begin{proof}
Adding and subtracting $\abs{n}\widehat{h^{-}}(n)e^{inx}$ in \eqref{definition:widetilde:f} we obtain that
\begin{equation}
	\mathcal{Z}(x)=\frac{1}{2\pi} \displaystyle\sum_{n=-\infty}^{n=\infty}\left( \abs{n}\widehat{h^{-}}(n)e^{inx}+W_{1}(n)\widehat{h^{+}}(n)e^{inx}+W_{2}(n)\widehat{h^{-}}(n)e^{inx}\right)
\end{equation}
where $\abs{W_{i}(n)}\leq C e^{-Ln}$ for $i=1,2$. Therefore, recalling that 
$$\partial_{x}\mathcal{H}h^{-}(x)=\mathcal{H}\partial_{x}h^{-}(x)=\frac{1}{2\pi} \displaystyle\sum_{n=-\infty}^{n=\infty} \abs{n}\widehat{h^{-}}(n)e^{inx}$$
where $\mathcal{H}$ is the periodic Hilbert transform we arrive at
\begin{align*}
	\norm{\mathcal{Z}}_{C^{2,\alpha}(\SS^1)} &\leq C \left(\norm{ \mathcal{H}\partial_{x}h^{-}}_{C^{2,\alpha}(\SS^1)} + \norm{h^-}_{L^{\infty}(\SS^1)}+\norm{h^+}_{L^{\infty}(\SS^1)} \right)\\
	&\leq C \left( \norm{ \mathcal{H}h^{-}}_{C^{3,\alpha}(\SS^1)}+ \norm{h^+}_{L^{\infty}(\SS^1)} \right).
\end{align*}
Moreover, using the fact that the Hilbert transform is a bounded operator in the class of H\"older spaces (cf. \cite{Mushk}) namely 
$$\norm{\mathcal{H}}_{\mathcal{L}(C^{3,\alpha}(\SS^1))}\leq C,$$
and that $\partial_{x}h^{-}(x)=(f(x,L)-A)\in C^{2,\alpha}(\SS^1)$ we obtain
\begin{align*}
	\norm{\mathcal{Z}}_{C^{2,\alpha}(\SS^1)} &\leq C \left( \norm{ \mathcal{H}h^{-}}_{C^{3,\alpha}(\SS^1)}+ \norm{h^+}_{L^{\infty}(\SS^1)} \right)\leq C  \norm{f}_{C^{2,\alpha}(\SS^1)}.
\end{align*}
\end{proof}

\begin{lemma}\label{Lemma:G:12}
Let $g\in C^{2,\alpha}(\partial\Omega_{-})$ and define $\widetilde{g}(x)=g(x)-\mathcal{Z}(x)$ where $\mathcal{Z}(x)$ is given in \eqref{definition:widetilde:f}. Then for $x\in\SS^1$ we define \begin{equation}\label{definition:G:boundary}
	\mathsf{G}(x)=-\frac{1}{2\pi}\int_{\mathbb{S}^{1}} \displaystyle\sum_{n=-\infty}^{n=\infty}\left[\frac{n \sinh(nL)}{\cosh(nL)-1}\right]e^{in(x-\eta)}\widetilde{g}(\eta) d\eta.
\end{equation} 
Moreover, we have that
\begin{equation}\label{bound:G:C1alpha}
	\norm{\mathsf{G}}_{C^{1,\alpha}(\SS^1)}\leq C \left( \norm{f}_{C^{2,\alpha}(\SS^1)}+ \norm{g}_{C^{2,\alpha}(\SS^1)}\right).
\end{equation}
\end{lemma}
\begin{proof}
Recalling \eqref{func:G:rig}, we have that
\begin{align*}
	\mathsf{G}(x)&= \mathcal{H}\p_{x}\widetilde{g}(x)- \int_{\SS^1}\widetilde{Q}(x-\eta)\widetilde{g}(\eta) \ d\eta,
\end{align*}
where $$\widetilde{Q}(x)=\frac{1}{2\pi}\displaystyle\sum_{n=-\infty}^{n=\infty}\widetilde{Q}_{n}e^{inx}$$ and 
$$\widetilde{Q}_{n}=\frac{1}{n}\left(n\frac{\sinh(nL)}{\cosh(nL)-1}-\abs{n}\right), \mbox{ for } n\neq 0, \quad \widetilde{Q}_{0}=\frac{2}{L} \mbox{ for } n=0.$$
Therefore, using the fact
$$\norm{\mathcal{H}}_{\mathcal{L}(C^{1,\alpha}(\SS^1))}\leq C, \quad \abs{\widetilde{Q}_{n}}\leq e^{-nL}$$
and  \eqref{formula:holder:exp} we infer that
\begin{equation}
	\norm{\mathsf{G}}_{C^{1,\alpha}(\SS^1)}\leq C \left(\norm{\mathcal{H}\p_{x}\widetilde{g}}_{C^{1,\alpha}(\SS^1)}+ \norm{\widetilde{g}}_{L^{\infty}(\SS^1)}\right)
	\leq C \norm{\widetilde{g}}_{C^{2,\alpha}(\SS^1)}.
\end{equation}
To conclude, we recall that $\widetilde{g}(x)=g(x)-\mathcal{Z}(x)$ and thus by means of Lemma \ref{Lemma:G:1} we obtain that
\begin{equation}
	\norm{\mathsf{G}}_{C^{1,\alpha}(\SS^1)}\leq C \norm{\widetilde{g}}_{C^{2,\alpha}(\SS^1)}\leq C \left( \norm{f}_{C^{2,\alpha}(\SS^1)}+ \norm{g}_{C^{2,\alpha}(\SS^1)}\right).
\end{equation}
\end{proof}

\section{Existence of solutions to the integral equation for \texorpdfstring{$j_{0}$}{Lg}}\label{S5}
In this section we are interested in studying the existence of a solution $j_0\in C^{1,\alpha}(\partial\Omega_{-})$ to the integral equation \eqref{final:eq:j_0} given by
\begin{equation}\label{integral:eq:fixed:2}
j_0(x)=- \displaystyle\sum_{i=1}^{4}(\mathsf{T}_{i}j_{0}(x)-\langle \mathsf{T}_{i}j_{0} \rangle)-\langle j_0 f^{-}\rangle+\mathsf{G}(x)-\langle \mathsf{G} \rangle
\end{equation}
To that purpose, let us first introduce the following notation. Given $j_{0}\in C^{1,\alpha}(\partial\Omega_{-})$ we define the operator
\begin{equation*}
\Upsilon:B_{\delta_{0}}(C^{2,\alpha}(\Omega))\times B_{\delta_{1}}(C^{1,\alpha}(\Omega)) \to \mathcal{L}\left(C^{1,\alpha}(\partial\Omega_{-})\right)
\end{equation*}
such that for $(\Lambda,\vartheta)\in B_{\delta_{0}}(C^{2,\alpha}(\Omega))\times B_{\delta_{1}}(C^{1,\alpha}(\Omega))$ we have that 
\begin{equation}\label{eq:Gamma}
\Upsilon[\Lambda,\vartheta](j_{0})= -\displaystyle\sum_{i=1}^{4}\left(\mathsf{T}_{i}[\Lambda,\vartheta]j_{0}-\langle \mathsf{T}_{i}[\Lambda,\vartheta]j_{0}  \rangle\right)-\langle f^{-}j_{0} \rangle
\end{equation}
where $\mathsf{T}_{i}$ are defined in \eqref{def:T1},\eqref{def:T2},\eqref{def:T3} and \eqref{def:T4}. 

\begin{remark}
More precisely, the operators 
$\mathsf{T}_{i}[\Lambda,\vartheta]$ as stated in \eqref{def:T1},\eqref{def:T2},\eqref{def:T3} and \eqref{def:T4} are written for the particular case where $\vartheta(\eta,y)=\partial_{\xi} \Theta(X(\eta,y),y))$. However, we will show the existence of solutions for a more general class of integral equations, namely for general functions $\Lambda$ and $\vartheta$ which satisfied certain regularity and smallness assumptions. We will later check that for the particular case where $\vartheta(\eta,y)=\partial_{\xi} \Theta(X(\eta,y),y))$ the required assumptions are  satisfied (cf. Section \ref{sec:6}).
\end{remark}

\begin{remark}
In a similar manner as we did for the operators $\mathsf{T}_{1}, \ldots, \mathsf{T}_{4}$, we now define an operator $\Upsilon$ acting either in $C^\alpha$ or in   $C^{1,\alpha}$. We will not use different symbols for operators acting in different spaces for the sake of simplicity, (cf. Remark \ref{remark:operator:T}). 
\end{remark} 
In the first place, we have the following two lemmas
\begin{lemma}\label{lemma:1:fixed}
Let $M_{0}\leq \displaystyle\min\{\delta_0,\delta_1\}$ where $\delta_0, \delta_1$ are defined in Assumption \ref{assumption:Lambda} and  Assumption \ref{assumption:Theta}, respectively. Let $\widetilde{M}\leq M_{0}$ and 
\begin{equation}\label{smallness:lambda:vartheta}
	\norm{\Lambda}_{C^{2,\alpha}(\Omega)}+    \norm{\vartheta}_{C^{1,\alpha}(\Omega)}+\norm{f^{-}}_{C^{2,\alpha}(\partial\Omega_{-})} \leq \widetilde{M}.
\end{equation}
Then we have that
\begin{equation}\label{estimate:upsilon:1}
	\norm{\Upsilon[\Lambda,\vartheta]}_{\mathcal{L}\left(C^{1,\alpha}(\partial\Omega_{-})\right)}\leq C \widetilde{M}.
\end{equation}
Furthermore, the operator $\Upsilon[\Lambda,\vartheta]$ is Lipschitz in $C^{\alpha}(\partial\Omega_{-})$, i.e. for any  $\Lambda^{1},\Lambda^{2}$ and $\vartheta^{1},\vartheta^{2}$ satisfying \eqref{smallness:lambda:vartheta} we have that
\begin{equation}\label{estimate:upsilon:1:contraction}
	\norm{\Upsilon[\Lambda^1,\vartheta^1]- \Upsilon[\Lambda^2,\vartheta^2]}_{\mathcal{L}\left(C^{\alpha}(\partial\Omega_{-})\right)}\leq C\left( \norm{\Lambda^1-\Lambda^2}_{C^{1,\alpha}(\Omega)}+ \norm{\vartheta^2-\vartheta^1}_{C^{\alpha}(\Omega)}\right).
\end{equation}
\end{lemma}
\begin{proof}
The proof of \eqref{estimate:upsilon:1} is a consequence of the estimates  \eqref{estimate:T1:C1alpha},   \eqref{estimate:T2:C1alpha}, \eqref{estimate:prop:T3:C1alpha} and \eqref{estimate:prop:T4:C1alpha}, as well as the fact that those estimates are preserved for the averaging operators $\langle \cdot \rangle$. Notice that the derivatives of the averaging operators are zero since they are just constant functions.

Indeed, applying those bounds we readily check that
\begin{equation}
	\norm{\Upsilon[\Lambda,\vartheta]}_{\mathcal{L}\left(C^{1,\alpha}(\partial\Omega_{-})\right)}\leq C  \widetilde{M}
\end{equation}
for $  \widetilde{M}\leq M_{0}$. To show estimate \eqref{estimate:upsilon:1:contraction}, we invoke bounds \eqref{estimate:difference:T1:Calpha}, \eqref{estimate:difference:T2:Calpha}, \eqref{estimate:difference:T3:Calpha} and \eqref{estimate:difference:T4:Calpha} to obtain 
\begin{align*}
	\norm{\Upsilon[\Lambda^1,\vartheta^1]- \Upsilon[\Lambda^2,\vartheta^2]}_{\mathcal{L}\left(C^{\alpha}(\partial\Omega_{-})\right)} &\leq  \displaystyle \sum_{i=1}^{4}\norm{\mathsf{T}_{i}[\Lambda^1,\vartheta^1]- \mathsf{T}_{i}[\Lambda^2,\vartheta^2]}_{\mathcal{L}\left(C^{\alpha}(\partial\Omega_{-})\right)}  \\
	&\quad \quad + \displaystyle \sum_{i=1}^{4} \norm{ \left(\langle \mathsf{T}_{i}[\Lambda_{1},\vartheta_{1}](\cdot) \rangle- \mathsf{T}_{i}[\Lambda_{2},\vartheta_{2}](\cdot) \right)}_{\mathcal{L}\left(C^{\alpha}(\partial\Omega_{-})\right)}\\
	&\leq C\left( \norm{\Lambda^1-\Lambda^2}_{C^{1,\alpha}(\Omega)}+ \norm{\vartheta^2-\vartheta^1}_{C^{\alpha}(\Omega)}\right)
\end{align*}
where the action of the operators $\langle \mathsf{T}_{i}[\Lambda_{\ell},\vartheta_{\ell}](\cdot) \rangle$ for $i=1,\ldots,4$ and $\ell=1,2$ acting on the function $j_0$ is given by $\langle \mathsf{T}_{i}[\Lambda_{\ell},\vartheta_{\ell}]j_{0}  \rangle$. We also use the fact that the terms $\langle f^{-}j_0 \rangle$ cancel out.
\end{proof}

\begin{lemma}\label{lemma:2:fixed}
There exists $M_{0}\leq \displaystyle\min\{\delta_0,\delta_1\}$ such that for any $  \widetilde{M}\leq M_{0}$ and $\Lambda,\vartheta, f^{-}$ satisfying \eqref{smallness:lambda:vartheta} the operator $\left(\mathsf{I}-\Upsilon\right)$ is invertible in $C^{1,\alpha}(\partial\Omega_{-})$. More precisely, there exists an operator $\Pi[\Lambda,\vartheta]= \left(\mathsf{I}-\Upsilon[\Lambda,\vartheta]\right)^{-1}$ such that
\begin{equation*}
	\Pi:B_{\widetilde{M}}(C^{2,\alpha}(\Omega))\times B_{\widetilde{M}}(C^{1,\alpha}(\Omega)) \to \mathcal{L}\left(C^{1,\alpha}(\partial\Omega_{-})\right).
\end{equation*}
Moreover, the operator $\Pi[\Lambda,\vartheta]$ is Lipschitz in $C^{\alpha}(\partial\Omega_{-})$, i.e. for any  $\Lambda^{1},\Lambda^{2}$ and $\vartheta^{1},\vartheta^{2}$ satisfying \eqref{smallness:lambda:vartheta} we have that
\begin{equation}\label{difference:PI:lower}
	\norm{\Pi[\Lambda^1,\vartheta^1]- \Pi[\Lambda^2,\vartheta^2]}_{\mathcal{L}\left(C^{\alpha}(\partial\Omega_{-})\right)}\leq C\left( \norm{\Lambda^1-\Lambda^2}_{C^{1,\alpha}(\Omega)}+ \norm{\vartheta^2-\vartheta^1}_{C^{\alpha}(\Omega)}\right).
\end{equation}
\end{lemma}
\begin{proof}
Under the hypothesis of Lemma \ref{lemma:1:fixed} we have shown that  $\norm{\Upsilon[\Lambda,\vartheta]}_{\mathcal{L}\left(C^{1,\alpha}(\partial\Omega_{-})\right)}\leq C \widetilde{M}$ for some $\widetilde{M}$ sufficiently small. Therefore, the existence of  $\left(\mathsf{I}-\Upsilon\right)^{-1}$
follows from the classical Neumann series (cf. \cite{ReedsSimon}). Indeed,  we have that
\begin{equation}\label{Neumann:series}
	\Pi[\Lambda,\vartheta]=\left(\mathsf{I}-\Upsilon[\Lambda,\vartheta]\right)^{-1}=\displaystyle \sum_{n=0}^{\infty}(-1)^{n}\Upsilon^{n}[\Lambda,\vartheta]
\end{equation}
where $n=0$, $\Upsilon^{0}[\Lambda,\vartheta]=\mathcal{I}$ is the identity operator. Moreover, for $\Upsilon[\Lambda,\vartheta]\in \mathcal{L}\left(C^{1,\alpha}(\partial\Omega_{-})\right)$ and using estimate \eqref{estimate:upsilon:1} we find that
\begin{equation}
	\norm{\Pi[\Lambda,\vartheta]}_{\mathcal{L}\left(C^{1,\alpha}(\partial\Omega_{-})\right)} = \norm{\displaystyle \sum_{n=0}^{\infty}(-1)^{n}\Upsilon^{n}[\Lambda,\vartheta]}_{\mathcal{L}\left(C^{1,\alpha}(\partial\Omega_{-})\right)}\leq \frac{1}{1-C\widetilde{M}}.
\end{equation}
Denoting by $A_{1}^{n}=\Upsilon^{n}[\Lambda^1,\vartheta^1]$, $A_{2}^{n}=\Upsilon^{n}[\Lambda^2,\vartheta^2]$ we can find that 
\begin{align}
	A_{1}^{n}-A_{2}^{n}=A_{1}^{n-1}\left(A_{1}-A_{2}\right)+A_{1}^{n-2}\left(A_{1}-A_{2}\right)A_{2}+\ldots+A\left(A_{1}-A_{2}\right)A_{2}^{n}.\label{combination:Neumann:series}
\end{align}
Thus, combining  \eqref{Neumann:series} and \eqref{combination:Neumann:series} we find that
\begin{align*}
	\Pi[\Lambda^1,\vartheta^1]-\Pi[\Lambda^2,\vartheta^2]&=\displaystyle \sum_{n=0}^{\infty}(-1)^{n}A^{n}_{1}-\displaystyle \sum_{n=0}^{\infty}(-1)^{n}A^{n}_{2} \\
	&= \displaystyle \sum_{n=0}^{\infty}(-1)^{n} \bigg[ A_{1}^{n-1}\left(A_{1}-A_{2}\right)+A_{1}^{n-2}\left(A_{1}-A_{2}\right)A_{2}+\ldots+A\left(A_{1}-A_{2}\right)A_{2}^{n} \bigg]
\end{align*}
and hence by means of bounds \eqref{estimate:upsilon:1}-\eqref{estimate:upsilon:1:contraction}we infer that if $M_0$ is sufficiently small that
\begin{align*}
	\norm{\Pi[\Lambda^1,\vartheta^1]- \Pi[\Lambda^2,\vartheta^2]}_{\mathcal{L}\left(C^{\alpha}(\partial\Omega_{-})\right)}&\leq \norm{A_{1}-A_{2}}_{\mathcal{L}\left(C^{\alpha}(\partial\Omega_{-})\right)}  \displaystyle \sum_{n=0}^{\infty} n (C\widetilde{M})^{n-1} \\
	&\leq  C \norm{\Upsilon^{n}[\Lambda^1,\vartheta^1]-\Upsilon^{n}[\Lambda^2,\vartheta^2]}_{\mathcal{L}\left(C^{\alpha}(\partial\Omega_{-})\right)} \\
	&\leq C\left( \norm{\Lambda^1-\Lambda^2}_{C^{1,\alpha}(\Omega)}+ \norm{\vartheta^2-\vartheta^1}_{C^{\alpha}(\Omega)}\right).
\end{align*}
\end{proof}
Combining both lemmas we can provide the existence of solutions to the integral equation \eqref{integral:eq:fixed:2} which reads
\begin{proposition}\label{Prop5.4}
Let the hypothesis of Lemma \ref{lemma:1:fixed} hold. Let also $f\in C^{2,\alpha}(\Omega)$, $g\in C^{2,\alpha}(\partial\Omega_{-})$ and $\mathsf{G}$ given as in \eqref{func:G:rig}. Then there exists a solution $j_0\in C^{1,\alpha}(\partial\Omega_{-})$ to \eqref{integral:eq:fixed:2} given by 
\begin{equation}\label{solution:explicit:omega_0}
	j_{0}=\Pi[\Lambda,\vartheta]\left( \mathsf{G}-\langle \mathsf{G} \rangle \right).
\end{equation}
Furthermore, 
\begin{equation}\label{bound:Prop5.4}
	\norm{j_{0}}_{C^{1,\alpha}(\partial\Omega_{-})}\leq C \left( \norm{f}_{C^{2,\alpha}(\Omega)}+ \norm{g}_{C^{2,\alpha}(\partial\Omega_{-})}\right).
\end{equation}
\end{proposition}
\begin{proof}
The fact that $j_0$ as given in \eqref{solution:explicit:omega_0} solves \eqref{integral:eq:fixed:2} is a consequence of the definition of the operator $\Pi[\Lambda,\vartheta]$ in Lemma \ref{lemma:2:fixed}. On the other hand, we notice by means of Lemma \ref{Lemma:G:12} and estimate \eqref{bound:G:C1alpha} yields
\begin{equation}
	\norm{\mathsf{G}-\langle\mathsf{G}\rangle}_{C^{1,\alpha}(\SS^1)}\leq C \left( \norm{f}_{C^{2,\alpha}(\SS^1)}+ \norm{g}_{C^{2,\alpha}(\SS^1)}\right)
\end{equation}
where $\mathsf{G}$ is as in \eqref{func:G:rig}. This estimate as well as the fact that $\Pi:B_{\widetilde{M}}(C^{2,\alpha}(\Omega))\times B_{\widetilde{M}}(C^{1,\alpha}(\Omega)) \to \mathcal{L}\left(C^{1,\alpha}(\partial\Omega_{-})\right)$ 
\begin{equation}
	\norm{j_{0}}_{C^{1,\alpha}(\partial\Omega_{-})}\leq C \left( \norm{f}_{C^{2,\alpha}(\Omega)}+ \norm{g}_{C^{2,\alpha}(\partial\Omega_{-})}\right).
\end{equation}
\end{proof}

\section{The fixed point argument}\label{sec:6}
In this section we will provide the fixed point argument, this is, we will define an adequate operator $\Gamma$ on a subspace of $C^{2,\alpha}(\Omega)$ which has a fixed point $b$ such that $B=(0,1)+b$ is a solution to \eqref{MHS2D:eq} and \eqref{bvc}. 

We define the operator $\Gamma: B_{M}(C^{2,\alpha}(\Omega)) \to C^{2,\alpha}(\Omega)$ using several intermediate building blocks. Given $b\in B_{M}(C^{2,\alpha}(\Omega))$ we define the flow map associated with the vector field $B=(0,1)+b$ as the mapping $X[b]: B_{M}(C^{2,\alpha}(\Omega))\to C^{2,\alpha}(\Omega)$ which satisfies the ordinary differential equation 
\begin{equation}\label{ode:flow:X}
	\left\lbrace
	\begin{array}{lll}
		\partial_{y} X[b](\xi,y)= \frac{b^{1}(X(\xi,y),y)}{1+b^{2}(X(\xi,y),y)} \\
		X[b](\xi,0)=\xi.
	\end{array}\right.
\end{equation}
Moreover, we denote by $X[b]^{-1}(\xi,y)$ the inverse function of $X[b]$ in the first variable, namely $X[b](X[b]^{-1}(\xi,y),y)=\xi$. Then we define $\Lambda:B_{M}(C^{2,\alpha}(\Omega))\to C^{2,\alpha}(\Omega)$ as
\begin{equation}\label{def:Lambda:sec6}
	\Lambda[b](\eta,y)= X[b](\eta,y)-\eta
\end{equation}
Actually, in Lemma \ref{flow:map:bound2} we will show the stronger result $Im(\Lambda)\subset B_{CM}(C^{2,\alpha}(\Omega))$ for $C>0$.
We define also the function $\Theta:B_{M}(C^{2,\alpha}(\Omega))\to C^{2,\alpha}(\Omega)$ as
\begin{equation}\label{def:Theta:V}
	\Theta[b](\xi,y)=X[b]^{-1}(\xi,y)-\xi.
\end{equation}
Finally $\vartheta[b]:B_{M}(C^{2,\alpha}(\Omega))\to C^{1,\alpha}(\Omega)$ is given by
\begin{equation}\label{vartheta:Theta:V}
	\vartheta[b](\eta,y)=\partial_{\xi} \Theta[b](X[b](\eta,y),y).
\end{equation}
Notice that combining \eqref{def:Theta:V}-\eqref{vartheta:Theta:V} we can write
\begin{equation}\label{def:Comb}
	\vartheta[b](\eta,y)= \partial_{\xi} X^{-1}[b]\left(X[b](\eta,y),y)\right)-1.
\end{equation}
Moreover, arguing as in Lemma \ref{flow:map:bound2} we prove that $Im(\vartheta)\subset B_{CM}(C^{1,\alpha}(\Omega))$.

Therefore, we now introduce the following operator $$\Psi[b]: B_{M}(C^{2,\alpha}(\Omega))\to B_{CM}(C^{2,\alpha}(\Omega))\times B_{CM}(C^{1,\alpha}(\Omega))$$ defined by $\Psi[b]=(\Lambda[b], \vartheta[b])$ and choose $\widetilde{M}=CM\leq M_{0}$.
Next, notice that the function $j_{0}\in C^{1,\alpha}(\partial\Omega_{-})$ given by \eqref{solution:explicit:omega_0} solving the integral equation \eqref{eq:Gamma} can be expressed as the following composition of operators
\begin{equation}\label{omega:cero:fixed}
	j_0= \Pi[\Psi[b]](\mathsf{G}-\langle \mathsf{G} \rangle).
\end{equation}
The condition $\widetilde{M}=CM\leq M_{0}$ must be satisfied so that Proposition \ref{Prop5.4} can be applied.
To conclude the construction, we use two additional building blocks. First, for $b\in B_{M}(C^{2,\alpha}(\Omega))$  and  $j_0= \Pi[\Psi[b]](\mathsf{G}-\langle \mathsf{G} \rangle)$, we define $$j=T[b,j_0]:B_{M}(C^{2,\alpha}(\Omega)) \times C^{1,\alpha}(\Omega_{-})\to C^{1,\alpha}(\Omega)$$ 
where $j$ is the unique solution to the transport type problem 
\begin{equation}\label{transport:problem:full:fixed}
	T[b,j_{0}]:
	\left\lbrace
	\begin{array}{lll}
		((0,1)+b)\cdot\nabla j =0, \ \mbox{in} \ \Omega \\
		j = j_{0}, \ \mbox{on} \ \partial \Omega_{-}.
	\end{array} \right.
\end{equation} 
To conclude, the new velocity field $W\in C^{2,\alpha}(\Omega)$ is recovered by means of the div-curl problem (also known as Biot-Savart operator) given by
$$W=B_{s}[j,f,J]:C^{1,\alpha}(\Omega)\times C^{2,\alpha}(\partial\Omega)\times\mathbb{R}\to C^{2,\alpha}(\Omega)$$
where $W$ is the unique solution to
\begin{equation}\label{div:curl:problem:full:fixed}
	B_{s}[j,f,J]:
	\left\lbrace
	\begin{array}{lll}
		\nabla\times W= j, \ \mbox{in} \ \Omega \\
		\mbox{div } W=0, \ \mbox{in} \ \Omega \\ 
		W\cdot n=f, \ \mbox{on} \ \partial \Omega \\
		\int_{0}^{L} W_{1}(0,y) \ dy=J.
	\end{array} \right.
\end{equation} 
where $J=J[f,g,b]$ is given by
\begin{equation}\label{combination:5}
\frac{2J}{L^2}=-\frac{1}{2\pi}\int_{0}^{2\pi}\mathsf{G}(x)  dx+\frac{1}{2\pi} \displaystyle\sum_{i=1}^{4}\int_{0}^{2\pi} \mathsf{T}_{i}j_{0}(x) \ dx \ dx -\frac{1}{2\pi}\int_{0}^{2\pi} j_0(x)f^{-}(x) \ dx
\end{equation}
(cf. \eqref{combination:3}) with $j_0= \Pi[\Psi[b]](\mathsf{G}-\langle \mathsf{G} \rangle)$.
Then,  we define $\Gamma(b)=W$. In particular, the full operator can be expressed as the following composition of operators
\begin{equation}\label{full:composiotion:operator}
	W=\Gamma(b)=B_{s}[T[b,\Pi[\Psi[b]](\mathsf{G}-\langle \mathsf{G}\rangle)], f,J[f,g,b]].
\end{equation}

The precise statement of the theorem reads as follows:
\begin{theorem} \label{theorem2}
	Let $f\in C^{2,\alpha}(\Omega)$ satisfying \eqref{compcond} and $g\in C^{2,\alpha}(\partial\Omega_{-})$. There exist $\epsilon_{0}>0, M_{0}=M_{0}(L,\alpha)$ sufficiently small such that if
	\begin{equation}\label{small:f:g}
		\norm{f}_{C^{2,\alpha}(\Omega)}+\norm{g}_{C^{2,\alpha}(\partial\Omega_{-})}\leq \epsilon_{0}M, \mbox{ for } M\leq M_{0},
	\end{equation}
	then $\Gamma(B_{M}(C^{2,\alpha}(\Omega)))\subset B_{M}(C^{2,\alpha}(\Omega))$. Furthermore, the operator $\Gamma$ has a unique fixed point in $B_{M}(C^{2,\alpha}(\Omega))$. 
\end{theorem}

\subsection{Preliminary estimates: ODE, transport problem and div-curl problem}
Before showing the proof of Theorem \ref{theorem2}, let us first show several Lemmas that will be needed to provide the proof of Theorem \ref{theorem2}. The first result summarizes general H\"older estimates for $\Lambda[b]$ and $\vartheta[b]$.

	\begin{lemma}\label{flow:map:bound2}
		Let $M_{0}$ be sufficiently small and let 
		$b\in B_{M}(C^{2,\alpha}(\Omega))$ with $M\leq M_0$. Then, for $\Lambda[b], \vartheta[b]$ given in \eqref{def:Lambda:sec6} and \eqref{def:Comb}, we have that $\Lambda[b]\in C^{2,\alpha}(\Omega)$ $\vartheta[b]\in C^{1,\alpha}(\Omega)$ and
		\begin{align}
			\norm{\Lambda[b]}_{C^{2,\alpha}(\Omega)}\ &\leq C M, \label{bound:Lambda:V} \\
			\norm{\vartheta[b]}_{C^{1,\alpha}(\Omega)} & \leq  CM. \label{bound:vartheta:V}
		\end{align}
		Moreover, the operator $\Lambda[b]$ and $\vartheta[b]$ are Lipschitz in $C^{1,\alpha}(\Omega)$ and $C^{\alpha}(\Omega)$ respectively. This is for any $b^{1},b^{2} \in B_{M}(C^{2,\alpha}(\Omega))$
		\begin{align}
			\norm{\Lambda[b^1]- \Lambda[b^2]}_{C^{1,\alpha}(\Omega)}\leq C \norm{b^1-b^2}_{C^{1,\alpha}(\Omega)}, \label{estimate:Lambda:contraction} \\
			\norm{\vartheta[b^1]- \vartheta[b^2]}_{C^{\alpha}(\Omega)}\leq C \norm{b^1-b^2}_{C^{\alpha}(\Omega)}.\label{estimate:vartheta:contraction}
		\end{align}
	\end{lemma}
	\begin{proof}
		The proof of Lemma \ref{flow:map:bound2} is the standard argument used to compute the dependence of the solutions for an ODE in their parameters. More precisely, the main idea of the proof is to control incremental quotients of the form $\frac{f(x+h)-f(x)}{h}$ for $h>0$, as well as terms quotients of the form 
		$\frac{\abs{f(x)-f(y)}}{\abs{x-y}^{\alpha}}$ using Gr\"onwall type arguments. A bound similar to \eqref{bound:Lambda:V}, \eqref{bound:vartheta:V} in Lemma \ref{flow:map:bound2} but estimating only the $C^{1,\alpha}$ H\"older norm have been shown in \cite[Lemma 3.7]{Alo-Velaz-2021}. Moreover, the proof of \eqref{estimate:Lambda:contraction}, \eqref{estimate:vartheta:contraction} is obtained computing the differences of the solutions of the differential equations which define $\Lambda$, $\vartheta$ (cf. \eqref{def:Lambda:sec6}, \eqref{def:Comb} and \eqref{ode:flow:X}) with $b=b^1$ and $b=b^2$.
		
	\end{proof}
	
	The next results deals with H\"older estimates for solutions to the hyperbolic transport type problem \eqref{transport:problem:full:fixed}. For a proof of this result we refer the reader to \cite[Proposition 3.8]{Alo-Velaz-2021}, where a more general result is shown.
	\begin{proposition}\label{TP:prop}
		Let $M_{0}$ be sufficiently small. Then for every $M\leq M_{0}$, $b\in B_{M}(C^{2,\alpha}(\Omega))$ and $j_0\in  C^{1,\alpha}(\partial\Omega_{-})$,  there exists a unique $j\in C^{1,\alpha}(\Omega)$ solving
		\begin{equation}\label{transport:propo:6}
			T[b,j_0]:
			\left\lbrace
			\begin{array}{lll}
				(B_{0}+b)\cdot\nabla j =0 \ \mbox{in} \ \Omega, \\
				j = j_{0} \ \mbox{on} \ \partial \Omega_{-}.
			\end{array} \right.
		\end{equation}
		Moreover, there exists a constant $C=C(\alpha,L)>0$ such that the following estimate holds
		\begin{equation}\label{estimate:TP}
			\norm{j}_{C^{1,\alpha}(\Omega)} \leq C  \norm{j_0}_{C^{1,\alpha}(\partial\Omega_{-})}. 
		\end{equation}
		Furthermore, let $j_{1},j_{2}\in C^{\alpha}(\Omega)$ be  two different solutions to  \eqref{transport:propo:6} with  $b$ given by $b^{1},b^{2}\in B_{M}(C^{2,\alpha}(\Omega))$  respectively. Then
		\begin{equation}\label{estimate:difference:TP}
			\norm{j^{1}-j^{2}}_{C^{\alpha}(\Omega)} \leq  C \left(\norm{j^{1}_{0}-j^{2}_{0}}_{C^{\alpha}(\partial\Omega_{-})}+ \norm{j^{1}_{0}}_{C^{1,\alpha}(\partial\Omega_{-})}\norm{V^1-V^2}_{C^{\alpha}(\Omega)}\right)
		\end{equation}
		where $C=C(\alpha,L)>0$.
	\end{proposition}
	To conclude let us also recall the following result regarding H\"older estimate for the div-curl problem, cf.  \cite[Proposition 3.11]{Alo-Velaz-2021} for a detailed proof.
	\begin{proposition}\label{DCP:prop}
		For every $J\in \mathbb{R}$, $j\in C^{1,\alpha}(\Omega)$ and $f\in C^{2,\alpha}(\partial\Omega)$ satisfying \eqref{compcond},  there exists a unique solution $W\in C^{2,\alpha}(\Omega)$ solving
		\begin{equation}\label{DCP:problem}
			B_{s}[j,f,J]:
			\left\lbrace
			\begin{array}{lll}
				\nabla \times W = j, \quad \mbox{in } \Omega \\ 
				\mbox{div } W =0,  \quad \mbox{in } \Omega  \\
				W \cdot n = f ,  \quad \mbox{on } \partial \Omega \\
				\int_{\mathcal{C}}W\cdot n \ dS= J.
			\end{array}\right.
		\end{equation}
		where the curve $\mathcal{C}=\{ (0,y), y\in [0,L]\}.$ Moreover, the solution satisfies the inequality
		\begin{equation}\label{estimate:DCP}
			\norm{W}_{C^{2,\alpha}(\Omega)} \leq C \left(\norm{j}_{C^{1,\alpha}(\Omega)}+ \norm{f}_{C^{2,\alpha}(\partial\Omega)} + \abs{J}\right),
		\end{equation}
		where  $C=C(L,\alpha)>0$. 
	\end{proposition}
	\subsection{Proof of Theorem \ref{theorem2}}
	First, we show that the operator  $\Gamma$ maps $ b\in B_{M}(C^{2,\alpha}(\Omega))$ into itself and second, that the operator  $\Gamma$ is a contraction mapping in the lower order norm $C^{1,\alpha}(\Omega)$. By combining both ingredients, we can invoke Banach fixed point theorem to infer that the operator $\Gamma$ has a unique fixed point in $B_{M}(C^{2,\alpha}(\Omega)).$ Let us start with the former assertion.  By means of \eqref{full:composiotion:operator} we find that for $b\in B_{M}(C^{2,\alpha}(\Omega))$
	\begin{align*}
		\norm{\Gamma(b)}_{C^{2,\alpha}(\Omega)}&=\norm{B_{s}[T[b,\Pi[\Psi[b]](\mathsf{G}-\langle \mathsf{G} \rangle)], f,J]}_{C^{2,\alpha}(\Omega)} \\
		&\leq C\left(\norm{T[b,\Pi[\Psi[b]](\mathsf{G}-\langle \mathsf{G} \rangle)]}_{C^{1,\alpha}(\Omega)}+\norm{f}_{C^{2,\alpha}(\partial\Omega)} + \abs{J}\right) \\
		&\leq C\left(\norm{\Pi[\Psi[b]](\mathsf{G}-\langle \mathsf{G} \rangle)]}_{C^{1,\alpha}(\Omega)}+\norm{f}_{C^{2,\alpha}(\partial\Omega)} + \abs{J}\right).
	\end{align*}
	where in the first inequality we have used \eqref{estimate:DCP} in Proposition \ref{DCP:prop} and in the latter we invoked \eqref{estimate:difference:TP} in Proposition \ref{TP:prop}.
	On the other hand, recall that by definition $\Psi[b]=(\Lambda[b],\vartheta[b])$. Hence, combining inequalities \eqref{bound:Lambda:V}-\eqref{bound:vartheta:V} in Lemma \ref{flow:map:bound2} and estimate \eqref{bound:Prop5.4} in Proposition \ref{Prop5.4} we have that 
	\begin{equation*}
		\norm{\Pi[\Psi[b]](\mathsf{G}-\langle \mathsf{G} \rangle)]}_{C^{1,\alpha}(\Omega)}\leq C \left( \norm{f}_{C^{2,\alpha}(\Omega)}+ \norm{g}_{C^{2,\alpha}(\partial\Omega_{-})}\right).
	\end{equation*}
	Moreover, we can show using the expression  of $J$ given in \eqref{combination:5} and the previous estimates \eqref{estimate:T1:C1alpha},   \eqref{estimate:T2:C1alpha}, \eqref{estimate:prop:T3:C1alpha}, \eqref{estimate:prop:T4:C1alpha} and  \eqref{bound:Prop5.4} that
	\begin{equation}
	    \abs{J}\leq C \left( \norm{f}_{C^{2,\alpha}(\Omega)}+ \norm{g}_{C^{2,\alpha}(\partial\Omega_{-})}\right).
	\end{equation}

	Thus, we readily check that 
	\begin{align}\label{Gamma:estimate:ball}
		\norm{\Gamma(b)}_{C^{2,\alpha}(\Omega)}\leq C \left( \norm{f}_{C^{2,\alpha}(\Omega)}+ \norm{g}_{C^{2,\alpha}(\partial\Omega_{-})} \right) \leq C\epsilon_{0}M
	\end{align}
	where in the second inequality we have used the smallness assumption \eqref{small:f:g}. Choosing $\epsilon_{0}=\frac{1}{4C}$, we obtain that $\Gamma(B_{M}(C^{2,\alpha}(\Omega)))\subset B_{M}(C^{2,\alpha}(\Omega))$. 
	
	We now claim that the $B_{M}(C^{2,\alpha}(\Omega))$ endowed with the topology $C^{1,\alpha}$ is a complete metric space which we will denote by $(B_{M}(C^{2,\alpha}(\Omega)), \norm{\cdot}_{C^{1,\alpha}})$.  In order to show this it is sufficient to check that $B_{M}(C^{2,\alpha}(\Omega))$ is a closed subset of $C^{1,\alpha}(\Omega)$, (cf. \cite[Proof of Lemma 3.12]{Alo-Velaz-2021}).

	Moreover, we also claim that  
	$$\Gamma: (B_{M}(C^{2,\alpha}(\Omega)), \norm{\cdot}_{C^{1,\alpha}}) \to (B_{M}(C^{2,\alpha}(\Omega)), \norm{\cdot}_{C^{1,\alpha}})$$
	is a contraction mapping. To this end, we have to show that for $b^{1},b^{2}\in  B_{M}(C^{2,\alpha}(\Omega))$, we need to estimate the difference $\norm{\Gamma(b^1)-\Gamma(b^2)}_{C^{1,\alpha}(\Omega)}$. To that purpose, using the expression of the full operator given in \eqref{full:composiotion:operator} and noticing that the Biot-Savart operator defined in \eqref{div:curl:problem:full:fixed} is a linear operator, we obtain by means of Proposition \ref{DCP:prop} that
	\begin{align}
		\norm{\Gamma(b^1)-\Gamma(b^2)}_{C^{1,\alpha}(\Omega)}&=\norm{B_{s}[T[b^1,\Pi[\Psi[b^1]](\mathsf{G}-\langle \mathsf{G} \rangle)], f,J^{1}]-B_{s}[T[b^2,\Pi[\Psi[b^2]](\mathsf{G}-\langle \mathsf{G} \rangle)], f,J^{2}]}_{C^{1,\alpha}(\Omega)} \nonumber \\
		&\leq C\left(\norm{T[b^1,\Pi[\Psi[b^1]](\mathsf{G}-\langle \mathsf{G} \rangle)-T[b^2,\Pi[\Psi[b^2]](\mathsf{G}-\langle \mathsf{G} \rangle)}_{C^{1,\alpha}(\Omega)}\right)+\abs{J^{1}-J^{2}} \label{contraction:gamma:est1}
	\end{align}
	where $J^{\ell}$ is given by \eqref{combination:5} with $j_0= \Pi[\Psi[b^{\ell}]](\mathsf{G}-\langle \mathsf{G} \rangle)$ for $\ell=1,2$.
	To deal with the transport type operator $T$  given in \eqref{transport:problem:full:fixed}, we invoke inequality \eqref{estimate:difference:TP} in Proposition \ref{TP:prop} to find that
	\begin{align*}
		&\norm{T[b^1,\Pi[\Psi[b^1]](\mathsf{G}-\langle \mathsf{G} \rangle)-T[b^2,\Pi[\Psi[b^2]](\mathsf{G}-\langle \mathsf{G} \rangle)}_{C^{1,\alpha}(\Omega)}  \\
		& \quad \quad \quad \leq C\bigg[ \norm{\Pi[\Psi[b^1]](\mathsf{G}-\langle \mathsf{G} \rangle)-\Pi[\Psi[b^2]](\mathsf{G}-\langle \mathsf{G} \rangle)}_{C^{\alpha}(\partial\Omega_{-})} \\
		&\quad \quad  \quad \quad \quad+ \norm{\Pi[\Psi[b^1]](\mathsf{G}-\langle \mathsf{G} \rangle}_{C^{1,\alpha}(\partial\Omega_{-})}\norm{b_{1}-b_{2}}_{C^{1,\alpha}(\Omega)}\bigg].
	\end{align*}
	On other hand by means of \eqref{difference:PI:lower} in Lemma \ref{lemma:2:fixed} and recalling that $\Psi[b^1]=(\Lambda[b^1],\vartheta[b^1])$ we arrive at
	\begin{align}
		&\norm{\Pi[\Psi[b^1]](\mathsf{G}-\langle \mathsf{G} \rangle)-\Pi[\Psi[b^2]](\mathsf{G}-\langle \mathsf{G} \rangle)}_{C^{\alpha}(\partial\Omega_{-})} \nonumber \\
		& \hspace{4cm}  \leq C\left(\norm{\Lambda^1-\Lambda^2}_{C^{1,\alpha}(\Omega)}+ \norm{\vartheta^2-\vartheta^1}_{C^{\alpha}(\Omega)}\right)\norm{\mathsf{G}-\langle \mathsf{G} \rangle)}_{C^{\alpha}(\partial\Omega_{-})} \label{cota:diferencia:psi:b1b2}
	\end{align}
	and
	\begin{align*}
		\norm{\Pi[\Psi[b^1]]}_{\mathcal{L}(C^{1,\alpha}(\partial\Omega_{-}))} \leq C.
	\end{align*}
	Hence, combining both estimates with the fact that 
	\begin{equation*}
		\norm{\mathsf{G}}_{C^{1,\alpha}(\SS^1)}+	\norm{\langle \mathsf{G}\rangle}_{C^{1,\alpha}(\SS^1)}\leq C \left( \norm{f}_{C^{2,\alpha}(\SS^1)}+ \norm{g}_{C^{2,\alpha}(\SS^1)}\right)\leq C\epsilon_{0}M
	\end{equation*}
	and using \eqref{small:f:g} we find that
	\begin{align}
		\norm{T[b^1,\Pi[\Psi[b^1]](\mathsf{G}-\langle \mathsf{G}\rangle)-T[b^2,\Pi[\Psi[b^2]](\mathsf{G}-\langle \mathsf{G}\rangle)}_{C^{1,\alpha}(\Omega)} &\leq C\epsilon_{0}M\bigg[\norm{\Lambda^1-\Lambda^2}_{C^{1,\alpha}(\Omega)}\nonumber \\
		&+ \norm{\vartheta^2-\vartheta^1}_{C^{\alpha}(\Omega)} + \norm{b_{1}-b_{2}}_{C^{1,\alpha}(\Omega)}\bigg]. \label{combination:7}
	\end{align}
	On the other hand using equation \eqref{combination:5} and noticing that the first term on the right hand side of \eqref{combination:5} cancels out we infer that
	\begin{align*}
	    \abs{J^{1}-J^{2}} &\leq \frac{1}{2\pi} \displaystyle\sum_{i=1}^{4}\int_{0}^{2\pi}  \abs{\mathsf{T}_{i}[\Psi[b^{1}]]j^{1}_{0}(x) -  \mathsf{T}_{i}[\Psi[b^{2}]]j^{2}_{0}(x)} \ dx \\
	  &  \hspace{3cm}+ \frac{1}{2\pi}\int_{0}^{2\pi} \abs{\left(j^{1}_0(x)-j^{2}_{0}(x)\right)f^{-}(x)} \ dx  \\
	  & \leq C\epsilon_{0}M\bigg[\norm{\Lambda^1-\Lambda^2}_{C^{1,\alpha}(\Omega)}+ \norm{\vartheta^2-\vartheta^1}_{C^{\alpha}(\Omega)}+ \norm{b_{1}-b_{2}}_{C^{1,\alpha}(\Omega)}\bigg]
	\end{align*}
	where we have argued as in the derivation of \eqref{combination:7} and using that $j^{\ell}_0= \Pi[\Psi[b^{\ell}]](\mathsf{G}-\langle \mathsf{G}\rangle)$ with $\ell=1,2.$

	Combining the later estimate with \eqref{contraction:gamma:est1}, \eqref{combination:7} and making use of the estimates \eqref{estimate:Lambda:contraction}-\eqref{estimate:vartheta:contraction} in Lemma \ref{flow:map:bound2} we conclude 
	\begin{equation}
		\norm{\Gamma(b^1)-\Gamma(b^2)}_{C^{1,\alpha}(\Omega)}\leq C\epsilon_{0}M\norm{b_{1}-b_{2}}_{C^{1,\alpha}(\Omega)} \leq \beta \norm{b_{1}-b_{2}}_{C^{1,\alpha}(\Omega)}
	\end{equation}
	where $\beta$ is strictly less than one for  $\epsilon_{0}= \frac{1}{2CM}$. Therefore, 
	$$\Gamma: (B_{M}(C^{2,\alpha}(\Omega)), \norm{\cdot}_{C^{1,\alpha}}) \to (B_{M}(C^{2,\alpha}(\Omega)), \norm{\cdot}_{C^{1,\alpha}})$$
	is a contraction mapping for $M\leq M_{0}$. 
	Invoking Banach fixed point theorem we find that $\Gamma$ admits a unique fixed point	$b\in B_{M}(C^{2,\alpha}(\Omega))$ and thus $\Gamma(b)=b$, which concludes the proof.

	\section{Proof of Theorem \ref{theorem2D}}\label{Sec:7}
	
	Take $\epsilon_{0}>0$ and $M_{0}=M_{0}(L,\alpha)$ be the constants defined in Theorem \ref{theorem2}. Let also $M\leq M_{0}$. Then, Theorem \ref{theorem2} implies that $\Gamma$ has a unique fixed point $b\in B_{M}(C^{2,\alpha}(\Omega))$. 
	We claim that if $b\in B_{M}(C^{2,\alpha}(\Omega))$ is a fixed point operator of $\Gamma$ then $B=(0,1)+b$ is the velocity field which is a solution 
	$(B,p)\in C^{{2,\alpha}}(\Omega)\times C^{{2,\alpha}}(\Omega)$ to \eqref{MHS2D:eq} satisfying the boundary conditions
	$$
	B\cdot n =1+f \ \mbox{on} \ \partial \Omega, \ B\cdot \tau=g  \ \mbox{on} \ \partial \Omega_{-}.
	$$
	On the one hand, assuming that $b\in B_{M}(C^{2,\alpha}(\Omega))$ is a fixed point operator of $\Gamma$ it is straightforward to check by construction (see that $b$ solves \eqref{div:curl:problem:full:fixed}) that
	\begin{equation*}
		\nabla \cdot B= \nabla \cdot b=0, \mbox{ in } \Omega, \quad B\cdot n= 1+b\cdot n=1+f,   \mbox{ on } \partial\Omega.
	\end{equation*}
	On the other hand, since $b$ is a fixed point of of $\Gamma$ we find that
	\begin{equation*}
		\nabla \times B= \nabla\times b=\nabla\times \Gamma(b)=j  
	\end{equation*}
	where in the last equality we have use the first equation in \eqref{div:curl:problem:full:fixed}  where $j$ solves the transport system \eqref{transport:problem:full:fixed}. Thus,
	\begin{equation*}
		0=(B\cdot\nabla)j= \nabla\times \big[j\times B], \mbox{ in } \Omega
	\end{equation*}
	and $j_0$ as in \eqref{omega:cero:fixed}. Then we can define a uni-valued function $p$ in $\Omega$ given by means of
	\begin{equation}\label{construcion:p:non:final}
		p(\bm{x})=\int_{\bm{0}}^{\bm{x}} \big[j\times B\big](\bm{y})
		\ d\bm{y}
	\end{equation}
	where the integral on the right hand side is the line integration computed along any curve connecting $\bm{0}=(0,0)$ and $\bm{x}\in\Omega$. In order to check that $p$ is a uni-valued function on $\Omega$, we only need to show that \eqref{condition:pressure:1:nonlineal} holds or equivalently that \eqref{combination:1} is satisfied. However, this follows because $J$ has been chosen as in \eqref{combination:5} (cf. \eqref{combination:3}). 
	
	Finally, since $B\in C^{2,\alpha}(\Omega)$ and $j\in C^{1,\alpha}(\Omega)$ it follows from \eqref{construcion:p:non:final} that $p\in C^{2,\alpha}(\Omega)$ and
	$$j\times B=\nabla p, \mbox{ in } \Omega$$ holds.
	Therefore, $(B,p)\in C^{{2,\alpha}}(\Omega)\times C^{{2,\alpha}}(\Omega)$ solves \eqref{MHS2D:eq}. 
	\subsection{Checking the tangential boundary value condition}\label{S:71}
	To conclude the proof of Theorem \ref{theorem2D}, it is only left to show that $B\cdot \tau=g, \mbox{ on } \partial\Omega_{-}.$  To that purpose, let us first show some consequences of the a priori estimates in Section \ref{S4}.
	
	\begin{corollary}\label{Cor:71}
		Let $j_{0}\in C^{1,\alpha}(\partial\Omega_{-})$. Then, for $i=1,\ldots,4$ the operators
		\begin{equation}
			\mathcal{T}_{i}^{NL}j_{0}(x)=\mathcal{T}_{0}^{NL}\mathsf{T}_{i}j_{0}(x): C^{1,\alpha}(\partial\Omega_{-})\to C^{2,\alpha}(\partial\Omega_{-})
		\end{equation}
		are well defined operators.  Furthermore, they can be expressed as the perturbation of convolution operators given in \eqref{TNL:j}.
	\end{corollary}
	\begin{proof}
		By means of Proposition \ref{Proposition:estimate:T1}, Proposition \ref{Proposition:estimate:T2} and Proposition \ref{prop:T3T3} we have that for $j_0\in C^{1,\alpha}(\partial\Omega_{-})$ the operators $\mathsf{T}_{i}j_{0}$ are well defined and $\mathsf{T}_{i}j_{0}\in C^{1,\alpha}(\partial\Omega_{-})$, $i=1,\ldots,4.$  On the other hand, for $j_{0}\in C^{1,\alpha}(\partial\Omega_{-})$, the operator $\mathcal{T}_{0}^{NL}j_{0}$ given as in \eqref{rig:TNL0} is well defined and $\mathcal{T}_{0}^{NL}j_{0}\in C^{2,\alpha}(\partial\Omega_{-})$. Thus combining both facts yields that 
		$$\mathcal{T}_{i}^{NL}=\mathcal{T}_{0}^{NL}\mathsf{T}_{i}: C^{1,\alpha}(\partial\Omega_{-})\to C^{2,\alpha}(\partial\Omega_{-}), \quad  i=1,\ldots,4.$$ 
		
		We now show that we can express $\mathcal{T}_{i}^{NL}j_{0}(x)$ as the convolution operators given in \eqref{TNL:j}, for $i=1,\ldots,4.$ We will just provide the proof for $i=1$, since the cases $i=2,3,4$ are very similar. Recalling that the $\mathsf{T}_{1}$ is understood as the limit operator \eqref{def:T1} 
		\begin{equation}
			\mathsf{T}_{1}j_{0}(x)=- \frac{1}{2\pi}\displaystyle \lim_{\epsilon\to 0^{+}}\int_{\mathbb{S}^{1}} \mathfrak{G}_{1,\epsilon}(x-\eta,\eta)j_{0}(\eta) \ d\eta:= \displaystyle \lim_{\epsilon\to 0^{+}}\mathsf{T}_{1,\epsilon}j_{0}(x)
		\end{equation}
		we have that
		$$\mathcal{T}_{0}^{NL}\mathsf{T}_{1}j_{0}=\mathcal{T}_{0}^{NL}\lim_{\epsilon\to 0^{+}}\mathsf{T}_{1,\epsilon}j_{0}(x).$$
		By means of Proposition \ref{Proposition:estimate:T1}, we have shown the uniform estimate
		\begin{equation}
			\norm{\mathsf{T}_{1,\epsilon}j_{0}}_{C^{1,\alpha}(\partial\Omega_{-})}\leq C, \mbox{ and }    \lim_{\epsilon\to 0^{+}}\mathsf{T}_{1,\epsilon}j_{0}(x)=\mathsf{T}_{1}j_{0}(x)
		\end{equation}
		Therefore, by the Lebesgue Dominated Convergence Theorem and the fact that 
		$\mathcal{T}_{0}^{NL}$ has an integrable kernel, we conclude that
		\begin{equation}
			\mathcal{T}_{0}^{NL}\lim_{\epsilon\to 0^{+}}\mathsf{T}_{1,\epsilon}j_{0}(x)=  \lim_{\epsilon\to 0^{+}}(\mathcal{T}_{0}^{NL}\mathsf{T}_{1,\epsilon}j_{0})(x)
		\end{equation}
		Using the formal Fourier computations in Section \ref{S:32}, we have that for $\epsilon>0$ 
		\begin{align*}
			\mathcal{T}_{0}^{NL}\mathsf{T}_{1,\epsilon}j_{0}(x)=\int_{\SS^1}\mathcal{G}_{1,\epsilon}^{NL}(x-\eta,\eta) \omega_0(\eta) d\eta
		\end{align*}
		where 
		\begin{equation}
			\mathcal{G}_{1,\epsilon}^{NL}(x,\eta)= \displaystyle\sum_{n=-\infty}^{n=\infty}e^{inx}\int_{\epsilon}^{L}e^{-\abs{n}y} \left[ \frac{e^{-in\Lambda(\eta,y)}-1}{1+\partial_{\xi} \Theta(X(\eta,y),y)} \right]  dy.
		\end{equation}
		Computing the summation in $n$, in a similar fashion in \eqref{G11+:formula} we find that
		$$\abs{\mathcal{G}_{1,\epsilon}^{NL}(x,\eta)}\leq C \displaystyle\log(\abs{x}+\epsilon) \leq C \displaystyle\log(\abs{x})$$
		and hence by Lebesgue Dominated Convergence Theorem we conclude that
		\begin{equation}
			\lim_{\epsilon\to 0^{+}}\int_{\SS^1}\mathcal{G}_{1,\epsilon}^{NL}(x-\eta,\eta) \omega_0(\eta) d\eta= \int_{\SS^1} \lim_{\epsilon\to 0^{+}}\mathcal{G}_{1,\epsilon}^{NL}(x-\eta,\eta) \omega_0(\eta) d\eta.
		\end{equation}
		Moreover $ \lim_{\epsilon\to 0^{+}}  \mathcal{G}_{1,\epsilon}^{NL}(x,\eta)\to \mathcal{G}_{1}^{NL}(x,\eta),  \forall x\neq 0$. Therefore, combining the previous computations we obtain that
		\begin{equation}
			\mathcal{T}_{1}^{NL}\omega_0(x)= \mathcal{T}_{0}^{NL}\mathsf{T}_{1}\omega_0(x)= \int_{\SS^1}\mathcal{G}_{1}^{NL}(x-\eta,\eta)j_{0}(\eta
			) \ d\eta
		\end{equation}
		which shows the desired asserted expression as in \eqref{TNL:j}.
	\end{proof}
	The following lemma gives the tangential velocity in terms of the Biot-Savart system \eqref{DCP:problem}. 
	\begin{lemma}\label{lema:final:7}
		Let  $b\in C^{2,\alpha}(\Omega), X^{-1}\in C^{2,\alpha}(\Omega)$,  $f\in C^{2,\alpha}(\partial\Omega)$ and $j_{0}\in C^{1,\alpha}(\partial\Omega_{-})$ satisfy the following system
		\begin{equation}\label{div:curl:problem:full:fixed:2}
			\left\lbrace
			\begin{array}{lll}
				\nabla\times b= j_{0}(X^{-1}(x,y)), \ \mbox{in} \ \Omega \\
				\mbox{div } b=0, \ \mbox{in} \ \Omega \\ 
				b\cdot n=f, \ \mbox{on} \ \partial \Omega \\
				\int_{0}^{L} b_{1}(0,y) \ dy=J.
			\end{array} \right.
		\end{equation} 
		with $J$ as in \eqref{combination:5} (cf. \eqref{combination:3}). 
		
		Then, we have that 
		$$(b\cdot \tau)(x,0)=-\frac{J}{L}+\mathcal{Z}(x) +\frac{1}{2\pi}\mathcal{T}_{0}^{NL}j_{0}(x)+\frac{1}{2\pi}\displaystyle\sum_{i=1}^{4}\mathcal{T}_{i}^{NL}j_{0}(x)$$
		with
		$$\mathcal{Z}(x)=\frac{1}{2\pi} \displaystyle\sum_{n=-\infty}^{n=\infty}\left(\widehat{h^{+}}(n) \frac{\abs{n}}{\sinh (\abs{n}L)}-
	\widehat{h^{-}}(n)\frac{\abs{n}}{\tanh (\abs{n}L)}\right) e^{inx}. $$
	\end{lemma}
	
	\begin{proof}
		Arguing as in Subsection \ref{S:31}, we have that since $b$ solves \eqref{div:curl:problem:full:fixed:2} there exists a stream function $\psi$ such that
		\begin{equation}
			\left\lbrace
			\begin{array}{lll}
				\Delta \psi= j_{0}(X^{-1}(x,y)),\ \mbox{in } \Omega \\ 
				\psi(x,L)= -J+\displaystyle\int_{0}^{x} (f(\xi,L)-A) \ d\xi  ,  \ x\in \mathbb{R} \\
				\psi(x,0)= \displaystyle\int_{0}^{x} (f(\xi,0)-A) \ d\xi  ,  \ x\in \mathbb{R}
			\end{array}\right.
		\end{equation}
		for $A=\int_{\partial\Omega_{+}} f \ dS=\int_{\partial\Omega_{-}} f\ dS$. Moreover, using the fundamental solution $\Phi(x,y)$ solving the problem
		\begin{equation}
			\left\lbrace
			\begin{array}{lll}
				\Delta \Phi (x,y)= \delta(x)\delta(y-y_{0}), \ \mbox{in } \Omega, \\ 
				\Phi=0, \ \mbox{on } \partial \Omega. \\
			\end{array}\right.
		\end{equation}
		we can readily check (cf. Subsection \ref{S:31}) that the normal derivative at $y=0$ is given by
		\[\partial_{y} \Phi(x,0,y_{0})= -\frac{1}{2\pi} \displaystyle\sum_{n=-\infty}^{n=\infty} \frac{\sinh(n(L-y_{0}))}{\sinh(nL)} e^{inx}. \]
		Computing an homogeneous solution and imposing the boundary value conditions using Fourier techniques as in \eqref{psi:fourier:exp} in Section \ref{Sec:2}
		we conclude  that 
		\begin{equation}
			\partial_{y} \psi (x,0)= -\frac{J}{L}+
			\mathcal{Z}(x)-\frac{1}{2\pi}\int_{0}^{L} dy_{0} \int_{\mathbb{S}^{1}}\displaystyle\sum_{n=-\infty}^{n=\infty} \frac{\sinh(n(L-y_{0}))}{\sinh(nL)} e^{in(x-\xi)}j_{0}(X^{-1}(\xi,y_{0})) \ d \xi \label{psi:invers:full:2}
		\end{equation}
		where $\mathcal{Z}(x)$ as in \eqref{B}. Since $f\in C^{2,\alpha}(\partial\Omega)$, it is straightforward to check that boundary condition term $\mathcal{Z}(x)$ is well defined.  
		On the other hand, invoking Corollary \ref{Cor:71} we have that for $j_0\in C^{1,\alpha}(\partial\Omega_{-})$ the operators 
		\begin{equation*}
			\mathcal{T}_{i}^{NL}j_{0}(x)=\mathcal{T}_{0}^{NL}\mathsf{T}_{i}j_{0}(x): C^{1,\alpha}(\partial\Omega_{-})\to C^{2,\alpha}(\partial\Omega_{-})
		\end{equation*}
		are well defined operators and can be expressed as the convolution operators given in \eqref{TNL:j}. Hence, by recalling the definition  \eqref{decomposition:operator:T} we infer that
		\begin{equation*}
			\mathcal{T}^{NL}j_{0}(x)=\mathcal{T}^{NL}_{0}j_{0}(x)+ \displaystyle\sum_{i=1}^{4}\mathcal{T}^{NL}_{i}j_{0}(x)
		\end{equation*}
		it admit the representation formula
		\begin{align*}
			\mathcal{T}^{NL}j_{0}(x)=-\frac{1}{2\pi}\int_{\mathbb{S}^1}\mathcal{G}^{NL}(x-\eta,\eta) j_{0}(\eta)  \ d\eta,
		\end{align*}
		where
		\begin{align*}
			\mathcal{G}^{NL}(x,\eta) = \ \displaystyle\sum_{n=-\infty}^{n=\infty}  a_{n}(\eta)e^{inx}, a_{n}(\eta)=  \int_{0}^{L} \frac{\sinh(n(L-y))}{\sinh(nL)} \frac{e^{-in \Lambda (\eta,y)}}{(1+\partial_{\xi} \Theta(X(\eta,y),y))} dy.
		\end{align*}
		where $\Lambda(\eta,y)=X(\eta,y)-\eta \in C^{2,\alpha}(\Omega)$ and $X^{-1}(\xi,y)=\xi+\Theta(\xi,y)=\eta$. Unraveling notation (cf. computations \eqref{psi:invers:full}-\eqref{full:integral:eq}), it is easy to check that 
		\begin{equation}\label{def:TNL:sec7}
			\mathcal{T}^{NL}j_{0}(x)=-\frac{1}{2\pi}\int_{0}^{L} dy_{0} \int_{\mathbb{S}^{1}}\displaystyle\sum_{n=-\infty}^{n=\infty} \frac{\sinh(n(L-y_{0}))}{\sinh(nL)} e^{in(x-\xi)}j_{0}(X^{-1}(\xi,y_{0})) \ d \xi
		\end{equation}
		is a well defined operator. Therefore, combining \eqref{psi:invers:full:2}-\eqref{def:TNL:sec7} and noticing that $-\dfrac{\partial \psi }{\partial y}=b\cdot \tau$ on $\partial\Omega_{-}$ provides our claim.
	\end{proof}
	
	\begin{corollary}\label{Cor:72}
		We have that $B\cdot\tau =g \ \mbox{on} \ \partial \Omega_{-}.$
	\end{corollary}
	\begin{proof}
		Applying Lemma \ref{lema:final:7} and noticing that by construction $j_0(x)$ solves 
		\begin{equation}\label{full:integral:eq:2}
			\mathcal{T}^{NL}j_{0}(x)=-\frac{1}{2\pi}\int_{\mathbb{S}^1}\mathcal{G}^{NL}(x-\eta,\eta) j_{0}(\eta)  \ d\eta=\widetilde{g}+\frac{J}{L}
		\end{equation}
		where $\widetilde{g}(x)=-g(x)-\mathcal{Z}(x)$ we conclude that  $B\cdot\tau=b\cdot \tau=g  \ \mbox{on} \ \partial \Omega_{-}$.
	\end{proof}

\end{document}